\documentclass[12pt,final]{amsart}
\usepackage[utf8]{inputenc}
\usepackage[T1]{fontenc}
\usepackage[letterpaper,centering]{geometry}
\usepackage{lmodern}
\usepackage{mathrsfs}
\usepackage{amsmath,amssymb,amsfonts}
\usepackage[unicode,pdfborder={0 0 0},final]{hyperref}
\usepackage{enumerate}
\usepackage{multicol}

\usepackage[all]{xy}
{\setbox0\hbox{$ $}}\fontdimen16\textfont2=\fontdimen17\textfont2
\entrymodifiers={+!!<0pt,\the\fontdimen22\textfont2>}
\SelectTips{cm}{12}

\theoremstyle{plain}
\newtheorem{thm}{Theorem}[section]
\newtheorem{thmintro}{Theorem}

\newtheorem{lem}[thm]{Lemma}
\newtheorem{question}[thm]{Question}
\newtheorem{cor}[thm]{Corollary}
\newtheorem{prop}[thm]{Proposition}
\theoremstyle{definition}
\newtheorem{defn}[thm]{Definition}
\newtheorem{rmk}[thm]{Remark}
\newtheorem{rmks}[thm]{Remarks}
\newtheorem{Step}{Step}
\newtheorem{example}[thm]{Example}

\numberwithin{equation}{section}

\newcommand{\alg}{{\mathrm{alg}}}
\newcommand{\cl}{{\mathrm{cl}}}
\newcommand{\Gm}{\mathbf{G}_\mathrm{m}}
\newcommand{\isoto}{\myxrightarrow{\,\sim\,}}
\makeatletter
\def\myrightarrow{{\setbox\z@\hbox{$\rightarrow$}\dimen0\ht\z@\multiply\dimen0 6\divide\dimen0 10\ht\z@\dimen0\box\z@}}
\def\myrightarrowfill@{\arrowfill@\relbar\relbar\myrightarrow}
\newcommand{\myxrightarrow}[2][]{\ext@arrow 0359\myrightarrowfill@{#1}{#2}}
\def\myleftarrow{{\setbox\z@\hbox{$\leftarrow$}\dimen0\ht\z@\multiply\dimen0 6\divide\dimen0 10\ht\z@\dimen0\box\z@}}
\def\myleftarrowfill@{\arrowfill@\myleftarrow\relbar\relbar}
\newcommand{\myxleftarrow}[2][]{\ext@arrow 3095\myleftarrowfill@{#1}{#2}}
\makeatother
\newcommand{\kP}{{\mathfrak P}}
\newcommand{\kQ}{{\mathfrak Q}}
\newcommand{\kR}{{\mathfrak R}}
\newcommand{\kS}{{\mathfrak S}}

\newcommand{\kU}{{\mathfrak U}}

\newcommand{\kX}{{\mathfrak X}}
\newcommand{\kY}{{\mathfrak Y}}
\newcommand{\kZ}{{\mathfrak Z}}
\newcommand{\sC}{{\mathscr C}}
\newcommand{\sH}{{\mathscr H}}
\newcommand{\sL}{{\mathscr L}}
\newcommand{\sE}{{\mathscr E}}
\newcommand{\sF}{{\mathscr F}}

\newcommand{\sM}{{\mathscr M}}
\newcommand{\sO}{{\mathscr O}}
\newcommand{\sP}{{\mathscr P}}
\newcommand{\sS}{{\mathscr S}}
\newcommand{\sU}{{\mathscr U}}

\newcommand{\A}{{\mathbf A}}

\renewcommand{\C}{{\mathbf C}}

\newcommand{\N}{{\mathbf N}}
\renewcommand{\P}{{\mathbf P}}

\newcommand{\R}{{\mathbf R}}

\newcommand{\Z}{{\mathbf Z}}
\newcommand{\CH}{\mathrm{CH}}
\newcommand{\Gal}{\mathrm{Gal}}

\newcommand{\Id}{\mathrm{Id}}
\newcommand{\BM}{\mathrm{BM}}
\newcommand{\Pic}{\mathrm{Pic}}
\newcommand{\an}{\mathrm{an}}
\newcommand{\GL}{\mathrm{GL}}
\newcommand{\Spec}{\mathrm{Spec}}
\newcommand{\Sper}{\mathrm{Sper}}
\let\phivar\phi
\renewcommand{\phi}{\varphi}
\renewcommand{\emptyset}{\varnothing}

\newcommand{\surj}{\twoheadrightarrow}
\newcommand{\inj}{\hookrightarrow}
\newcommand{\et}{\text{ét}}
\newcommand{\bS}{{\mathbf S}}
\newcommand{\Frac}{\mathrm{Frac}}
\newcommand{\ci}{\sC^\infty}
\newcommand{\Bbsecondezero}{B_b''^{\mkern1.5mu0}}
\date{July 24th, 2019; revised on December 14th, 2020}
\title{The tight approximation property}

 \author{Olivier Benoist}
\address{D\'epartement de math\'ematiques et applications, \'Ecole normale sup\'erieure, 45~rue d'Ulm, 75230 Paris Cedex 05, France}
 \email{olivier.benoist@ens.fr}

 \author{Olivier Wittenberg}
\address{Institut Galil\'ee, Universit\'e Sorbonne Paris Nord, 99~avenue Jean-Baptiste Cl\'ement, 93430 Villetaneuse, France}
\email{wittenberg@math.univ-paris13.fr}

\begin{document}
\begin{abstract}
This article introduces and studies the tight approximation property, a property of algebraic
varieties defined over the function field of a complex or real curve
that refines the weak
approximation property (and the known cohomological obstructions to it)
by incorporating an approximation condition
in the Euclidean topology.
We prove that the tight approximation property is a stable birational invariant,
is compatible with fibrations,
and satisfies descent under torsors of linear algebraic groups.
Its validity for a number of rationally connected varieties follows.
Some concrete consequences are:
smooth loops in the real locus of a
smooth compactification of a real linear algebraic group, or in a smooth cubic hypersurface
of dimension~$\geq 2$, can be approximated
by rational algebraic curves;
homogeneous spaces of linear algebraic groups over the function field of a real curve
satisfy weak approximation.
\end{abstract}

\maketitle

\section{Introduction}

One of the basic results of the theory of rationally connected varieties
states that on any smooth and proper variety over~$\C$
which is rationally connected---that is, on which two general points can be joined
by a rational curve---any finite set of points
can in fact be joined by a single rational curve
(see \cite[Chapter~IV, Theorem~3.9]{Kollarbook}).

Questions on the existence of algebraic curves
on algebraic varieties have been at the core of further developments of the theory.

Graber, Harris and Starr~\cite{ghs}
have shown that a dominant map $f:\kX\to B$ between smooth proper varieties over~$\C$, where $B$ is a connected curve, admits a section if its geometric generic fibre is rationally connected.
Equivalently, the generic fibre~$X$ of~$f$ admits a rational point over the function field $F=\C(B)$.
Under the same hypotheses, Hassett and Tschinkel~\cite{HT} have proved that~$f$ admits sections with any given
prescribed jet, of any finite order, along any finite subset of $B(\C)$ over which~$f$ is smooth
(see also \cite{Hassett}).
The case of a product fibration
(\emph{i.e.}\ $\kX$ is the product of~$B$ with a variety over~$\C$, and~$f$ is the projection)
and jet of order~$0$
recovers the result mentioned at the beginning of the introduction.
A more general property, which takes into account the singular fibres of~$f$ as well, is expected to hold: the
\emph{weak approximation property} for~$X$, that is,
the density of
$X(F)$ diagonally embedded in $\prod_{b \in B(\C)} X(F_b)$,
where~$F_b$ denotes the completion of~$F$ at $b \in B(\C)$, with respect to
the product of the $b$\nobreakdash-adic topologies.
Weak approximation over $F=\C(B)$ is easily seen to be a stable birational invariant, to be compatible with
fibrations
(in the sense that whenever $g:X\to X'$ is a dominant morphism between smooth varieties
over~$F$,
weak approximation holds for~$X$ as soon as it holds for~$X'$ and for the fibres of~$g$ above
the $F$\nobreakdash-points of a dense open subset of~$X'$; see \cite{CTG}),
and it has been shown to hold for smooth proper models of homogeneous spaces of linear groups
(\emph{loc.\ cit.})\ and for smooth cubic hypersurfaces of dimension~$\geq 2$ (see~\cite{tiancubic}).

A related line of investigation concerns the integral Hodge conjecture for $1$\nobreakdash-cycles:
for any smooth and proper rationally connected variety~$X$ over~$\C$,
the integral homology group $H_2(X(\C),\Z)$ is expected to be spanned by classes of rational curves.
This is known to hold when $\dim(X)\leq 3$
(see \cite{voisinthreefolds}, \cite{voisinsomeaspects}, \cite{TianZong}).

Analogous questions have been formulated for varieties defined over the field~$\R$ of
real numbers.  Given a finite set of closed
points on a smooth, proper, rationally connected real algebraic variety
(by which we mean a variety defined over~$\R$ whose underlying complex algebraic
variety is rationally connected), there is an obvious obstruction to the existence of a single
rational curve passing smoothly through them: the real points in the set
must lie in the same connected component of the real
locus of the variety.  Under this hypothesis, and assuming in addition that the real locus
in question is non-empty,
Kollár  \cite{Kollocal, Kolfertile} has shown that such a rational curve exists.
Similarly, given a dominant map $f:\kX\to B$ between smooth proper varieties over~$\R$ whose target $B$ is a connected curve and whose geometric generic fibre is rationally connected, there are topological obstructions
to the existence of a section of~$f$, resp.\ to weak approximation for the generic fibre~$X$ over
the function field $F=\R(B)$. Namely,
the map $f(\R):\kX(\R)\to B(\R)$
must admit a $\ci$ section, resp.\ a~$\ci$ section with a prescribed jet.
These obstructions were studied, at first under a slightly different guise based on a reciprocity law
and on unramified cohomology in analogy with the Brauer--Manin obstruction in number theory,
by Colliot-Thélène~\cite{CTgroupes}, Scheiderer~\cite{Scheiderer}, Ducros~\cite{Ducros},
thus leading to the natural hope that a section of~$f$ exists, with a prescribed jet of finite
order along a finite set of closed points of~$B$, as soon as a~$\ci$ section with this prescribed jet exists.
This conjecture, which in the case of a product fibration with non-empty real locus and jet of order~$0$
amounts to the theorem
of Kollár recalled above,
is still very much open.
It has been verified when~$X$ is a smooth proper model of a
homogeneous space of a linear algebraic group with connected geometric stabilizer
(in \cite{CTgroupes, ducrosclassiques, Scheiderer}) and when~$X$ is a conic bundle surface over~$\P^1_F$ (in~\cite{Ducros}) or more generally a variety fibred
into Severi--Brauer varieties over~$\P^1_F$ (in~\cite{DucrosCRAS}).
More recently, Pál and Szabó~\cite{PalSzabo} proved its compatibility with fibrations
whose base is rational over~$F$, thus extending Ducros' results about conic bundle surfaces and fibrations into Severi--Brauer varieties.

In a different but related direction,
the homology group $H_1(X(\R),\Z/2\Z)$ is expected to be spanned by classes of algebraic curves,
perhaps of rational curves,
for any smooth and proper rationally connected variety~$X$ over~$\R$
(see \cite{BWI, BWII}).  It makes sense to ask the following much stronger  $\ci$ approximation question
(in the sense of the
$\ci$ compact-open topology \cite[Chapter~2, \textsection1]{Hirsch}):

\begin{question}
\label{q:ciapprox}
Let~$B$ be a smooth projective connected curve over $\R$ and~$X$ be a smooth, proper rationally connected variety over $\R$.  Let $\varepsilon: B(\R)\to X(\R)$ be a~$\ci$ map.
Are there morphisms of algebraic varieties $B \to X$ which induce~$\ci$ maps $B(\R)\to X(\R)$ arbitrarily close
to~$\varepsilon$ in the $\ci$ compact-open topology?
\end{question}

Building on the Stone--Weierstrass theorem, Bochnak and Kucharz  \cite[\S 2]{BKrat} provide a positive answer for $X=\P^n_\R$ and more generally for all varieties that are rational over~$\R$.
To this day, the answer to Question~\ref{q:ciapprox} is not known for a single~$X$
that possesses a real point and that is not rational over~$\R$,
despite renewed interest in the problem (see \cite{kollarmangolte}).
One of our goals in the present article is to remedy this situation, by establishing
the following theorem:

\begin{thmintro}
\label{intro:thmA}
Question~\ref{q:ciapprox} admits a positive answer if~$X$ is birationally equivalent
to a variety belonging to any of the following families:
\begin{enumerate}
\item smooth cubic hypersurfaces in $\P^n_\R$ for $n \geq 3$;
\item smooth intersections of two quadrics in $\P^n_\R$ for $n \geq 4$;
\item homogeneous spaces of linear algebraic groups over~$\R$.
\end{enumerate}
\end{thmintro}

Even the particular case of cubic surfaces is new, and in fact disproves
a conjecture formulated by Bochnak and Kucharz in \cite[p.~12]{BKalgebraicapprox}:
according to Theorem~\ref{intro:thmA}~(1), any real smooth cubic surface whose real locus has two connected components
is a ``Weierstrass'' surface in the terminology of \emph{op.\ cit.}, even though it is not rational.

Theorem~\ref{intro:thmA} is a by-product of more general results
that we now set out to explain.

To approach the many questions that we have mentioned so far,
we propose and study, in this article,
the \emph{tight approximation} property (from the French \emph{approximation fine}), which simultaneously
 generalizes weak approximation and~$\ci$ approximation, both in the complex context and in the real context.
We formulate it in a relative setting, for sections
of a morphism over a smooth projective connected curve,
though ultimately it is a property of the generic fiber.

\begin{defn}[see Definitions \ref{deftight} and~\ref{deftight2}]
\label{intro:deftight}
A smooth variety~$X$ over the function field of a complex curve~$B$
satisfies the \textit{tight approximation property}
if for any proper flat morphism $f:\kX \to B$ 
such that~$\kX$ is regular and whose
generic fiber is birationally equivalent to~$X$,
any holomorphic section $u: \Omega\to\kX(\C)$ of $f(\C):\kX(\C)\to B(\C)$
over an open subset $\Omega \subset B(\C)$
can be approximated arbitrarily well, in
the $\ci$ compact-open topology,
by maps $\Omega \to \kX(\C)$ induced by algebraic
sections $B\to\kX$ of $f$ having the same jets as $u$
along any prescribed finite subset of $\Omega$,
at any prescribed
finite order. (By an ``algebraic section of~$f$'', we mean a morphism of algebraic varieties $B \to \kX$
that is a section of~$f$.)

The definition of \emph{tight approximation} for a smooth variety~$X$ over the function field of a real curve~$B$
is the same, except that we require the open subset $\Omega \subset B(\C)$ to contain $B(\R)$ and to be stable under $\Gal(\C/\R)$
and the holomorphic map~$u$ to be $\Gal(\C/\R)$\nobreakdash-equivariant.
\end{defn}

Let us immediately point out three essential features of this definition.
First, by prescribing jets of finite order at finitely many points,
the tight approximation property incorporates in its very definition a condition of weak approximation type.
This is necessary,
even if one is only interested in~$\ci$ approximation,
to obtain a birationally invariant property,
as was already noted by Bochnak and Kucharz \cite{BKrat}.
Secondly,
the tight approximation property also incorporates a~$\ci$ approximation condition.
This is necessary,
even if one is only interested in weak approximation
or in the homology of the real locus,
to obtain a property that behaves well in fibrations (into
conics over an arbitrary rationally connected base, for instance),
as can be seen already in the proof of \cite[Theorem~6.1]{BWII}.
Finally,
unlike Question~\ref{q:ciapprox},
the
tight approximation property is formulated for one-dimensional families of
real varieties.
This also turns out to be crucial
for the property to behave well in fibrations,
even if one is only interested in~$\ci$ approximation on
individual real varieties.

We henceforth let~$B$ denote a smooth projective connected curve over~$\R$ and
do not assume that~$B$
is geometrically connected; in a tautological way, this makes the complex
context a particular case of the real one.
We let $F=\R(B)$
and now discuss in detail what we prove concerning the tight approximation property.

As a consequence of Whitney's approximation theorem of $\ci$ maps by real analytic maps,
the tight approximation property for~$X$ implies that any $\ci$ section
of $f(\R):\kX(\R)\to B(\R)$ can be approximated arbitrarily well
by the real locus of an algebraic section $B \to \kX$ with a prescribed jet, of any finite order,
at finitely many points of~$B$
(Proposition~\ref{implitight}).
Thus, tight approximation implies, in the case of a product fibration, a positive answer
to Question~\ref{q:ciapprox}, and at the same time, for arbitrary fibrations, the best possible weak approximation statement
(which coincides with weak approximation itself when $B(\R)=\emptyset$, \emph{e.g.}\ when $B$ is a complex curve).

We prove that the property appearing in Definition~\ref{intro:deftight} does not depend on the choice of the
model $f:\kX\to B$ (Theorem~\ref{birinv}).
We also prove,
using a version of Runge's approximation theorem for compact Riemann
surfaces,
that~$\P^d_F$ satisfies tight approximation, and deduce that the tight approximation property is a stable birational
invariant (Theorem~\ref{projective}, Proposition~\ref{stableinv}).

The first main theorem of the article asserts the compatibility of the tight approximation property
with fibrations:

\begin{thmintro}[Theorem~\ref{up}]
\label{intro:thmB}
Let $g: X\to X'$ be a dominant morphism between smooth varieties over~$F$.
If~$X'$ and the fibers of~$g$ above the $F$\nobreakdash-points of a dense open subset of~$X'$ satisfy the
tight approximation property, then so does~$X$.
\end{thmintro}

It should be noted that Theorem~\ref{intro:thmB}---unlike the main theorem of~\cite{PalSzabo},
which only deals with weak approximation---does
not assume the base to be rational over~$F$, or even geometrically rational.
As was remarked above, for such a
general fibration statement to be accessible,
it is crucial to incorporate a~$\ci$ approximation condition into the property under consideration.

The proof of Theorem~\ref{intro:thmB} relies on
the weak toroidalization theorem
of Abramovich, Denef and Karu~\cite{ADK} and on some toroidal geometry.
Letting~$\kX$ and~$\kX'$ denote proper regular models of~$X$ and~$X'$ over~$B$
such that~$g$ extends to a morphism $g:\kX \to \kX'$,
the key tool here is an algebro-geometric statement
according to which for $x\in \kX$, a general germ of curve on~$\kX$ through~$x$
whose image on~$\kX'$ is smooth
can be made to miss the non-smooth
locus of $g:\kX \to \kX'$ by replacing~$\kX$ and~$\kX'$ with suitable modifications and the germ
with its strict transform. (See Proposition~\ref{toroidaltheorem} for a precise statement.)
This assertion, valid for a morphism $g:\kX \to \kX'$ of smooth varieties over an arbitrary field of
characteristic~$0$, can be viewed as an extension of
the Néron smoothening process to higher-dimensional bases
(see \cite[3.1/3]{neronmodels}).

As a corollary to Theorem~\ref{intro:thmB},
we deduce
Theorem~\ref{intro:thmA}~(1) and (2).
Indeed, smooth cubic hypersurfaces
of dimension~$\geq 2$ over~$\R$, as well as smooth complete intersections of two quadrics of dimension $\geq 2$ over $\R$ with a real point,
are birationally equivalent to quadric bundles over projective spaces,
and positive-dimensional quadrics over~$F$ satisfy the tight approximation
property (see
Examples~\ref{quadrique}, \ref{cubique} and~\ref{ic2c}).

The second main theorem asserts that the method of descent under arbitrary linear groups,
pioneered over number fields by Colliot-Thélène and Sansuc and extended to non-abelian
groups by Harari and Skorobogatov, can be successfully applied
for verifying the tight approximation property:

\begin{thmintro}[Theorem~\ref{descent}]
\label{intro:thmC}
Let $X$ be a smooth variety over~$F$
and~$S$ be a linear algebraic group over~$F$.
Let $Q\to X$ be a left $S$\nobreakdash-torsor.
If every twist of~$Q$ by a right $S$\nobreakdash-torsor over~$F$ satisfies the tight approximation property,
then so does~$X$.
\end{thmintro}

We note that Theorem~\ref{intro:thmC} does not assume~$X$ to be proper or~$S$ to be connected.
For the weak approximation property,
a similar descent theorem was established
by Ducros~\cite[Théorème~5.22]{Ducros} under the assumption that~$S$ is connected.

Using standard reductions in the theory of homogeneous spaces,
together with Scheiderer's Hasse principle for homogeneous spaces of linear algebraic groups over~$F$
(see \cite{Scheiderer}),
we deduce from the combination of Theorem~\ref{intro:thmB}
and Theorem~\ref{intro:thmC}:

\begin{thmintro}[Theorem~\ref{homogeneous}]
\label{intro:thmD}
Homogeneous spaces of connected linear algebraic groups over $F$ satisfy the tight approximation property.
\end{thmintro}

In more detail, the proof of Theorem~\ref{intro:thmD} starts with the remark that the
tight approximation property holds for quasi-trivial tori over~$F$ since it holds for
projective spaces (Theorem~\ref{projective});
we then
successively deduce
its validity for arbitrary tori over~$F$
(by descent, using Theorem~\ref{intro:thmC}, in Proposition~\ref{torus}),
then for connected linear algebraic groups over~$F$
(exploiting the structure of a fibration into tori over an affine space over~$F$,
using Theorem~\ref{intro:thmB}, in Proposition~\ref{group}),
then for homogeneous spaces of connected linear algebraic groups over~$F$ that possess an $F$\nobreakdash-point
(using Theorem~\ref{intro:thmC} again),
and finally for arbitrary homogeneous spaces of connected linear algebraic groups over~$F$
(using Scheiderer's Hasse principle).

Theorem~\ref{intro:thmA}~(3) immediately results from Theorem~\ref{intro:thmD}, in the particular case of a product fibration.

On the other hand,
Scheiderer's work allows us to conclude from Theorem~\ref{intro:thmD}
that homogeneous spaces of connected linear algebraic groups over $F$
satisfy the weak approximation property (Theorem~\ref{weakhomoreal}).
This solves a conjecture put forward by Colliot-Thélène~\cite{CTgroupes}.
Scheiderer~\cite{Scheiderer} had proved it when the stabilizer of a
geometric point is connected and had left the general case open.

In fact, Scheiderer allows in \emph{op.\ cit.}\ the field~$\R$ to be replaced with an arbitrary real closed field.  Following his lead, we adapt the proof of Theorem~\ref{intro:thmD}
and establish
Colliot-Thélène's conjecture for function fields of curves over real closed
fields:

\begin{thmintro}[Theorem~\ref{weakhomo}]
\label{intro:thmE}
Homogeneous spaces of connected linear algebraic groups over the function
field of a curve over a real closed field satisfy the weak approximation
property.
\end{thmintro}

If~$X$ is a smooth compactification of a homogeneous
space of a connected linear algebraic group over~$\R$,
it is a noteworthy
 consequence of Theorem~\ref{intro:thmD} (in fact, of
Theorem~\ref{intro:thmA}~(3))
that
the homology group $H_1(X(\R),\Z/2\Z)$ is spanned by classes of rational curves.
By contrast, it is still an open question whether $H_2(X(\C),\Z)$ is spanned by classes of algebraic
curves, \emph{i.e.}\ whether~$X_\C$ satisfies the integral Hodge conjecture for $1$\nobreakdash-cycles.

In Proposition~\ref{dP}, we verify that $H_1(X(\R),\Z/2\Z)$ is again spanned by classes
of rational curves if $X$ is a rationally connected surface, even though
Question~\ref{q:ciapprox} remains unanswered in this case.

In view of the stable birational invariance of the tight approximation property,
of its compatibility with fibrations, of its compatibility with descent under torsors of linear algebraic groups, and of its
validity for homogeneous spaces of linear algebraic groups,
it makes sense to raise the question of the validity of the tight approximation property
for arbitrary rationally connected varieties
(Question~\ref{qmain}).

For some varieties that are not known to satisfy tight approximation,
the ideas underlying
the proof of Theorem~\ref{intro:thmB} turn out to still be useful for proving the weaker property
that $H_1(X(\R),\Z/2\Z)$ is spanned by classes of algebraic curves, when~$X$ is smooth, proper and defined over~$\R$.
We explore this direction in the last section of the article, where
we formulate a version of this property for varieties defined over~$F$ rather than over~$\R$
(see Definition~\ref{defn:falg}, Proposition~\ref{birinvBH}),
prove the analogue of Theorem~\ref{intro:thmB} for this property (Theorem~\ref{fibBH})
and apply it to smooth cubic hypersurfaces.
Using the results of \cite{BWII} about
cubic surfaces over~$F$, we deduce:

\begin{thmintro}[Corollary~\ref{corBHcubic}]
\label{intro:thmF}
For any smooth proper variety~$X$ over~$\R$
which is, birationally, an iterated fibration
into smooth cubic hypersurfaces of dimension $\geq 2$,
the group $H_1(X(\R),\Z/2\Z)$ is spanned by classes of algebraic curves.
\end{thmintro}

We stress that it is a key point in our proof of Theorem~\ref{intro:thmF} that we work, all the way through,
with one-dimensional families of real cubic hypersurfaces
rather than with real cubic hypersurfaces,
even though it is in the latter that we are ultimately interested.
Indeed, applying our fibration theorems (Theorem~\ref{up} or Theorem~\ref{fibBH})
leads one to consider fibers over
$F$\nobreakdash-rational points of the base that are not $\R$\nobreakdash-points, even when the base
and the total space are defined over~$\R$.  For the same reason,
as we do not know
the tight approximation property for cubic surfaces
over~$F$,
an answer to Question~\ref{q:ciapprox}
for the varieties appearing in Theorem~\ref{intro:thmF} is out of reach.

The article ends with an appendix gathering $\Gal(\C/\R)$\nobreakdash-equivariant versions
of a number of tools of complex analytic geometry that we use throughout, as well as a
description of the relationship, for $\Gal(\C/\R)$\nobreakdash-equivariant Riemann surfaces with compact real locus,
between the real points and the orderings of the field of equivariant meromorphic functions
(Proposition~\ref{spectrereel}).

\bigskip
\emph{Acknowledgements.}
Scheiderer's paper \cite{Scheiderer} plays a fundamental role in our results on homogeneous spaces of linear
algebraic groups.
Its influence is gratefully acknowledged.
We thank Roland Huber for sending us a copy of~\cite{huberthesis},
Dmitri Pavlov for having made \cite{Guralnick} available to us, Jean-Louis Colliot-Thélène for a useful comment on a previous version of this article and the referees for their careful work.

\bigskip
\subsection{Notation and conventions}
\label{conventions}
Let $k$ be a field.
We denote by $\overline{k}$ an algebraic closure of~$k$. A \emph{variety} over $k$ is a separated scheme of finite type over $k$.  If $f: X\to Y$ is
a morphism of varieties over $k$, and $k\subset k'$ is a field extension, we define $f(k'): X(k')\to Y(k')$
to be the induced map on $k'$-points. We say that $f$ is a \textit{modification} if it is proper and birational,
a \emph{regular modification} if in addition~$X$ is regular.
A \textit{rational curve} on a variety $X$ over $k$ is a non-constant morphism $f: C\to X$ where $C$ is a smooth projective
geometrically connected curve of genus~$0$ over $k$.
We say that a variety $X$ over $k$ is \textit{rationally connected} if $X_{\overline{k}}$ is rationally connected in the sense of Koll\'ar--Miyaoka--Mori
 \cite[IV Definition 3.2]{Kollarbook}.
%La def de [KMM, Rationally connected varieties] J Alg Geom 1 est erronée sur un corps dénombrable.
If~$G$ denotes an algebraic group over~$k$, a \emph{homogeneous space} of~$G$ is a variety~$X$ over~$k$
endowed with an action of~$G$ such that the induced action of $G(\overline{k})$ on~$X(\overline{k})$
is transitive; in particular~$X$ is non-empty, while~$X(k)$ may be empty.

Let $\R$ and $\C$ be the fields of real and complex numbers, and let $G$ be the Galois group $\Gal(\C/\R)=\Z/2\Z$.
Let $\sigma\in G$ be the complex conjugation. In all this text (with the exception of~\textsection\ref{subsec:wahom}, where the ground
field is an arbitrary real closed field), we fix a smooth projective connected curve $B$ over $\R$ with function
field $F$ and generic point $\eta$.
Let $F_b=\Frac(\widehat{\sO_{B,b}})$ be the completion of $F$ associated with a closed point $b\in B$.
If $X$ is a variety over $F$, a \textit{model} of $X$ is a flat morphism $f: \kX\to B$ of varieties over $\R$ with
a dense open embedding $X\subset\kX_\eta$ of varieties over $F$.

All $\ci$ or real-analytic manifolds and all complex spaces are assumed to be Hausdorff and second countable. If $M$ and $M'$ are $\ci$  manifolds, we endow the space of $\ci$ maps $\ci(M,M')$ with the weak $\ci$
topology \cite[p.~36]{Hirsch}.  We refer to \cite[p.~60]{Hirsch} for the definition of the $r$-jet of a $\ci$ map
$g: M\to M'$ at $x\in M$.
We denote by $[M]\in H_d(M,\Z/2\Z)$ the fundamental class of a compact $\ci$ manifold of dimension $d$.

Let $\Omega$ be a complex manifold of dimension $1$, and let $\sO_{\Omega,x}$ be the discrete valuation ring of germs of holomorphic complex-valued functions at $x\in\Omega$. If $X$ is a variety over~$\C$, morphisms $\Spec(\sO_{\Omega,x})\to X$ correspond bijectively to germs at $x\in\Omega$ of holomorphic maps $\Omega\to X(\C)$. Let $f:X\to Y$ be a modification of varieties over~$\C$ and let $g:\Omega\to Y(\C)$ be a holomorphic map such that the image by~$g$ of any connected component of $\Omega$ meets the locus where $f$ is an isomorphism. By the valuative criterion of properness applied to the $(\sO_{\Omega,x})_{x\in\Omega}$, the morphisms $\Spec(\sO_{\Omega,x})\to Y$ induced by~$g$ lift uniquely to morphisms $\Spec(\sO_{\Omega,x})\to X$. This gives rise to a holomorphic lift $h:\Omega\to X(\C)$ of~$g$, called the \textit{strict transform} of~$g$. These constructions also work if $\Omega$ is a real-analytic manifold of dimension $1$ and $g$ is a real-analytic map, using the discrete valuation rings $\sO^{\an}_{\Omega,x}$ of germs of real-analytic real-valued functions at $x\in\Omega$.

\section{Tight approximation}
\label{sectight}
Recall that we have fixed in \S\ref{conventions} a smooth projective connected curve~$B$ over~$\R$, and that  $F=\R(B)$.

\subsection{Tight approximation for models over \texorpdfstring{$B$}{𝐵}}

Tight approximation will be a property of smooth varieties over $F$ (see Definition \ref{deftight2}).
It is convenient to first define it for a proper regular model over $B$ of such a variety.

\begin{defn}
\label{deftight}
Let $f: \kX\to B$ be a proper flat
 morphism with $\kX$ regular.
One says that $f$
 satisfies the \textit{tight approximation property}
 if for all $G$-stable open neighbourhoods $\Omega$ of $B(\R)$ in $B(\C)$, all $m\geq 0$,
all $b_1,\dots, b_m \in \Omega$, all $r\geq 0$ and all $G$-equivariant holomorphic sections $u: \Omega\to\kX(\C)$ of $f(\C):\kX(\C)\to B(\C)$ over $\Omega$, there exists a sequence $s_n: B\to\kX$ of sections of $f$ having the same $r$-jets as $u$ at the $b_i$ and such that $s_n(\C)|_\Omega$ converges to $u$ in $\ci(\Omega,\kX(\C))$.
\end{defn}

In practice, to verify that the tight approximation property holds, we will use the following variant of Definition \ref{deftight}.

\begin{prop}
\label{deftightK}
A proper flat
 morphism $f: \kX\to B$ with $\kX$ regular satisfies the tight approximation property
if and only if for all $G$-stable compact subsets $B(\R)\subset K\subset B(\C)$,
all $m\geq 0$, all $b_1,\dots, b_m \in K$, all $r\geq 0$, all $G$-stable Stein open neighbourhoods $\Omega$ of $K$ in $B(\C)$ and all $G$-equivariant holomorphic sections $u: \Omega\to\kX(\C)$ of $f(\C)$ over $\Omega$, there exists a sequence $s_n: B\to\kX$ of sections of $f$ with the same $r$-jets as $u$ at the $b_i$ such that $s_n(\C)|_K$ converges uniformly to $u|_K$.
\end{prop}

\begin{proof}
The condition is obviously necessary, and we show that it is sufficient.
Since the statement in Definition \ref{deftight} is automatically verified for $\Omega=B(\C)$ as $u$ is then algebraic by GAGA (see \S\ref{Ranal}), we may assume that $\Omega\neq B(\C)$, hence that $\Omega$ is Stein \cite[p.~134]{Stein}.
%Pour un ouvert d'une surface de Riemann compacte, l'énoncé est évident.
By \cite[Theorem 2.2.3]{Hormander}, the sequence $s_n(\C)|_\Omega$ converges to $u$ in $\ci(\Omega,\kX(\C))$ if and only if it converges to $u$ uniformly on every compact subset $K\subset \Omega$.
Choosing an exhaustion of $\Omega$ by $G$-stable compact subsets containing the $b_i$ finishes the proof.
\end{proof}

\subsection{Particular cases and variants}
\label{particular}

\subsubsection{Complex fibrations}
\label{complexfibrations}

If $B'$ is a smooth projective connected curve over $\C$ with generic point $\eta'$ and $f': \kX'\to B'$ is a proper flat morphism with $\kX'$ regular, one can mimic Definition \ref{deftight}, by disregarding all references to the action of $G$. The resulting property of $f'$ is equivalent to the tight approximation property for $f: \kX\to B$, where $f=f'$, $\kX=\kX'$ and $B=B'$, but
 $B$ is viewed as a (not geometrically connected) smooth projective connected curve over $\R$ by composing its structural morphism with $\Spec(\C)\to\Spec(\R)$. (Use the fact that $B(\C)=(B'\times_{\Spec(\R)}\Spec(\C))(\C)$ is the disjoint union of two copies of $B'(\C)$ exchanged by the action of $G$, and the similar description of $\kX(\C)$.) In this way, Definition \ref{deftight} encompasses both real and complex fibrations.

\subsubsection{Constant fibrations}
\label{constantfibrations}

When $\kX=X\times B$ for some smooth and proper variety $X$ over $\R$ and $f$ is the second projection,
 the tight approximation property generalizes
Bochnak and Kucharz' property $(B)$ for the variety $X$, defined in \cite[p.~604]{BKrat} when $K=B(\R)$
(where $K\subset B(\C)$ is a compact subset as in Proposition \ref{deftightK})
and in \cite[p.~88]{BKC} when $B$ is a smooth projective connected curve over $\C$ (keeping \S\ref{complexfibrations} in mind).

\subsubsection{Nash approximation}
\label{Nash}

  In \cite{DLS}, Demailly, Lempert and Shiffman consider a variant of the tight approximation property. They restrict to constant complex fibrations (as in \S\S \ref{complexfibrations}--\ref{constantfibrations}), allow possibly higher-dimensional bases, but only look for Nash sections, not algebraic ones. Dropping the algebraicity requirement on the section yields a very general theorem \cite[Theorem 1.1]{DLS}. Tight approximation cannot possibly hold in this generality (see Proposition \ref{tightimpliesRC}).

\subsection{Sections over the real locus}
\label{realsecpar}

The tight approximation property is only interesting for morphisms $f: \kX\to B$ as above such that $f(\R): \kX(\R)\to B(\R)$ has a $\ci$ section. (Otherwise, it holds trivially as the property to be checked is vacuous.) This condition is equivalent to the existence of a $\sC^0$ section of $f(\R)$ with values in the locus $U\subset X(\R)$ where $f(\R)$ is submersive. Indeed, $\ci$ sections of $f(\R)$ always take values in $U$, and one can approximate $\sC^0$ sections with values in $U$ by $\ci$ sections (using the normal form theorem for submersions and \cite[Chapter~2, Theorem 2.4]{Hirsch}).
When a $\ci$ section of $f(\R)$ exists, the tight approximation property allows one to approximate it by algebraic sections, as we show in Corollary~\ref{realsectight}.

\begin{prop}
\label{realsecprop}
Let $f: \kX\to B$ be a proper flat
 morphism with $\kX$ regular. Let $v: B(\R)\to \kX(\R)$ be a $\ci$ section  of $f(\R)$. Choose $b_1,\dots, b_m\in B(\R)$ and $r\geq 0$.
Then there exist a sequence of $G$-stable open neighbourhoods $\Omega_n$ of $B(\R)$ in $B(\C)$ and $G$-equivariant holomorphic sections $u_n:\Omega_n\to \kX(\C)$ of $f(\C)$ over $\Omega_n$ such that the $(u_n)|_{B(\R)}$ have the same $r$-jets as $v$ at the $b_i$ and converge to $v$ in $\ci(B(\R),\kX(\R))$.
\end{prop}

\begin{cor}
\label{realsectight}
Under the assumptions of Proposition \ref{realsecprop}, if $f$ satisfies the tight approximation property, there exists a sequence $s_n: B\to\kX$ of sections of $f$ such that the $s_n(\R)$ have the same $r$-jets as $v$ at the $b_i$ and converge to $v$ in $\ci(B(\R),\kX(\R))$.
\end{cor}

%Pour les sections C^0 générales, même sans imposer de 0-jet, c'est vraiment un problème si la section passe par une composante triple de la fibre, par exemple dans le cas z\mapsto z^3 de P^1->P^1.
%par contre C^r-sections pour r>=1 marche comme \ci.

\begin{proof}[Proof of Proposition \ref{realsecprop}]
By Whitney's approximation theorem (Lemma \ref{Whitney} below),
there exists a sequence $v_n:  B(\R)\to \kX(\R)$ of real-analytic maps
with the same $r$-jets as $v$ at the $b_i$ converging to $v$ in $\ci(B(\R),\kX(\R))$.
If $n\gg 0$, the map $f(\R)\circ v_n$ is a diffeomorphism
of $B(\R)$ \cite[Corollary 5.7]{Michor},
 and $v'_n:=v_n\circ(f(\R)\circ v_n)^{-1}$ converges to $v$ in $\ci(B(\R),\kX(\R))$
 \cite[Theorem 7.6]{Michor}. Locally at $x\in B(\R)$, the map $v'_n$ is given by convergent power series, and extends in a unique way to a germ of holomorphic map $B(\C)\to\kX(\C)$ at $x$. By uniqueness, these local extensions glue to a $G$-equivariant holomorphic map $u_n: \Omega_n\to \kX(\C)$ on some $G$-stable open neighbourhood $\Omega_n$ of $B(\R)$ in $B(\C)$ each connected component of which meets~$B(\R)$, and $u_n$ is a section of $f(\C)$ over $\Omega_n$ by analytic continuation.
\end{proof}

\begin{lem}
\label{Whitney}
Let $g: M\to M'$ be a $\ci$ map between real-analytic manifolds, let $S\subset M$ be a discrete subset, let $r: S\to \N$ be a function, and let $\sU\subset \ci(M,M')$ be a neighbourhood of $g$ for the strong $\ci$ topology \cite[p.~36]{Whitney}.
%Whitney fait mieux : il peut faire tendre la quantité de dérivées approchées près du bord vers l'infini
Then there exists a real-analytic map $h\in\sU$ that has the same $r(s)$-jet as $g$ at $s$ for all $s\in S$.
\end{lem}

\begin{proof}
We may assume~$M$ and~$M'$ to be connected.
By the Grauert--Morrey theorem \cite[Theorem 3]{Grauert}, we may then assume
that~$M$ and~$M'$ are embedded in Euclidean spaces. In this case, the lemma without jets is Whitney's approximation theorem \cite[Theorem 2]{Whitney}. With jets, it is claimed without proof in \cite[p.~654]{Whitney}. It can be reduced to the case where $M'=\R$ by the argument of \cite[Lemma 22]{Whitney}, in which case it follows from \cite[Theorem 3.3]{Tognoli}.
\end{proof}

\subsection{Weak approximation}
\label{parapprox}
We finally discuss the relation between the tight approximation property and the more classical weak approximation property.

\begin{defn}
\label{weakdef}
A smooth variety $X$ over $F$ satisfies the \textit{weak approximation property} if the image of the diagonal map $X(F)\to\prod_b X(F_{b})$ is dense with respect to the product topology (endowing each~$X(F_b)$ with
the topology defined by the discrete valuation on~$F_b$), where the product runs over all closed points $b\in B$.

We will say that a proper flat morphism $f:\kX\to B$ with $\kX$ regular satisfies the weak approximation property if so does its generic fiber $\kX_\eta$.
\end{defn}

This property may fail since local data at finitely many real points of $B$ might not be interpolable by a $\ci$ section of $f(\R)$, let alone by an algebraic section of~$f$. This is precisely Colliot-Th\'el\`ene's reciprocity obstruction \cite[\S3]{CTgroupes} to the weak approximation property, as reinterpreted by Ducros (\cite[Th\'eor\`eme 4.3]{Ducros}, see also \cite[Proposition 3.19]{PalSzabo}).

\begin{defn}
\label{reciprocity}
Let $f:\kX\to B$ be a proper flat morphism with $\kX$ regular.
Let $\Xi\subset \prod_b\kX_\eta(F_{b})$ be the subset of $(v_b)\in\prod_b\kX_\eta(F_{b})$ whose projection in $\prod_{b\in B(\R)}\kX_\eta(F_{b})$ is induced by a $\ci$ section $v:B(\R)\to\kX(\R)$ of $f(\R)$.
We say that \textit{the reciprocity obstruction is the only obstruction to the validity of the weak approximation property for $f$} if the image of the diagonal map $\kX_{\eta}(F)\to\prod_b\kX_\eta(F_{b})$ is dense in $\Xi$. We say that there is \textit{no reciprocity obstruction to the validity of the weak approximation property for $f$} if $\Xi$ is dense in $\prod_b\kX_\eta(F_{b})$.
\end{defn}

\begin{rmk}
Using
\cite[Lemme~3.5, Lemme~3.6]{CTgroupes},
\cite[Théorème~3.5, proof of Proposition~4.1]{Ducros}
and \cite[Chapter~2, Theorem 2.4]{Hirsch},
one checks that the closure of~$\Xi$
in $\prod_b\kX_\eta(F_{b})$ coincides
with the subset denoted $\sE(\kX_\eta)$ in \cite[Définition~1.8]{Ducros}.
\end{rmk}

The tight approximation property is designed to take into account the reciprocity obstruction, by incorporating the assumption that $B(\R)\subset\Omega$ into Definition \ref{deftight}.

\begin{prop}
\label{implitight}
Let $f:\kX\to B$ be a proper flat morphism with $\kX$ regular that satisfies the tight approximation property. Choose $(v_b)\in\prod_b\kX_\eta(F_{b})$ whose projection in $\prod_{b\in B(\R)}\kX_\eta(F_{b})$ is induced by a $\ci$ section $v:B(\R)\to\kX(\R)$ of $f(\R)$. Let $\Sigma\subset B$ be a finite set of closed points, and fix $r\geq 0$.
Then there exists a sequence $s_n: B\to\kX$ of sections of $f$ that coincide to order $r$ with $v_{b}:\Spec(\widehat{\sO_{B,b}})\to \kX$ for all $b\in \Sigma$ and such that $s_n(\R)$ converges to $v$ in $\ci(B(\R),\kX(\R))$.
\end{prop}

\begin{cor}
\label{tightweak}
Let $f:\kX\to B$ be a proper flat morphism with $\kX$ regular. If $f$ satisfies the tight approximation property, then the reciprocity obstruction is the only obstruction to the validity of the weak approximation property for $f$.
\end{cor}

\begin{proof}[Proof of Proposition \ref{implitight}]
Let $\Sigma_1\subset \Sigma$ (resp.\ $\Sigma_2\subset \Sigma$) be the subset of points with real (resp.\ complex) residue fields.
By Proposition \ref{realsecprop}, there exists a sequence of $G$-stable open neighbourhoods $\Omega_n$ of $B(\R)$ in $B(\C)$ and of $G$-equivariant holomorphic sections $u_n:\Omega_n\to \kX(\C)$ of $f(\C)$ over $\Omega_n$ such that the $(u_n)|_{B(\R)}$ coincide with $v_{b}$ to order $r$ for all $b\in \Sigma_1$ and converge to $v$ in $\ci(B(\R),\kX(\R))$.
For $b\in \Sigma_2$, let $\sO_{B(\C),b}\subset \widehat{\sO_{B,b}}$ be the subring of convergent power series. It is a Henselian discrete valuation ring with completion~$\widehat{\sO_{B,b}}$ by \cite[Theorem~45.5]{Nagata}. By Greenberg's theorem \cite[Corollary~1]{Greenberg}, we may assume that $v_{b}$ is given by a morphism $v_{b}:\Spec(\sO_{B(\C),b})\to\kX$ for all $b\in \Sigma_2$. The $(v_{b})_{b\in \Sigma_2}$ then give rise to a $G$-equivariant holomorphic section $u':\Omega'\to\kX(\C)$ of $f(\C)$ over some $G$-stable open subset $\Omega'\subset B(\C)$ containing~$\Sigma_2(\C)$.
After shrinking~$\Omega_1$ and~$\Omega'$
and then replacing~$\Omega_n$ with $\Omega_1\cap\Omega_n$ for all~$n$, we may assume that $\Omega_n \cap \Omega'=\emptyset$ for all~$n$.
Applying Definition \ref{deftight} to the section $(u_n,u'):\Omega_n\cup\Omega'\to\kX(\C)$ of $f(\C)$ over $\Omega_n\cup\Omega'$, to the finite set $\Sigma(\C)\subset B(\C)$ and to the integer $r$ concludes the proof.
\end{proof}

Using the following proposition, it is sometimes possible to deduce the weak approximation property itself (for instance when $B(\R)=\varnothing$; see also Theorem \ref{weakhomoreal}).

\begin{prop}
\label{weakconnected}
Let $f:\kX\to B$ be a proper flat morphism with $\kX$ regular. Assume that the reciprocity obstruction is the only obstruction to the validity of the weak approximation property for $f$
and that there exists a dense open subset $U\subset B$ such that $\kX_b(\R)$ is connected for all $b\in U(\R)$. Then $f$ satisfies the weak approximation property.
\end{prop}

\begin{proof}
After shrinking~$U$, we may assume that~$f$ is smooth above~$U$.
If $\kX_{\eta}(F_b)=\varnothing$ for some $b\in B(\R)$, then~$f$ satisfies
the weak approximation property.  Otherwise,
as~$f$ is proper,
the map $f^{-1}(U(\R))\to U(\R)$ induced by~$f$
is surjective.
On the other hand
it is a locally trivial fibration with connected fibres,
by Ehresmann's theorem;
therefore it possesses a~$\ci$ section $v:U(\R)\to\kX(\R)$, which we fix.

Choose $r\geq 0$, $b_1,\dots,b_m\in B(\R)$ and $v_{b_i}\in \kX_{\eta}(F_{b_i})$ for $1\leq i \leq m$. After enlarging $\{b_1,\dots ,b_m\}$ and shrinking~$U$, we may assume that $B(\R)\setminus\{b_1,\dots ,b_m\} = U(\R)$.
View the $v_{b_i}$ as morphisms $v_{b_i}:\Spec(\widehat{\sO_{B,b_i}})\to \kX$ by the valuative criterion of properness. Applying Greenberg's theorem as in the proof of Proposition~\ref{implitight} yields an open neighbourhood $W_i$ of $b_i$ in $B(\R)$ and a real-analytic section  $v_i:W_i\to\kX(\R)$ of $f(\R)$ over $W_i$ that coincides with $v_{b_i}$ to order $r$ at $b_i$.

 Since $\kX_b(\R)$ is connected for $b\in U(\R)$, one can use Ehresmann's theorem again to modify $v$ in small neighbourhoods of the $b_i$ so that it can be glued to the~$v_i$ to yield
a $\ci$ section $v':B(\R)\to\kX(\R)$ of $f(\R)$ that coincides with $v_i$ near~$b_i$. This shows that there cannot be any reciprocity obstruction to the validity of the weak approximation property for $f$, as required.
\end{proof}

\section{Birational aspects}

\subsection{Birational invariance}

In Theorem \ref{birinv}, we show the birational invariance of the tight approximation property.
If one restricts to constant fibrations as in \S\ref{constantfibrations} and to $K=B(\R)$ in the notation of Proposition \ref{deftightK}, this is due to Bochnak and Kucharz
\cite[\S 2]{BKrat}.
The corresponding result for the
weak approximation property is due to Kneser \cite[\S 2.1]{Kneser}.

\begin{thm}
\label{birinv}
Let $f: \kX\to B$ and $f': \kX'\to B$ be proper flat morphisms with $\kX$ and $\kX'$ regular.
Let $g: \kX\dashrightarrow \kX'$ be a birational map such that $f'\circ g=f$.  If $f'$ satisfies the
tight approximation property, then so does $f$.
\end{thm}

We start with the particular case of Theorem \ref{birinv} where $h:=g^{-1}$ is a morphism.

\begin{lem}
\label{birfacile}
Let $f: \kX\to B$ and $f': \kX'\to B$ be
proper flat morphisms with $\kX$ and $\kX'$ regular.
Let $h: \kX'\to \kX$ be a birational morphism with $f\circ h=f'$.
If $f'$ satisfies the tight approximation property, then so does $f$.
\end{lem}

\begin{proof}
Let $K$, $\Omega$, $b_i$, $r$ and $u$ be as in Proposition \ref{deftightK}.
By Proposition \ref{tubularR} (ii), one may assume that no connected component of $u(\Omega)$ is included in the locus above which $h$ is not an isomorphism. The strict transform of $u$ in $\kX'(\C)$ is a section $u': \Omega\to\kX'(\C)$ of $f'(\C)$ above $\Omega$ that lifts $u$. Applying the tight approximation property of $f'$ to
$K$, $\Omega$, $b_i$, $r$ and $u'$ yields a sequence $s'_n: B\to\kX'$ of sections of $f$ such that the sequence $s_n=h\circ s'_n$ has the required properties.
\end{proof}

\begin{proof}[Proof of Theorem \ref{birinv}]
By Hironaka's resolution of indeterminacies theorem \cite[\S 5]{Hironaka} and by Lemma \ref{birfacile}, one
may assume that $g$ is the blow-up of a smooth integral subvariety of $\kX'$ with exceptional divisor
$E\subset \kX$.
Let $K$, $\Omega$, $b_i$, $r$ and $u$ be as in Proposition \ref{deftightK}. By Proposition \ref{tubularR} (ii), one may assume, after shrinking~$\Omega$, that no connected component of $u(\Omega)$ is included in $E(\C)$. Since $K$ is compact, shrinking~$\Omega$ further, one may assume that $u(\Omega)$ meets $E(\C)$ at only finitely many points. We add these points to the collection of the $b_i$. A local computation (carried out in \cite[Proof of Lemma~2.1]{BKrat} in a
very similar context, see also \cite[Proof of Lemma~2.1]{BKC}) shows that if $r'\geq 0$ is big enough,
and if $s_n': B\to\kX'$ is a sequence of sections of $f'$ obtained by applying the tight approximation property of $f'$ to $K$, $\Omega$, $b_i$, $r'$ and $g\circ u$, then the sequence $s_n: B\to\kX$ of strict transforms of $s'_n$ has the required properties.
\end{proof}

\subsection{Tight approximation for varieties over \texorpdfstring{$F$}{𝐹}}

We can now give the definition.

\begin{defn}
\label{deftight2}
A smooth variety $X$ over $F$ satisfies the \textit{tight approximation property} if some proper regular model $f: \kX\to B$ of $X$ over $B$ satisfies the tight approximation property in the sense of Definition \ref{deftight}.
\end{defn}

\begin{rmks}
(i) By Theorem \ref{birinv}, this does not depend on the model $f: \kX\to B$ of $X$, and the tight approximation property is a birational invariant of $X$.

(ii) The case when $X$ is defined over $\R$ corresponds to the case of product fibrations $f: X\times B\to B$ and is of particular interest.
\end{rmks}

This property is only relevant for the varieties whose smooth compactifications are rationally connected (in the geometric sense specified in \S\ref{conventions}).

\begin{prop}
\label{tightimpliesRC}
Let $f: \kX\to B$ be a proper flat morphism with $\kX$ regular and generic fiber $X$. If $f(\R)$ has a $\ci$ section and if $f$ satisfies the tight approximation property, then $X$ is rationally connected.
\end{prop}

\begin{proof}
Since $f$ satisfies the tight approximation property and $f(\R)$ has a $\ci$ section,
Corollary~\ref{realsectight}
implies that $f$ has a section. As moreover the reciprocity obstruction is the only obstruction to the validity of the weak approximation property for~$f$ by Corollary~\ref{tightweak}, it follows that $X_{F(\sqrt{-1})}$ satisfies the hypothesis of \cite[Corollary~2.16]{Hassett}, hence is rationally connected.
\end{proof}

Proposition \ref{tightimpliesRC} is in the spirit of \cite[Theorem 1.2]{BKrat} and \cite[Theorem 1.3]{BKC}, with a stronger conclusion.
We are interested in converse statements.

\begin{question}
\label{qmain}
Do smooth proper rationally connected varieties over $F$ satisfy the tight approximation property?
\end{question}

\begin{rmks}
\label{sanspoint}
(i) When $B(\R)=\varnothing$, the weaker question whether smooth proper rationally connected varieties $X$ over $F$ satisfy the weak approximation property (see \S \ref{parapprox}) is open. It is not even known whether $X$, if non-empty, always has an $F$-point, even if $X$ is defined over $\R$.
 Applied with $B$ equal to the anisotropic conic over $\R$, this would show the existence of a rational curve in an arbitrary positive-dimensional rationally connected variety~$X$ over $\R$ when $X(\R)=\emptyset$,
which would answer a question raised in \cite[Remarks 20]{Arako}.

(ii) If $B(\R)=\varnothing$, and if $X$ is a smooth proper rationally connected variety over~$F$ that is known to have an $F$-point, weak approximation holds at places of good reduction, as the arguments of Hassett and Tschinkel \cite[Theorem 3]{HT} adapt to this situation. It is not known if it also holds at places of bad reduction, even if $B$ is a complex curve (see \cite[Conjecture 2]{HT}, \cite{Hassett}).
\end{rmks}

\begin{rmk}
Although the tight approximation property is interesting only for rationally connected varieties, one may ask weaker questions for other classes of varieties, such as $K3$ surfaces, or Enriques surfaces.
For instance, if $X$ is a $K3$ surface over $\R$, can any homologically trivial loop on $X(\R)$ be approximated for the $\ci$ topology by the real locus of a rational curve on $X$? More generally, what about loops whose homology classes belong to the image $H_1^{\alg}(X(\R),\Z/2\Z)$ of the Borel--Haefliger cycle class map? The particular case of the Fermat quartic surface is raised by Bochnak and Kucharz \cite[p.~602]{BKrat}. We do not even know if all real $K3$ surfaces with real points contain rational curves, although complex K3 surfaces always contain rational curves by
Bogomolov and Mumford \cite[Appendix]{MM}.
\end{rmk}

\section{Rational varieties}

\subsection{Projective spaces}
As a first example, we show that projective spaces satisfy the tight approximation property. Variants of this statement are due to Bochnak and Kucharz (\cite[Lemma 2.3]{BKrat}, \cite[Lemma 2.5]{BKC}), and we adapt their arguments.
We need a real analogue of Runge's approximation theorem, which is easy to deduce from classical statements.
%Stone Weierstrass avec dérivées mais sans jets = [Narasimhan, Analysis..., 1.6.2].

\begin{prop}
\label{approx}
Let $K\subset B(\C)$ be a $G$-stable compact subset, let $\Omega$ be a $G$-stable open neighbourhood of $K$ in $B(\C)$, let $f: \Omega\to\C$ be a $G$-equivariant holomorphic function, and choose $b_1,\dots,b_m\in K$ and $r\geq 0$. Then there is a sequence of
rational functions $f_n\in\R(B)$ such that the $f_n(\C)$ have no poles on $K$, have the same $r$-jets as $f$ at the $b_i$, and such that $f_n(\C)|_K$ converges uniformly to $f|_K$.
\end{prop}

\begin{proof}
The statement is obvious if $K=B(\C)$, as $f$ is then constant. Otherwise, using the Riemann--Roch theorem, we find rational functions $g,h\in \R(B)^*$ such that $g(\C)$ and $h(\C)$ have no poles on $K$, such that $g(\C)$ has the same $r$-jets as $f$ at the $b_i$, and such that $h(\C)$ vanishes to order exactly $r$ at the $b_i$. Set $f'=(f-g(\C))/h(\C)$. By \cite[Satz 1]{KT} (see also \cite[Theorem 1.1]{Simha}), there exists a sequence $f'_n$ of rational functions on $B$ with no poles on $K$ such that $f'_n(\C)|_K$ converges uniformly to $f'|_K$. Replacing $f'_n$ with $(f_n'+\sigma(f_n'))/2$, we may assume that $f'_n\in\R(B)$. The sequence $f_n=hf'_n+g$ has the required properties.
%non compact \cite[Satz 13]{BS}
%cf [The legacy of Weierstrass, Runge, Oka, Weil..., Theorem 4]
\end{proof}

\begin{thm}
\label{projective}
 The projective space $\P^d_F$ satisfies the tight approximation property.
\end{thm}

\begin{proof}
Consider the model $\kX=\P^d_\R\times B$ of $\P^d_F$ with first projection $p: \P^d_\R\times B\to \P^d_\R$.
Let $K$, $\Omega$, $b_i$, $r$ and $u$ be as in Proposition \ref{deftightK}.
Set $v=p(\C)\circ u: \Omega\to\P^d(\C)$.
We only have to construct morphisms $g_n: B\to \P^d_\R$ such that the $g_n(\C)$ have the same $r$-jets as $v$ at the $b_i$, and converge uniformly to $v$ on $K$. Indeed, one can then choose $s_n=(g_n,\Id): B\to \P^d_\R\times B$.

Analytifying $\sO_{\P^d_\R}(1)$ as in \S\ref{Ranal}
endows the invertible sheaf $\sO_{\P^d(\C)}(1)$ on $\P^d(\C)$ with a structure of $G$-equivariant invertible sheaf whose coordinate sections $Y_0,\dots, Y_d$ are $G$-invariant.
 Consider the commutative diagram:
\begin{equation}
\label{clR}
\begin{aligned}
\xymatrix
{
\Pic(B)=\Pic_G(B(\C))\ar[r]\ar^{\cl_\R}[d]&\Pic_G(\Omega)\ar^{\cl_\R}[d]  \\
H^1(B(\R),\Z/2\Z)\ar[r]&H^1(\Omega^G,\Z/2\Z),
}
\end{aligned}
\end{equation}
whose horizontal arrows are restriction maps, whose vertical maps are Borel--Haefliger maps (defined in \S\ref{RPicard}), and where the equality $\Pic(B)=\Pic_G(B(\C))$ follows from GAGA (see \S \ref{Ranal}).
The right vertical map is an isomorphism by Proposition \ref{linebundleR}. The bottom horizontal map is a surjection because $\Omega^G$ is an open subset of the one-dimensional $\ci$ manifold $B(\C)^G=B(\R)$. The left vertical map is surjective because if $\Theta\subset B(\R)$ is a connected component and $x\in \Theta$, the class $\cl_{\R}(\sO(x))$ is non-zero in $H^1(\Theta,\Z/2\Z)$ but vanishes on the other connected components of $B(\R)$, by \cite[Remark 1.3.2]{krasnovcharacteristicclasses} or \cite[Theorem 4.2]{krasnovequivariant}. We then deduce from (\ref{clR}) the existence of a line bundle $\sL\in\Pic(B)$ such that $\sL^{\an}|_\Omega\simeq v^*\sO_{\P^d(\C)}(1)$.

Replacing $\sL$ with $\sL(lx)$ for some closed point $x\in B$ not in $\Omega$ and $l\gg0$, we may assume that there exists a surjection $q: \sO^N_B\surj\sL$ of locally free sheaves on~$B$.
It induces a surjection $q^{\an}|_\Omega: \sO_{\Omega}^N\surj v^*\sO_{\P^d(\C)}(1)$ of $G$-equivariant locally free sheaves on $\Omega$.
 By Lemma \ref{SteinR} (ii), one can lift $v^*Y_0,\dots, v^*Y_d$ to $G$-invariant sections $\zeta_0,\dots,\zeta_d\in H^0(\Omega,\sO_{\Omega}^N)^G$. For $0\leq j\leq d$, view $\zeta_j$ as a collection $(f_{j,k})_{1\leq k\leq N}$ of $G$\nobreakdash-equivariant holomorphic functions $\Omega\to \C$.
 By Proposition \ref{approx}, there is a sequence $f_{j,k,n}\in\R(B)$ such that the $f_{j,k,n}(\C)$ have no poles on $K$, have the same $r$\nobreakdash-jets as~$f_{j,k}$ at the $b_i$ and converge uniformly to $f_{j,k}$ on $K$. Through the morphism~$q$, the $(f_{j,k,n})_{1\leq k\leq N}$ induce a rational section $\xi_{j,n}$ of $\sL$. For $n\gg0$, the~$\xi_{j,n}$ do not vanish simultaneously on the compact $K$, since the $v^*Y_j$ do not vanish simultaneously. It follows that they define rational maps $g_n: B\dashrightarrow \P^d_{\R}$, which extend to morphisms $g_n: B\to \P^d_{\R}$ by the valuative criterion of properness, and have the required properties.
\end{proof}

\subsection{Stable birational invariance}

Recall that two varieties $X$ and $X'$ over $F$ are \textit{stably birational} if there exist $d, d'\geq 0$ such that $\P^d_F\times X$ and $\P^{d'}_F\times X'$ are birational.
A variety is said to be \textit{stably rational} if it is stably birational to the point.

\begin{prop}
\label{stableinv}
Let $X$ and $X'$ be stably birational smooth varieties over $F$. If $X$ satisfies the tight approximation property, then so does $X'$.
\end{prop}

\begin{proof}
By Theorem \ref{birinv}, it suffices to show that, for $d\geq 0$, a smooth variety $X$ over~$F$ satisfies the tight approximation property if and only if so does $\P^d_F\times X$. Letting $f: \kX\to B$ be a proper regular model of $X$, and choosing $\P^d_\R\times \kX$ as a proper regular model for $\P^d_F\times X$, this is an immediate consequence of Theorem \ref{projective}.
\end{proof}

\begin{cor}
\label{rat}
Smooth stably rational varieties over $F$ satisfy the tight approximation property.
\end{cor}

\begin{example}
\label{quadrique}
Smooth quadrics of dimension~$\geq 1$ and Severi--Brauer varieties over~$F$ satisfy the tight approximation property. Indeed, if $f: \kX\to B$ is a proper regular model of a quadric of dimension~$\geq 1$ (resp.\ of a Severi--Brauer variety) $X$ over $F$,  and if $u: \Omega\to\kX(\C)$ is a $G$-equivariant holomorphic section of $f(\C)$ over a $G$-stable open neighbourhood $\Omega$ of $B(\R)$ in $B(\C)$ as in Definition \ref{deftight}, then $u|_{B(\R)}$ is a section of $f(\R)$, so that $X(F)\neq\varnothing$ by a theorem of Witt  \cite[Satz 22]{Witt} (resp.\ by \cite[(1.4)]{DemeyerKnus}). It follows that $X$ is $F$-rational, and Corollary \ref{rat} applies.
We will greatly generalize these examples in Theorem \ref{homogeneous}.
\end{example}

\section{The fibration method}
\label{secfib}

 In Theorem \ref{up}, we explain how to deduce the tight approximation property for the total space of a fibration if it is known for its base and for its general fibers.
In the complex setting, for the weak approximation property, this is due to Colliot-Th\'el\`ene and Gille \cite[Proposition 2.2]{CTG}. In the real setting, for the property that the reciprocity obstruction is the only obstruction to the validity of the weak approximation property, such a fibration theorem was proven by P\'al and Szab\'o \cite[Theorem 1.5]{PalSzabo}
under the additional assumption that the base is rational over~$F$.

\begin{thm}
\label{up}
Let $g: X\to X'$ be a dominant morphism of smooth varieties over~$F$.
If the tight approximation property holds for $X'$ as well as for $X_x=g^{-1}(x)$ for all
$F$\nobreakdash-points~$x$ of a dense open subset of~$X'$, then it holds for $X$.
\end{thm}

Our main tool is Proposition \ref{toroidaltheorem} which ensures that an analytic curve in the total space of a fibration $f:X\to Y$ can be made to avoid the singular locus of $f$ by passing to appropriate modifications of $X$ and $Y$. Its proof, given in \S\ref{paravoiding}, relies on toroidal geometry; we recall basic definitions in \S\ref{partoroidal} and refer to \cite{AKaru, ADK} for more information. Theorem \ref{up} is proven in \S\ref{parfib} and a few applications are given in~\S\ref{applfib}.

\subsection{Toroidal embeddings}
\label{partoroidal}

Let $k$ be a field of characteristic $0$. An open immersion $U\subset X$ of varieties over $k$ is a \textit{toroidal embedding} if for all $x\in X(\overline{k})$, there exist a toric variety $Z$ over $\overline{k}$ with open orbit $T\subset Z$, a point $z\in Z(\overline{k})$ and an isomorphism $\widehat{\sO_{X_{\overline{k}},x}}\isoto\widehat{\sO_{Z,z}}$ of $\overline{k}$-algebras sending the completion of the ideal of $X\setminus U$ to the completion of the ideal of $Z\setminus T$. We call such an isomorphism a \textit{toric chart} of $U\subset X$ at $x$.

 A toroidal embedding $U\subset X$ is said to be \textit{strict normal crossings} if~$X$ is smooth and the irreducible components $(D_i)_{1\leq i\leq n}$ of $X \setminus U$ have the property that for all $I\subset \{1,\dots,n\}$ of cardinality $c$, the subvariety $D_I:=\cap_{i\in I} D_i\subset X$ is smooth of pure codimension $c$ (possibly empty). The non-empty $D_I$ are called the \textit{strata} of $U\subset X$.

A \textit{toroidal morphism} $f:(U\subset X)\to (V\subset Y)$ of toroidal embeddings is a dominant morphism $f:X\to Y$ with $f(U)\subset V$ such that for all $x\in X(\overline{k})$
 the morphism $\hat{f}^{*}:\widehat{\sO_{Y_{\overline{k}},f(x)}}\to \widehat{\sO_{X_{\overline{k}},x}}$, viewed in appropriate toric charts of $U\subset X$ and $V\subset Y$, is induced by the completion of a toric morphism of toric varieties over $\overline{k}$.

\subsection{Avoiding the singular locus of a fibration}
\label{paravoiding}

The following statement will be key to the proof of Theorem \ref{up}.
In practice, it will be applied in combination with the weak toroidalization theorem \cite[Theorem 1.1]{ADK}.

\begin{prop}
\label{toroidaltheorem}
Let $f: (U\subset X)\to (V\subset Y)$ be a toroidal
morphism of snc toroidal embeddings,
with $X$ and $Y$ quasi-projective, over a field $k$ of characteristic~$0$.
Let $(R_j)_{1\leq j\leq m}$ be discrete valuation rings with valuations $v_j$, residue fields $\kappa_j$ and fraction fields~$K_j$, and set $T_j=\Spec(R_j)$, $t_j=\Spec(\kappa_j)$ and $\eta_j=\Spec(K_j)$.
Let $\phi_j:T_j\to X$ be morphisms and define $\psi_j:=f\circ\phi_j$. Assume that $\phi_j(\eta_j)\in U$, that the image of $\psi_j^*:\sO_{Y,\psi_j(t_j)}\to R_j$ contains a uniformizer of $R_j$, and that the induced field extensions $k\subset\kappa_j$ are algebraic.

Then there exist projective birational morphisms $\pi_X:X'\to X$ and $\pi_Y:Y'\to Y$ and a morphism $f':X'\to Y'$ with $f\circ\pi_X=\pi_Y\circ f'$ such that $\pi_X$ (resp.\ $\pi_Y$) is an isomorphism above $U$ (resp.\ above $V$)
 and such that $f'$ is smooth at $\phi_j'(t_j)$, where $\phi_j':T_j\to X'$ is the lift of $\phi_j$.
\end{prop}

Our proof has three steps. In the first two steps, we blow up $X$ and $Y$ appropriately, and the third is a verification that our goal has been reached.

\begin{Step}
\label{step1}
We may assume that the $\psi_j(t_j)$ do not belong to any stratum of $V\subset Y$ that has codimension~$\geq 2$ in $Y$.
\end{Step}

\begin{proof}
We claim that there exists a projective birational toroidal morphism of snc toroidal embeddings $\pi_Y:(V'\subset Y')\to (V\subset Y)$ such that the $\psi'_j(t_j)$ do not belong to any stratum of codimension~$\geq 2$ of $V'\subset Y'$, where $\psi'_j:T_j\to Y'$ is the lift of $\psi_j$.

We construct $\pi_Y$ as a composition of blow-ups. For some $j\in\{1,\dots,m\}$, consider the smallest stratum $D_I\subset Y$ containing $y:=\psi_j(t_j)$ (with the notation of \S\ref{partoroidal}). Let $(D_i)_{i\in I}$ be the irreducible components of $Y\setminus V$ containing $D_I$ and let $(z_i)_{i\in I}\in\sO_{Y,y}$ be local equations of the $D_i$. Define $\alpha_j(Y):=\sum_{i\in I}v_j(\psi_j^*z_i)$. Let $i_0\in I$ be such that $v_j(\psi_j^*z_{i_0})$ is minimal. Let $\hat{\pi}:\widehat{Y}\to Y$ be the blow-up of $D_I$ and define $\widehat{V}:=\hat{\pi}^{-1}(V)$ so that $\hat{\pi}:(\widehat{V}\subset\widehat{Y})\to (V\subset Y)$ is a toroidal morphism of snc toroidal embeddings (to verify that $\hat{\pi}$ is toroidal, apply \cite[Definition~3.3.17]{CLS} in toric charts). Consider the lift $\hat{\psi}_j:T_j\to\widehat{Y}$ of $\psi_j$. Local equations at $\hat{y}:=\hat{\psi}_j(t_j)$ of the irreducible components of $\widehat{Y}\setminus\widehat{V}$ through $\hat{y}$ can be chosen to be $z_{i_0}$ and the $z_i/z_{i_0}$ for those $i\in I$ such that $v_j(\psi_j^*z_{i})>v_j(\psi_j^*z_{i_0})$.
This shows that if $D_I$ has codimension $\geq 2$ in~$Y$, one has $\alpha_j(\widehat{Y})=\alpha_j(Y)-(|I|-1)v_j(\psi_j^*z_{i_0})<\alpha_j(Y)$. The claim thus follows by applying this procedure finitely many times for $j=1$,
then again for $j=2$, etc.

Considering the subdivision of the polyhedral complex associated with $U\subset X$ induced by $\pi_Y$ shows the existence of a toroidal embedding $\widetilde{U}\subset \widetilde{X}$ and of toroidal morphisms $\widetilde{f}:(\widetilde{U}\subset \widetilde{X})\to (V'\subset Y')$ and $\widetilde{\pi}:(\widetilde{U}\subset \widetilde{X})\to (U\subset X)$ with $\widetilde{\pi}$ projective and birational. (Apply \cite[Definition 8.3 1, Lemma 1.11]{AKaru} and note that this construction is defined over $k$ as it is canonical.) To conclude, consider a projective birational toroidal morphism $\nu:(U'\subset X')\to (\widetilde{U}\subset \widetilde{X})$
that is an isomorphism above the snc locus of $(\widetilde{U}\subset \widetilde{X})$ and
such that $(U'\subset X')$ is a snc toroidal embedding \cite[\S 2.5.2]{ADK}. Set $f':=\widetilde{f}\circ\nu$ and $\pi_X:=\widetilde{\pi}\circ\nu$, and replace $f$ with $f'$ and $\phi_j$ with its lift $\phi'_j:T_j\to X'$.
\end{proof}

\begin{Step}
\label{step2}
There exist projective birational morphisms $\pi_X:X'\to X$ and $\pi_Y:Y'\to Y$ and a toroidal morphism $f':(U'\subset X')\to (V'\subset Y')$ of snc toroidal embeddings with $f\circ\pi_X=\pi_Y\circ f'$ such that $\pi_X$
(resp.\ $\pi_Y$) is an isomorphism above $U$ (resp.\ above~$V$)
and such that the following holds. For $1\leq j\leq m$, letting $\psi_j':T_j\to Y'$ denote the lift of $\psi_j$, the minimal stratum $D'$ of $V'\subset Y'$ containing $\psi'_j(t_j)$ has codimension~$1$ in~$Y'$, and if $z_j\in \sO_{Y',\psi_j(t_j)}$ is a local equation of $D'$, then $(\psi'_j)^*z_j\in R_j$ is a uniformizer.
\end{Step}

\begin{proof}
The hypothesis that the field extensions $k\subset \kappa_j$ are algebraic implies that the $\psi_j(t_j)$ are closed points of $Y$. Let $H\subset Y$ be a general hypersurface, in a sufficiently ample linear system on $Y$, that contains the $\psi_j(t_j)$. By the Bertini theorem, $V\setminus (H\cap V)\subset Y$ is a snc toroidal embedding. By \cite[Proposition~3.2]{AKaru}, the morphism $f: (U\setminus(f^{-1}(H)\cap U)\subset X)\to(V\setminus (H\cap V)\subset Y)$ is toroidal.
 Let $\nu : (U_H\subset X_H)\to (U\setminus(f^{-1}(H)\cap U)\subset X)$ be a projective birational toroidal morphism that is an isomorphism above the snc locus of $(U\setminus(f^{-1}(H)\cap U)\subset X)$ and such that $(U_H\subset X_H)$ is a snc toroidal embedding \cite[\S 2.5.2]{ADK}. Define $f_H:=f\circ\nu : (U_H\subset X_H)\to(V\setminus (H\cap V)\subset Y)$.

Since $H$ has been chosen general and by our hypothesis on the $\psi_j$, a local equation $w_j$ of $H$ at $\psi_j(t_j)$ has the property that $\psi_j^*w_j\in R_j$ is a uniformizer of $R_j$. Applying the procedure of Step \ref{step1} to the morphism $f_H$ yields a toroidal morphism $f':(U'\subset X')\to (V'\subset Y')$ of snc toroidal embeddings.
Each blow-up of this procedure preserves the property that a local equation of some codimension $1$ stratum through the image of $t_j$ induces a uniformizer of $R_j$ (in the local coordinates chosen in Step \ref{step1}, this local equation can be chosen to be $z_{i_0}$). It follows that $f'$ has the required properties.
\end{proof}

\begin{Step}
The morphism $f'$ is smooth at  $\phi'_j(t_j)$, where $\phi_j':T_j\to X'$ is the lift of $\phi_j$.
\end{Step}

\begin{proof}
Since the question is of geometric nature, we may assume that $k=\overline{k}$.
Since the question is local, we may assume that $f'$ is a toric morphism of smooth affine toric varieties. By the classification of smooth affine toric varieties \cite[Proposition p.~29]{ToricFulton}, one can choose toric coordinates $x_1,\dots ,x_r$ on $X'$ and $y_1,\dots,y_s$ on $Y'$ such that $X'=\Spec(k[x_1,\dots,x_p,x_{p+1}^{\pm 1},\dots,x_{r}^{\pm 1}])$ and $Y'=\Spec(k[y_1,\dots y_q,y_{q+1}^{\pm 1},\dots,y_{s}^{\pm 1}])$, and such that $(f')^*y_b=\prod_a x_a^{m_{a,b}}$ for some matrix $M=(m_{a,b})_{1\leq a\leq r, 1\leq b\leq s}\in M_{r,s}(\Z)$.

After permuting $y_1,\dots,y_q$, we may assume that the stratum $D'$ of Step \ref{step2} has equation $\{y_1=0\}$. It follows that the valuation of $(\psi'_j)^*y_b$ is equal to $1$ if $b=1$ and to $0$ if $b\geq 2$. In view of the formula $(f')^*y_b=\prod_a x_a^{m_{a,b}}$, we deduce the existence of $1\leq a_0\leq r$ such that $m_{a_0,1}\neq 0$ and such that the valuation of $(\phi'_j)^*x_{a_0}$ is non-zero, hence positive.
As $f'$ is dominant by assumption, the matrix $M$ has rank $s$. We can therefore choose a subset $A\subset  \{1,\dots,r\}$ of cardinality $s$ containing $a_0$ such that $\det(m_{a,b})_{a\in A,1\leq b\leq s}\in\Z$ is non-zero.

The Jacobian matrix of $f'$ at $(x_1,\dots,x_r)$ is $(\frac{\partial y_b}{\partial x_a})=(\frac{m_{a,b}\cdot (f')^*y_b}{x_a})$. Its maximal minor associated with $A$ is $\Delta:=\prod_{1\leq b\leq s}(f')^*y_b\cdot\prod_{a \in A}x_a^{-1}\cdot\det(m_{a,b})_{a\in A,1\leq b\leq s}$. By the above choices, the valuation of the non-zero element $(\phi'_j)^*\Delta\in R_j$ is non-positive, hence zero. This shows that $\Delta$ is invertible at $\phi'_j(t_j)$ so that $f'$ is smooth at $\phi'_j(t_j)$.
\end{proof}

\subsection{Proof of the fibration theorem}
\label{parfib}

\begin{proof}[Proof of Theorem \ref{up}]
Choose proper regular models $f: \kX\to B$ and $f': \kX'\to B$ of~$X$ and $X'$ over $B$ such that $g$ extends to a morphism $g: \kX\to\kX'$.
By Theorem \ref{birinv} and \cite[Theorem 1.1]{ADK}, we may assume that $g$ underlies a toroidal morphism $g:(\kU\subset \kX)\to(\kU'\subset\kX')$ of snc toroidal embeddings and that~$\kX$ and~$\kX'$ are projective.

As in Proposition \ref{deftightK}, let $B(\R)\subset K\subset\Omega\subset B(\C)$, with $K$ compact and $G$-stable and $\Omega$ open and $G$-stable, let $b_1,\dots, b_m \in K$, let $r\geq 0$, and let $u: \Omega\to\kX(\C)$ be a $G$-equivariant holomorphic section of $f(\C)$.
Our goal is to construct a sequence $s_n:B\to\kX$ of sections of $f$ satisfying the conditions appearing in Proposition~\ref{deftightK}, \emph{i.e.}\ such
that the $s_n$ have the same $r$-jets as $u$ at the $b_i$ and $s_n(\C)|_K$ converges uniformly to $u|_K$.
 By Proposition \ref{tubularR} (ii), we may assume, after shrinking $\Omega$ and perturbating~$u$, that all the connected components of $u(\Omega)$ meet $\kU(\C)$.
Since $K$ is compact, we may assume, after shrinking~$\Omega$ again, that the set $\Sigma:=\{x\in\Omega\mid u(x)\notin\kU(\C)\}$ is finite. Applying Proposition \ref{toroidaltheorem} to the toroidal morphism $g$ and to the discrete valuation rings $\sO_{\Omega,x}$ for $x\in\Sigma$ shows the existence of regular modifications $\pi:\kZ\to\kX$ and $\pi':\kZ'\to\kX'$ and of a morphism $h:\kZ\to\kZ'$ such that $g\circ\pi=\pi'\circ h$, such that~$\pi$ (resp.~$\pi'$) is an isomorphism above $\kU$ (resp.\ above $\kU'$) and such that the strict transform $v:\Omega\to\kZ(\C)$ of $u$ has the property that $h(\C)$ is submersive at $v(x)$ for all $x\in \Sigma$. Since $g$ was already smooth  along $\kU$, the latter property in fact holds for all $x\in\Omega$.

By Theorem \ref{birinv}, we may replace $g:\kX\to\kX'$ and $u$ with $h:\kZ\to \kZ'$ and $v$ and thus assume that $g(\C)$ is submersive along $u(\Omega)$. We will not consider any toroidal structure on $g$ anymore.
Let $(X')^0 \subset X'$ be a dense open subset such that
the tight approximation property holds for $X_x=g^{-1}(x)$ for all
$x \in (X')^0(F)$.
Extend it to an open subset $(\kX')^0\subset\kX'$. By Proposition \ref{tubularR} (ii), we may assume, after shrinking $\Omega$ and perturbating $u$, that all the connected components of $u(\Omega)$ meet
$g^{-1}((\kX')^0)(\C)$. Define $u':=g(\C)\circ u: \Omega\to \kX'(\C)$. By Proposition \ref{tubularR}~(iii), there exist a $G$\nobreakdash-stable open neighbourhood $W$ of $u'(\Omega)$ in $\kX'(\C)$ and a $G$-equivariant holomorphic map $w:W\to \kX(\C)$ such that $g(\C)\circ w=\Id_W$ and $w\circ u'=u$.
Let $\Omega'\subset \Omega$ be a $G$-stable open neighbourhood of $K$ whose closure $K'$ in $\Omega$ is compact.
Applying the tight approximation property of $f'$ yields a sequence $s'_n: B\to\kX'$ of sections of $f'$ with the same $r$-jets as $u'$ at the $b_i$ such that $s'_n(\C)|_{K'}$ converges uniformly to~$u'|_{K'}$.
Replacing $\Omega$ with $\Omega'$, we can ensure that $s_n'(\C)(\Omega)\subset W$ for $n\gg 0$.
Define $v_n:=w\circ s'_n(\C)|_\Omega:\Omega\to \kX(\C)$.
For $n\gg0$, consider the commutative diagram
\begin{equation*}
\begin{aligned}
\xymatrix
{
\widetilde{\kY}_n\ar^{p_n}@/^01.0pc/[rr]\ar[r]\ar_{g_n}[rd]&\kY_n\ar[r]\ar
[d]&\kX\ar^g[d]  \\
&B\ar^{s_n'}[r]&\kX'\rlap{,}
}
\end{aligned}
\end{equation*}
whose square is cartesian and where $\widetilde{\kY}_n\to \kY_n$ is a resolution of singularities that is an isomorphism above the regular locus of $\kY_n$. The map $v_n: \Omega\to \kX(\C)$ induces a map $\Omega\to\kY_n(\C)$ whose strict transform $\tilde{v}_n: \Omega\to\widetilde{\kY}_n(\C)$ is a $G$-equivariant holomorphic section of $g_n(\C)$ over $\Omega$. Since we have ensured that no connected component of $u(\Omega)$ avoids $g^{-1}((\kX')^0)(\C)$, the morphism $g_n$ satisfies the tight approximation property (for $n\gg 0$) by our hypothesis on $(X')^0$ and by Theorem \ref{birinv}.

Applying the tight approximation property of $g_n$ yields a sequence $t_{n,k}: B\to\widetilde{\kY}_n$ of sections of $g_n$ with the same $r$-jets as $\tilde{v}_n$ at the $b_i$ such that $t_{n,k}(\C)|_K$ converges uniformly to $\tilde{v}_n|_K$. Then, if $k(n)$ is a well chosen sequence of integers, the sections $s_n:=p_n\circ t_{n,k(n)}$ of $f$ have the same $r$-jets as $u$ at the $b_i$, and the sequence $s_n(\C)|_K$ converges uniformly to $u|_K$. The proof is now complete.
\end{proof}

\subsection{Applications}
\label{applfib}

Using Theorem \ref{up}, Example \ref{quadrique} and Corollary \ref{rat}, we get:
\begin{prop}
\label{fibrationenquadriques}
A smooth variety over $F$ that is birational to an iterated fibration into quadrics satisfies the tight approximation property.
\end{prop}

\begin{example}
\label{cubique}
For $n \geq 3$, a smooth cubic hypersurface $X\subset\P^n_F$ containing a line $L\subset X$ over~$F$ satisfies the tight approximation property. Indeed, projecting from $L$ induces a rational conic bundle structure $X\dashrightarrow\P^{n-2}_F$, and 
Proposition~\ref{fibrationenquadriques} applies.

In particular,
if $X$ is a smooth cubic hypersurface of dimension~$\geq 2$ over $\R$, the variety $X_F$ satisfies the tight approximation property (as it always contains a line defined over $\R$, see \cite[Proof of Theorem 9.23]{BWII}). This yields a positive answer to Question \ref{q:ciapprox} for these varieties.

We do not know whether all smooth
cubic surfaces over~$F$ satisfy the tight approximation property, even if~$B$ is a complex curve.
See Theorem~\ref{BHcubic} for a weaker property that smooth cubic hypersurfaces
of dimension~$\geq 2$ over~$F$ do enjoy.
\end{example}

\begin{example}
\label{cubiquecomplexe}
Arguing exactly as in \cite[\S 1]{HTlowdegree}, but using Theorem \ref{up} instead of \cite[Proposition 2.2]{CTG}, shows that if $B$ is a complex curve and $\phivar:\N\to\N$ is defined as in \cite[p.\ 938]{HTlowdegree}, then smooth hypersurfaces of degree $d$ in $\P^n_F$ with $n\geq\phivar(d)$ satisfy the tight approximation property. This applies in particular to smooth cubic hypersurfaces of dimension $\geq 5$. We note that the weak approximation property has been shown to hold for smooth cubic hypersurfaces of dimension $\geq 2$ over function fields of complex curves by Tian \cite[Theorem 1.2]{tiancubic} (the crucial case being that of cubic surfaces).
\end{example}

\begin{example}
\label{ic2c}
For $n \geq 4$, a smooth complete intersection of two quadrics $X\subset\P^n_F$ that has a rational point $x\in X(F)$ satisfies the tight approximation property. Indeed, the pencil of hyperplane sections of $X$ that contain~$x$ and are singular at $x$ induces a rational quadric bundle structure $X\dashrightarrow \P^1_F$ (see \cite[Theorem~3.2]{CTSSD}), and one can apply Proposition \ref{fibrationenquadriques}.
The hypothesis that $X$ has a rational point is always satisfied if $B(\R)=\varnothing$.
In general this results from \cite[Theorem~0.13]{periodindex} applied to a two-dimensional hyperplane section of $X$;
when $n\geq 6$,
one can also split off three hyperbolic planes from a quadratic form vanishing on~$X$
to find a conic, and hence an $F$\nobreakdash-point, on~$X$
(see \cite[Satz 22]{Witt}).

If $n\geq 4$ and $X\subset\P^n_{\R}$ is a smooth complete intersection of two quadrics over $\R$,
the above argument
applied to $X_F$ yields a positive answer to Question \ref{q:ciapprox} for the variety $X$.
Indeed $X(F)\neq\varnothing$ if $X(\R)\neq\varnothing$ (obvious) or if $B(\R)=\varnothing$ (as there exist morphisms from $B$ to the anisotropic real conic $\Gamma$ by \cite[Satz 22]{Witt} and from $\Gamma$ to~$X$ by \cite[discussion below Teorema VI, p.~60]{comessatti}, see also \cite{colliotsanspoint}).
\end{example}

Theorem \ref{up} also has the following consequence:

\begin{prop}
\label{products}
If two smooth varieties $X$ and $X'$ over $F$ satisfy the tight approximation property, then so does their product $X\times X'$.
\end{prop}

\section{The descent method}

The goal of this section is to prove Theorem \ref{descent}, which is a descent theorem for the tight approximation property. In Section \ref{sechomo}, we shall use it in combination with the fibration technique of Section \ref{secfib} to study homogeneous spaces of linear groups.

\begin{thm}
\label{descent}
Let $X$ be a smooth variety over $F$, and let $S$ be a linear algebraic group over $F$. Let $Q\to X$ be a left $S$-torsor over $X$. Assume that any twist of $Q$ by a right $S$-torsor over $F$ satisfies the tight approximation property. Then $X$ satisfies the tight approximation property.
\end{thm}

In Theorem~\ref{descent}, the algebraic group~$S$ is not assumed to be connected.

We refer to \S\ref{torsors} for generalities about torsors. Proposition \ref{liftH1}, proven in \S\ref{torsorsreal}, is the heart of the proof of Theorem \ref{descent} given in  \S\ref{proofdescent}.

\subsection{Torsors}
\label{torsors}

Let $S$ be a smooth affine group scheme over a scheme~$X$. The set of isomorphism classes of right $S$-torsors on $X$ (in the sense of \cite[Definition~2.2.1]{Skorobogatov})
 is in bijection with the \'etale \v{C}ech cohomology pointed set $\check{H}_{\et}^1(X,S)$ (see \cite[Ch.~III, Proposition~4.6 and Remark~4.8~(a)]{Milne}).
When $X=\Spec(k)$ is the spectrum of a field $k$, this pointed set can be identified with the non-abelian Galois cohomology pointed set $H^1(k,S)$ defined in \cite[I~5.1]{CohoGalois}.

If $Y$ is an affine $X$-scheme equipped with a left $S$-action and $P$ is a right $S$-torsor on $X$, we can define the \textit{twist} ${}_{P}Y$ of $Y$ by $P$ to be the quotient of $P\times_X Y$ by the diagonal action $s\cdot(p,y)=(p\cdot s^{-1},s\cdot y)$ of $S$ \cite[Lemma 2.2.3]{Skorobogatov}. Letting $S$ act on itself by conjugation yields the \textit{inner form} ${}_{P}S$ of $S$ \cite[Example 1 p.~10]{Skorobogatov}. The right $S$-torsor $P$ has a natural and compatible structure of left ${}_{P}S$-torsor.

If $Q$ is a left $S$-torsor on $X$, we let $Q^{-1}$ be the right $S$-torsor on $X$ that is isomorphic to $Q$ as an $X$-scheme, with action given by $q\cdot_{Q^{-1}} s:=s^{-1}\cdot_{Q} q$: it is the \textit{inverse torsor} of $Q$. Then $Q^{-1}$ (resp.\ $Q$) is also naturally a left (resp.\ right) ${}_{Q^{-1}}S$-torsor. It follows that the \textit{composition} $P\circ Q:={}_{P}Q$ is a right ${}_{Q^{-1}}S$-torsor
and a left ${}_{P}S$-torsor
and that $Q^{-1}\circ Q$ is a trivial ${}_{Q^{-1}}S$-torsor (see \cite[Example 2 p.~10]{Skorobogatov} for more details).

Let $k$ be a field, let $S\subset H$ be a closed immersion of smooth affine group schemes over $X=\Spec(k)$, and let $S$ act on $H$ by left multiplication. This action admits a quotient scheme $H\to S\backslash H$ endowing $H$ with the structure of a left $S$-torsor over~$S\backslash H$ \cite[Expos\'e~$\textrm{VI}_{\textrm{A}}$, Th\'eor\`eme~3.3.2]{SGA31}.
If $P$ is a right $S$-torsor over $k$, one can consider the twist ${}_{P}H$ of $H$ by $P$. Right multiplication of $H$ on ${}_{P}H$ endows ${}_{P}H$ with the structure of a right $H$-torsor over $k$.

We will make use of topological analogues of the above notions. A \textit{finite group} over a topological space $T$ is a finite covering $q:\sS\to T$ of $T$ equipped with a continuous section and with a continuous map $\sS\times_T\sS\to \sS$ endowing the fibers of $q$ with a group structure. A right $\sS$\nobreakdash-torsor $p:\sP\to T$ is a surjective finite covering of $T$ endowed with a continuous map $\sP\times_T\sS\to \sP$ inducing a simply transitive right action of the fibers of $q$ on the fibers of $p$.
If $T$ is endowed with an action of $G$, we define $G$-equivariant finite groups $\sS$ over $T$ and $G$-equivariant right $\sS$-torsors $\sP$ by requiring that $\sS$ and $\sP$ be endowed with actions of $G$ for which all the above maps are $G$-equivariant.

\subsection{Torsors over real function fields}
\label{torsorsreal}

An open subset $\Omega\subset B(\C)$ is a \textit{domain with} $\ci$ \textit{boundary} if
it is the interior of a closed $\ci$ submanifold with boundary, say~$\overline\Omega$, of~$B(\C)$.
It has finitely many connected components by compactness of $B(\C)$.
We denote its boundary by $\partial\Omega = \overline\Omega \setminus \Omega$.

\begin{lem}
\label{lemmeabordlisse}
Let $K\subset B(\C)$ be a $G$-stable compact subset, and let $\Omega$ be a $G$-stable open neighbourhood of $K$ in $B(\C)$. Then there exists a $G$-stable domain with $\ci$ boundary $\Omega'\subset B(\C)$ such that $K\subset\Omega'$ and $\overline{\Omega'}\subset\Omega$.
\end{lem}

\begin{proof}
Using partitions of unity, we find a $\ci$ map $f:B(\C)\to[0,1]$ that is equal to $0$ on $K$ and to $1$ on $B(\C)\setminus\Omega$. Replacing $f$ with $x\mapsto (f(x)+f(\sigma(x)))/2$, we may assume that $f$ is $G$-invariant. By Sard's theorem, the map $f$ has a regular value $\varepsilon\in (0,1)$, and one can choose $\Omega'=f^{-1}([0,\varepsilon))$.
\end{proof}

We refer to \S\S\ref{meropar}--\ref{merodim}
for a study of the ring $\sM(\Omega)^G$ of $G$-equivariant meromorphic functions on a $G$-equivariant complex manifold $\Omega$ of dimension $1$.
Letting $(\Omega_i)_{i\in I}$ be the $G$-orbits of connected components of $\Omega$, one has $\sM(\Omega)^G=\prod_{i\in I}\sM(\Omega_i)^G$. If~$\Omega$ has finitely many connected components and if $S$ is a linear algebraic group over $\sM(\Omega)^G$, the pointed set $H^1(\sM(\Omega)^G,S)$ is the product $\prod_{i\in I} H^1(\sM(\Omega_i)^G,S)$.

The next proposition was inspired by~\cite[Th\'eor\`eme~4.2]{CTG}.

\begin{prop}
\label{liftH1}
Let $\Omega\subset B(\C)$ be a $G$-stable domain with $\ci$ boundary
and let $K\subset \Omega$ be a $G$-stable compact subset containing $B(\R)$.
 Let $S$ be a linear algebraic group over $F$, and let $\alpha\in H^1(\sM(\Omega)^G,S)$.
Then there exists a $G$-stable domain with $\ci$ boundary $\Omega'$
such that  $K\subset\Omega'\subset\Omega$ and such that $\alpha|_{\Omega'}\in
H^1(\sM(\Omega')^G,S)$ is in the image of the restriction map $H^1(F,S)\to  H^1(\sM(\Omega')^G,S)$.
\end{prop}

\begin{lem}
\label{finitecase}
Proposition \ref{liftH1} holds if $S$ is finite.
\end{lem}

\begin{proof}
The algebraic group $S$ extends to a finite \'etale group scheme $\kS$ over a dense open subset $B^0\subset B$. Then $\sS:=\kS(\C)$ is a $G$-equivariant finite group over $B^0(\C)$ in the sense of \S \ref{torsors}. Let $P$ be a right $S$-torsor over $\sM(\Omega)^G$ of class $\alpha$. By Proposition~\ref{Gramifiedcover},
$P$ corresponds to a $G$-equivariant finite ramified covering $p:\sP\to\Omega$ of $\Omega$ such that, denoting by $R\subset \Omega$ the ramification locus of $p$, $\sP|_{B^0(\C)\cap(\Omega\setminus R)}$ is a $G$-equivariant right $\sS|_{B^0(\C)\cap(\Omega\setminus R)}$\nobreakdash-torsor.
%par le lien entre produit tensoriel et produit fibré "normalisé" dans la correspondance de Proposition \ref{Gramifiedcover}.

 Since $R$ is a discrete subset of~$\Omega$, we may assume, after shrinking $\Omega$ using Lemma \ref{lemmeabordlisse}, that $R$ is finite and that $p|_{B^0(\C)\cap(\Omega\setminus R)}$ extends to a $G$-equivariant right $\sS|_{B^0(\C)\cap(\overline{\Omega}\setminus R)}$\nobreakdash-torsor $\overline{p}:\overline{\sP}\to B^0(\C)\cap(\overline{\Omega}\setminus R)$. Shrinking $B^0$ ensures that  $B^0(\C)\cap R=\varnothing$.
Since $B(\R)\cap\partial{\Omega}=\varnothing$, we may assume after further shrinking $B^0$ that the connected components of $B^0(\C)\cap\partial{\Omega}$ are contractible and not $G$\nobreakdash-stable. It follows that the $G$\nobreakdash-equivariant right $\sS|_{B^0(\C)\cap\partial{\Omega}}$\nobreakdash-torsor $\overline{p}|_{B^0(\C)\cap\partial{\Omega}}$ can be $G$-equivariantly trivialized. Gluing $\overline{p}$ with the trivial $G$-equivariant right $\sS|_{B^0(\C)\setminus (B^0(\C)\cap\Omega)}$-torsor using such trivializations yields a $G$-equivariant right $\sS$-torsor $p':\sP'\to B^0(\C)$ such that $p'|_{B^0(\C)\cap\Omega}$ is isomorphic to $p$.

By the $G$-equivariant analogue of Riemann's existence theorem (see \S\ref{Ranal}), $\sP'$ is the analytification of a right $\kS$-torsor $\kP'$ over $B^0$. The class in $H^1(F,S)$ of the generic fiber of $\kP'$ has image $\alpha$ in $H^1(\sM(\Omega)^G,S)$, proving the lemma.
\end{proof}

To prove Proposition \ref{liftH1} in general, we will use deep results of Scheiderer \cite{Scheiderer}. We first introduce some notation. Let $\Sper(k)$ be the real spectrum of a field $k$ (denoted by $\Omega_k$ in \cite{Scheiderer} and by $X_k$ in \cite{Lam}). It is the set of orderings of the field $k$, endowed with the Harrison topology \cite[VIII, \S 6]{Lam}. With an ordering $\xi\in \Sper(k)$ is associated a real closure $k_{\xi}$ of $k$.
If $S$ is an algebraic group over~$k$, Scheiderer defines a sheaf
of pointed sets $\sH^1(S)$ on $\Sper(k)$ whose stalk at $\xi$ is $H^1(k_{\xi},S)$, and a canonical map $h_S:H^1(k,S)\to H^0(\Sper(k),\sH^1(S))$ whose stalks are the restriction maps in Galois cohomology \cite[Definition 2.8, Proposition~2.9]{Scheiderer}.

These definitions extend at once to finite products of fields, as the real spectrum of such a product is the disjoint union of the real spectra of the factors.

\begin{proof}[Proof of Proposition \ref{liftH1}]
Denote by $S^0$ be the connected component of the identity of $S$ and consider the
commutative diagram of cohomology pointed sets
\begin{equation}
\label{diagH1s}
\begin{aligned}
\xymatrix
@R=0.4cm
{
H^1(F,S)\ar^{}[d] \ar[r]& H^1(F,S/S^0)\ar[d] \\
H^1(\sM(\Omega)^G,S)\ar[r]&H^1(\sM(\Omega)^G,S/S^0),
}
\end{aligned}
\end{equation}
whose vertical arrows are the restriction maps, and where we still denote by $S$ and~$S^0$ the extensions of scalars of $S$ and $S^0$ from $F$ to  $\sM(\Omega)^G$.

Applying Lemma \ref{finitecase} to $S/S^0$, we may assume, after shrinking $\Omega$, that there exists a class $\beta\in H^1(F,S/S^0)$ such that the images of $\alpha$ and $\beta$ in $H^1(\sM(\Omega)^G,S/S^0)$ in the diagram (\ref{diagH1s}) coincide.
Since $B(\R)\subset \Omega$ is compact, Proposition \ref{spectrereel} shows that the natural map $\phi:\Sper(\sM(\Omega)^G)\to\Sper(F)$ is bijective. Since $\phi$ is moreover continuous and closed \cite[Corollary~p.~272]{Lam}, it is a homeomorphism. By \cite[Proposition 1.2]{Scheiderer}, the natural map $\phi^*\sH^1(S)\to \sH^1(S)$ is an isomorphism at the level of stalks, hence an isomorphism of sheaves on $\Sper(\sM(\Omega)^G)$. We deduce an isomorphism of pointed sets
\begin{equation}
\label{H0Sper}
\phi^*:H^0(\Sper(F),\sH^1(S))\isoto H^0(\Sper(\sM(\Omega)^G),\sH^1(S)).
\end{equation}
Let $\gamma\in H^0(\Sper(F),\sH^1(S))$ be the inverse image of $h_S(\alpha)$ by (\ref{H0Sper}). Using \cite[Corollary 6.6]{Scheiderer}, choose a class $\delta\in H^1(F,S)$ inducing $\beta\in H^1(F,S/S^0)$ in (\ref{diagH1s}) and such that $h_S(\delta)=\gamma$.
Let $\varepsilon\in H^1(\sM(\Omega)^G,S)$ be the class induced by $\delta$ in~(\ref{diagH1s}).

Let $P$ be a right $S$-torsor over $F$ of class $\delta$. Letting $S$ act by conjugation on the exact sequence $1\to S^0\to S\to S/S^0\to 1$ yields a short exact sequence $1\to {}_PS^0\to {}_PS\to {}_P(S/S^0)\to 1$.
We consider the compatible bijections $\tau_P: H^1(F,{}_P S)\to H^1(F,S)$ and $\sigma_P: H^1(\sM(\Omega)^G,{}_P S) \to H^1(\sM(\Omega)^G,S)$ with $\tau_P(0)=\delta$ and $\sigma_P(0)=\varepsilon$ described in \cite[I 5.3, Proposition 35]{CohoGalois}.
Since $\alpha$ and~$\varepsilon$ have the same image in $H^1(\sM(\Omega)^G,S/S^0)$ in diagram (\ref{diagH1s}), \cite[I 5.4]{CohoGalois} shows that $\sigma_P^{-1}(\alpha)$ has the same image as $\sigma_P^{-1}(\varepsilon)=0$ in $H^1(\sM(\Omega)^G,{}_P(S/S^0))$.
We deduce from the natural commutative diagram of pointed sets with exact rows
\begin{equation}
\label{autrediagH1s}
\begin{aligned}
\xymatrix
@R=0.4cm
{
H^1(F,{}_PS^0)\ar[r]\ar^{}[d]&
H^1(F,{}_PS)\ar^{}[d] \ar[r]& H^1(F,{}_P(S/S^0))\ar[d] \\
H^1(\sM(\Omega)^G,{}_PS^0)\ar[r]&
H^1(\sM(\Omega)^G,{}_PS)\ar[r]&H^1(\sM(\Omega)^G,{}_P(S/S^0))
}
\end{aligned}
\end{equation}
the existence of $\zeta\in H^1(\sM(\Omega)^G,{}_PS^0)$ lifting $\sigma_P^{-1}(\alpha)$.
In the commutative diagram
\begin{equation}
\label{Sch4.1}
\begin{aligned}
\xymatrix
@R=0.4cm
@C=1.3cm
{
H^1(F,{}_PS^0)\ar^{}[d] \ar^(.4){h_{{}_P\hspace{-.05em}S^0}}[r]& H^0(\Sper(F),\sH^1({}_PS^0))\ar[d] \\
H^1(\sM(\Omega)^G,{}_PS^0)\ar^(.4){h_{{}_P\hspace{-.05em}S^0}}[r]&H^0(\Sper(\sM(\Omega)^G),\sH^1({}_PS^0)),
}
\end{aligned}
\end{equation}
the horizontal maps are bijective by \cite[Theorem 4.1]{Scheiderer} and Corollary \ref{cohodimvirt1},  and the right vertical map is bijective because $\phi$ is a homeomorphism and the natural map $\phi^*\sH^1({}_PS^0)\to \sH^1({}_PS^0)$ is an isomorphism by \cite[Proposition 1.2]{Scheiderer}. The left vertical arrow of (\ref{Sch4.1}), hence of (\ref{autrediagH1s}), is thus bijective. Let $\xi$ be the image in $H^1(F,{}_PS)$ of the inverse image of $\zeta$ by this arrow. The commutativity of (\ref{autrediagH1s}) shows that $\xi\in H^1(F,{}_PS)$ has image
 $\sigma_P^{-1}(\alpha)$ in $H^1(\sM(\Omega)^G,{}_P S)$. Consequently, the class $\tau_P(\xi)\in H^1(F,S)$ has image $\alpha$ in $H^1(\sM(\Omega)^G,S)$, which concludes the proof.
\end{proof}

\subsection{Descent along torsors}
\label{proofdescent}
We now reach the goal of this section.

\begin{proof}[Proof of Theorem \ref{descent}]
Let $f:\kX\to B$ be a proper regular model of $X$ over $B$.
Extend $S$ to a smooth affine group scheme $\kS$ over a dense open subset $B^0\subset B$.
Let $\kX^0\subset f^{-1}(B^0)$ be a dense open subset over which $Q$ extends to a left $\kS_{\kX^0}$-torsor $\kQ$ and let $[\kQ^{-1}]\in \check{H}^1_{\et}(\kX^0,\kS_{\kX^0})$
be the class of  the inverse torsor $\kQ^{-1}$ (see~\S\ref{torsors}).

 Let $K$, $\Omega$, $b_i$, $r$ and $u$ be as in Proposition \ref{deftightK}. Shrinking $\Omega$ and perturbating~$u$, we may assume that all the connected components of $u(\Omega)$ meet $\kX^0$ by Proposition~\ref{tubularR}~(ii) and that~$\Omega$ is a domain with $\ci$ boundary by Lemma \ref{lemmeabordlisse}.
Let $\alpha$ be the image of $[\kQ^{-1}]$ by the restriction map $\check{H}^1_{\et}(\kX^0,\kS_{\kX^0})\to H^1(\sM(\Omega)^G,S)$ induced by~$u$.
Proposition~\ref{liftH1} ensures, after further shrinking $\Omega$, that $\alpha$ is the image of a class $\beta\in H^1(F,S)$ by the restriction map $H^1(F,S)\to H^1(\sM(\Omega)^G,S)$.

Let $P$ be a right $S$-torsor over $F$ of class $\beta$.
After replacing $B^0$ with a dense open $U\subset B^{0}$ and $\kX^0$ with $\kX^0\cap f^{-1}(U)$, we extend $P$ to a right $\kS$-torsor $\kP$ over $B^0$. Consider the twist $\kP_{\kX^0}\circ \kQ$ of $\kQ$ by $\kP_{\kX^0}$ as in \S\ref{torsors}: it is a right ${}_{\kQ^{-1}}(\kS_{\kX^0})$\nobreakdash-torsor over $\kX^0$. By construction, its class is in the kernel of the restriction map $\check{H}^1_{\et}(\kX^0,{}_{\kQ^{-1}}(\kS_{\kX^0}))\to H^1(\sM(\Omega)^G,{}_{\alpha}S)$ induced by~$u$. This exactly means that there exists a $G$-equivariant meromorphic map $v:\Omega\dashrightarrow (\kP_{\kX^0}\circ \kQ)(\C)$ lifting~$u$.

Let $g:\kY\to B$ be a proper flat morphism such that $\kY$ is regular and contains $\kP_{\kX^0}\circ \kQ$ as a dense open subset, and such that there exists $h:\kY\to\kX$ with $g=f\circ h$.
By the valuative criterion of properness applied to the $(\sO_{\Omega,x})_{x\in\Omega}$ as in  \S\ref{conventions}, $v$ extends to a $G$-equivariant holomorphic map $v:\Omega\to\kY(\C)$ lifting $u$.  The tight approximation property of $g$ yields a sequence $t_n:B\to\kY$ of sections of $g$
with the same $r$-jets as $v$ at the $b_i$ and such that $t_n(\C)|_K$ converges uniformly to $v|_K$. The sections $s_n:=h\circ t_n$ of $f$ show that $f$ has the tight approximation property.
\end{proof}

\section{Homogeneous spaces}
\label{sechomo}
We now provide new examples of varieties satisfying the tight approximation property: homogeneous spaces of connected linear algebraic groups over $F$.

\subsection{Tight approximation for algebraic groups}

We start with the case of connected linear algebraic groups over $F$.

\begin{prop}
\label{torus}
Tori over $F$ satisfy the tight approximation property.
\end{prop}

\begin{proof}
Let $H$ be a torus over $F$. Writing the cocharacter lattice of $H$ as a quotient of a permutation $\Gal(\overline{F}/F)$-module yields an exact sequence $1\to S\to Q\to H\to 1$ of algebraic groups over $F$, where $Q$ is a quasi-trivial torus \cite[p.~187]{CTS}, that is, a product of Weil restriction of scalars of $\Gm$.

The torus $Q$ is a left $S$-torsor over $H$. Let $Q'$ be a twist of $Q$ by a right $S$\nobreakdash-torsor over~$F$. The variety $Q'$ is naturally a right $Q$-torsor over $F$ (see \S \ref{torsors}).
 Since $H^1(F,Q)=0$ by Shapiro's lemma and Hilbert's Theorem $90$, $Q'$ is isomorphic to~$Q$.
Since $Q$ is $F$-rational, as is any quasi-trivial torus \cite[p.~188]{CTS}, we deduce from Corollary \ref{rat} that $Q'$ satisfies the tight approximation property.

Theorem \ref{descent} now shows that $H$ satisfies the tight approximation property.
\end{proof}

\begin{prop}
\label{group}
Connected linear algebraic groups over $F$ satisfy the tight approximation property.
\end{prop}

\begin{proof}
By \cite[Corollaries 15.5 (ii) and 15.8]{Borel}, a connected linear algebraic group over $F$ is birational to the product of its unipotent radical, which is $F$-rational, and of a connected reductive group. By Proposition \ref{products} and Corollary \ref{rat}, we may therefore assume our group to be reductive.
By \cite[Expos\'e XIII, Th\'eor\`eme~3.1]{SGA32} and \cite[Expos\'e XIV, Th\'eor\`eme 6.1]{SGA32}, any connected reductive $F$-group is birational to a torus over an $F$-rational variety. By Theorem \ref{up} and Corollary \ref{rat}, we are reduced to the case of a torus, dealt with in Proposition \ref{torus}.
\end{proof}

\begin{rmk}
Proposition \ref{group} is new even for an algebraic group defined over $\R$.
Platonov \cite[p.~169]{Platonov} asks whether all connected linear algebraic groups over $\R$ are $\R$-rational. It is not even known whether they are always stably rational,
which would allow one to deduce Proposition~\ref{group} from Corollary \ref{rat}.
(See however \cite[Theorem~2]{Chernousov} for a partial result in this direction.)

For contrast, let us recall that homogeneous spaces of linear algebraic groups need not be stably rational,
even over~$\C$ (see \cite{saltmannoether}, \cite[Corollary~3.11]{noname}) and that over~$F$,
even tori are not stably rational in general (see, \emph{e.g.}, \cite[\textsection4.10]{voskresenski}).
\end{rmk}

\subsection{Tight approximation for homogeneous spaces}
We now come to the main theorem of this section:

\begin{thm}
\label{homogeneous}
Homogeneous spaces of connected linear algebraic groups over $F$ satisfy the tight approximation property.
\end{thm}

\begin{proof}
Let $f: \kX\to B$ be a proper regular model of such a (right) homogeneous space~$X$. In order to prove the tight approximation property, we may assume that there exists a $\ci$ section $u: B(\R)\to\kX(\R)$. It follows that $\kX_{\eta}$ has an $F_b$-point for every $b\in B(\R)$, hence that so does $X$,
by the inverse function theorem.
 The description of the orderings of $F$ (apply Proposition \ref{spectrereel} with $Z=B(\C)$)
shows that~$X$ has a point in every real closure of~$F$. By Scheiderer's Hasse principle \cite[Corollary 6.2]{Scheiderer}, one therefore has $X(F)\neq\varnothing$.

If $X$ is a right torsor under a connected linear algebraic group over $F$, then it is isomorphic to this algebraic group since $X(F)\neq\varnothing$, and Proposition \ref{group} applies.
%, this solves the particular case of (right) torsors under connected linear algebraic groups over $F$.

In general, $X=S\backslash H$ for a connected linear algebraic group $H$ over $F$ and an algebraic subgroup $S\subset H$.
We consider the quotient map $g: H\to S\backslash H$ (see~\S\ref{torsors}). The variety $H$ is a left
 $S$-torsor over $S\backslash H$. Let $_{P}H\to S\backslash H$ be a twist of this $S$\nobreakdash-torsor by a
right $S$-torsor $P$ over $F$.
The variety $_{P}H$ is a right $H$-torsor over $F$ (see~\S\ref{torsors}),
 hence satisfies the tight approximation property.
Theorem \ref{descent} now shows that $X$ satisfies the tight approximation property.
\end{proof}

In practice, one may combine Theorem \ref{homogeneous} with Theorem \ref{up} to show the tight approximation property for smooth varieties over $F$ that are, birationally, iterated fibrations into homogeneous spaces of connected linear algebraic groups.
Another consequence of Theorem \ref{homogeneous} is:

\begin{cor}
Let $H$ be a (not necessarily connected) reductive group over $F$ acting linearly on $\A^d_F$ with trivial generic geometric stabilizer. Then the smooth locus of $\A^d_F/H$ satisfies the tight approximation property.
\end{cor}

\begin{proof}
Let $H\subset \GL_N$ be a closed embedding of algebraic groups.
By \cite[Corollary 3.11]{noname}, the smooth locus of $\A^d_F/H$ is stably birational to $\GL_N/H$. It thus satisfies the tight approximation property by Theorem \ref{homogeneous} and Proposition \ref{stableinv}.
\end{proof}

\subsection{No reciprocity obstruction}
\label{subsec:noreciprocity}

The reciprocity obstruction (see \S\ref{parapprox}) causes the weak
approximation property to fail for some smooth proper rationally connected
varieties over $F$, even for some with an $F$\nobreakdash-point.  The next
proposition, which we extract from Scheiderer's work \cite{Scheiderer},
shows, in view of Proposition~\ref{weakconnected}, that such a phenomenon
cannot occur for smooth compactifications of homogeneous spaces of
connected linear algebraic groups over~$F$.
For its later use in~\textsection\ref{subsec:wahom}, we state it over an
arbitrary real closed field~$R$; we recall that when $R=\R$, the connectedness of a
semi-algebraic space over~$R$ is equivalent to the connectedness of the
underlying set for the usual topology (see \cite[Theorem~2.4.5]{bcr}).

\begin{prop}
\label{connectedhomo}
Let $X$ be a homogeneous space of a connected linear algebraic group over a real closed field $R$ and let $Y$ be a smooth compactification of $X$. Then the semi-algebraic space $Y(R)$ is connected.
\end{prop}

\begin{proof}
We view $X(R)$, $Y(R)$ and $\P^1(R)$ as semi-algebraic spaces over~$R$.
Since $Y$ is smooth and connected, all the connected components of $Y(R)$ are Zariski dense
(by \cite[Proposition~3.3.10 and Corollary~2.9.8]{bcr}),
hence meet $X(R)$. It thus suffices to show that any two points $x_1,x_2\in X(R)$ belong to the same connected component of $Y(R)$.
For $i\in\{1,2\}$, the morphism $\Spec(R(t))\to X_{R(t)}$ with value $x_i$ induces a map $\Sper(R(t))\to (X_{R(t)})_r$ at the level of real spectra. This map induces, in turn, a section $\sigma_i\in H^0(\Sper(R(t)),\sC_{X_{R(t)}})$, where we use the notation of \cite[Theorem and Definition 2.1]{Scheiderer} for $k=R(t)$.
Consider the partition $\Sper(R(t))=Z_1\cup Z_2$ where $Z_1$ (resp.\ $Z_2$) is the open subset of orderings for which $t>0$ (resp.\ $t<0$). Let $\sigma\in H^0(\Sper(R(t)),\sC_{X_{R(t)}})$ be the section that coincides with $\sigma_i$ on $Z_i$. By \cite[Theorem 8.15]{Scheiderer}, the section $\sigma$ is induced by some $x\in X(R(t))\subset Y(R(t))$. The point $x$ gives rise to a morphism $f:\P^1_{R}\to Y$ by the valuative criterion of properness. By construction, $f(t)$ belongs to the same connected component of $Y(R)$ as $x_1$ (resp.\ as $x_2$) if $t>0$ (resp.\ if $t<0$). Since~$\P^1(R)$ is connected, $x_1$ and $x_2$ lie in the same connected component of $Y(R)$.
\end{proof}

Combining
Theorem~\ref{homogeneous},
Corollary~\ref{tightweak},
Proposition~\ref{weakconnected} and Proposition~\ref{connectedhomo}
yields the following result concerning the weak approximation property itself.

\begin{thm}
\label{weakhomoreal}
Any homogeneous space of a connected linear algebraic group over~$F$ satisfies the weak approximation property.
\end{thm}

Theorem~\ref{weakhomoreal} solves
a conjecture of Colliot-Th\'el\`ene \cite[p.~151]{CTgroupes},
who had dealt with the case of
connected linear algebraic groups themselves \cite[Th\'eor\`eme~2.1]{CTgroupes}.
This conjecture had previously been shown to hold in the
case where the stabilizers of the homogeneous space are connected,
by Scheiderer \cite[Corollary~6.2 and Theorem~8.9]{Scheiderer},
and in the case where~$B$ is a complex curve, by Colliot-Th\'el\`ene and Gille \cite[Th\'eor\`eme~4.3]{CTG}.

\subsection{Real closed ground fields}
\label{subsec:wahom}

Unlike the tight approximation property, the weak
approximation property makes sense for varieties defined over the
function field of a curve over an arbitrary real closed field.  We now adapt our arguments
to establish Colliot-Thélène's conjecture over such function fields (Theorem~\ref{weakhomo} below).

One of the ingredients in the proof of Theorem~\ref{homogeneous} we gave was
Riemann's existence theorem, a statement about complex algebraic curves,
first used in this context
by Colliot-Th\'el\`ene and
Gille \cite{CTG}.
Here we need a version of Riemann's existence theorem
for algebraic curves defined over
the algebraic closure of an arbitrary real closed field.
Such a version was originally proved by Huber
and was reproved, more recently, by Peterzil and Starchenko.  This is
discussed in~\textsection\ref{subsubsec:retrealclosed}.
In~\textsection\ref{subsubsec:existencetorsors} we establish Proposition~\ref{prop:horrible},
which will serve as a substitute for Proposition~\ref{liftH1}.
In~\textsection\ref{subsubsec:descentwastronghp} we
apply
Proposition~\ref{prop:horrible} to obtain a general descent theorem for the conjunction
of the weak approximation property and of the so-called strong Hasse principle
(Theorem~\ref{th:wadescent}).
Finally, we deduce weak approximation for arbitrary homogeneous spaces
from Scheiderer's work, by descent,
in~\textsection\ref{subsubsec:ctconjrealclosed}.

From now on and until the end of~\textsection\ref{subsec:wahom},
we fix a real closed field~$R$ and a smooth projective connected curve~$B$ over~$R$.
We denote by~$F$ the function field of~$B$ and by~$F_b$ the completion of~$F$ with respect
to a closed point $b \in B$.  Finally, we set $C=R(\sqrt{-1})$ and $G=\Gal(C/R)$.

\subsubsection{Riemann's existence theorem over algebraically closed fields of characteristic~$0$}
\label{subsubsec:retrealclosed}

By definition, a \emph{finite semi-algebraic covering} of a locally complete semi-algebraic space~$T$ over~$R$
is a locally complete
semi-algebraic space~$S$ equipped with a
semi-algebraic map $p:S \to T$ with finite
fibres, subject to the following condition: $T$ can be covered by finitely many open semi-algebraic subsets~$U$
such that $p^{-1}(U)$ is semi-algebraically isomorphic, over~$U$, to $U \times E$ for a finite
set~$E$.

If~$V$, $W$ are (algebraic) varieties over~$C$, we consider~$V(C)$, $W(C)$ as semi-algebraic spaces over~$R$.
We note that any finite \'etale covering $W\to V$ functorially induces
a finite semi-algebraic covering $W(C)\to V(C)$
(see \cite[Example~5.5]{delfsknebuschintrolocallysemialg}).
Riemann's existence theorem over~$C$ can now be stated as follows.

\newcommand{\citehuberknebusch}{\cite[Satz~12.12]{huberthesis}, \cite[Theorem~6.1]{huberknebuschglimpse}}
\begin{thm}[Huber \citehuberknebusch]
\label{th:huberknebuschpeterzilstarchenko}
Let~$V$ be a variety over~$C$.
The functor $(W\to V) \mapsto (W(C)\to V(C))$ from the category of finite \'etale coverings of~$V$
to the category of finite semi-algebraic coverings of~$V(C)$ is an equivalence of categories.
\end{thm}

\begin{rmks}
\label{rem:semialgpi1}
(i) Given a connected locally complete semi-algebraic space~$T$ over~$R$ and $t \in T$,
Delfs and Knebusch have defined the semi-algebraic fundamental group $\pi_1(T,t)$
and have shown that the ``fiber at~$t$'' functor induces an equivalence from the category
of finite semi-algebraic coverings of~$T$ to the category of finite sets endowed with an action
of~$\pi_1(T,t)$
(see \cite[Theorems~5.9, 5.10, 5.11]{delfsknebuschintrolocallysemialg}).
Thus, Theorem~\ref{th:huberknebuschpeterzilstarchenko}
can be reformulated, when~$V$ is connected, by saying that
for any $v \in V(C)$,
the \'etale fundamental group $\pi_1^\et(V,v)$ is naturally isomorphic
to the profinite completion
of the (finitely generated) semi-algebraic fundamental group $\pi_1(V(C),v)$.

(ii) A proof of Theorem~\ref{th:huberknebuschpeterzilstarchenko} (in a more general o-minimal setting) can also be deduced from the work of Peterzil and Starchenko~\cite{peterzilstarchenko}.
\end{rmks}

The precise statements that we shall need are Corollary~\ref{cor:huberknebuschpeterzilstarchenko} and
Remark~\ref{rmk:geqdict} below.  Before stating them we introduce a definition.

\begin{defn}
\label{def:semialgequivtorsor}
Given a variety~$V$ over~$R$,
a finite \'etale group scheme $\kS \to V$
and a semi-algebraic subset $T \subseteq V(C)$ stable under~$G$,
a \emph{semi-algebraic $G$\nobreakdash-equivariant right $\kS(C)$\nobreakdash-torsor
over~$T$} is by definition a finite semi-algebraic covering $p:S\to T$
endowed, on the one hand,
 with a semi-algebraic map
$m:S \times_{V(C)} \kS(C) \to S$ inducing a simply transitive right action of the
set-theoretic fibres of $\kS(C) \to V(C)$ above~$T$ on the set-theoretic fibres of~$p$,
and, on the other hand, with a semi-algebraic action of~$G$ on~$S$ such that~$p$
and~$m$ are $G$\nobreakdash-equivariant.
\end{defn}

\begin{cor}
\label{cor:huberknebuschpeterzilstarchenko}
Let~$\kS$ be a finite \'etale group scheme over
a variety~$V$ over~$R$.
The functor $W \mapsto W(C)$ from the category of right $\kS$\nobreakdash-torsors
over~$V$ to the category of semi-algebraic $G$\nobreakdash-equivariant right $\kS(C)$\nobreakdash-torsors
over~$V(C)$
is an equivalence of categories.
\end{cor}

\begin{proof}
Apply Theorem~\ref{th:huberknebuschpeterzilstarchenko} to~$V_C$; note that the two categories
that appear in its statement carry an action of~$G$;
on each side of the equivalence,
pass to the category of torsor objects under a fixed group object, all endowed with a $G$\nobreakdash-equivariant structure (\emph{i.e.}\ with an isomorphism between the object in question and its conjugate, such that the composition of the isomorphism with its conjugate is the identity).
\end{proof}

\begin{rmk}
\label{rmk:geqdict}
By the same argument, the equivalence of categories of Remark~\ref{rem:semialgpi1}~(i)
can be upgraded as follows:
given a variety~$V$ over~$R$,
a finite \'etale group scheme~$\kS$ over~$V$,
a connected semi-algebraic subset $T \subseteq V(C)$ stable under~$G$
and a point $t \in T$ fixed by~$G$, if we denote by~$\kS_t$
the fibre of~$\kS(C) \to V(C)$ above~$t$,
the ``fiber at~$t$'' functor induces an equivalence from the category
of semi-algebraic $G$\nobreakdash-equivariant right $\kS(C)$\nobreakdash-torsors over~$T$
to the category of finite $(\kS_t \rtimes \pi_1(T,t))$\nobreakdash-sets
on which the action of~$\kS_t$ is simply transitive, endowed with
 a $G$\nobreakdash-equivariant
structure; that is, to the category of finite $((\kS_t \rtimes \pi_1(T,t)) \rtimes G)$-sets on
which the action of~$\kS_t$ is simply transitive.
\end{rmk}

\subsubsection{Existence of torsors}
\label{subsubsec:existencetorsors}

We now establish Proposition~\ref{prop:horrible}, the key result in the proof of weak
approximation for homogeneous spaces of connected linear algebraic groups
over~$F$.
The maps $h_S$ that appear below
are those defined by Scheiderer in \cite[Proposition~2.9~(c)]{Scheiderer}.

Proposition~\ref{prop:horrible} plays the same role,
over real closed fields, as Proposition~\ref{liftH1} did over the field of real numbers.
Over real closed fields, we need to use
 semi-algebraic substitutes for the fields of meromorphic functions appearing
in Proposition~\ref{liftH1}.
This forces technicalities both into the proof and into the statement of Proposition~\ref{prop:horrible}.

\begin{prop}
\label{prop:horrible}
Let~$S$ be a linear algebraic group over~$F$.  Let $\Sigma \subset B$ be a finite closed subset.
The natural commutative square
\begin{align}
\begin{aligned}
\label{eq:commsquarelemma}
\xymatrix@R=3ex{
H^1(F,S) \ar[r] \ar[d]^{h_S} & \displaystyle\smash[b]{\prod_{b \in \Sigma}} H^1(F_b,S) \ar[d]^{h_S} \\
H^0(\Sper(F),\sH^1(S)) \ar[r] & \displaystyle\prod_{b \in \Sigma} H^0(\Sper(F_b),\sH^1(S))
}
\end{aligned}
\end{align}
induces a surjection of the set $H^1(F,S)$ onto the fiber product of
$\prod_{b \in \Sigma} H^1(F_b,S)$
and
$H^0(\Sper(F),\sH^1(S))$
over $\prod_{b \in \Sigma} H^0(\Sper(F_b),\sH^1(S))$.
\end{prop}

When~$S$ is connected,
Proposition~\ref{prop:horrible} contains nothing new.
Indeed, in this case, the two maps~$h_S$
are bijections
according to Scheiderer~\cite[Theorem~4.1]{Scheiderer}.

\begin{proof}
Let us fix $\alpha_0 \in H^0(\Sper(F),\sH^1(S))$ and $(\alpha_b)_{b \in \Sigma} \in \prod_{b \in \Sigma} H^1(F_b,S)$
having the same image in $\prod_{b \in \Sigma} H^0(\Sper(F_b),\sH^1(S))$ and prove that there exists
an element of $H^1(F,S)$ which induces both~$\alpha_0$ and $(\alpha_b)_{b \in \Sigma}$.

Let us first reduce ourselves to the case of a finite algebraic group~$S$.

To this end, let us denote by $S^0 \subseteq S$ the connected component of the identity and
by $\beta_0 \in H^0(\Sper(F),\sH^1(S/S^0))$
and $(\beta_b)_{b \in \Sigma} \in \prod_{b \in \Sigma} H^1(F_b,S/S^0)$ the images of~$\alpha_0$ and $(\alpha_b)_{b \in \Sigma}$.  Assuming that the proposition holds for finite algebraic groups,
there exists an element of $H^1(F,S/S^0)$ which induces both $\beta_0$ and $(\beta_b)_{b \in \Sigma}$.
According to \cite[Corollary~6.6]{Scheiderer}, this element can be lifted to an element
$\delta \in H^1(F,S)$ such that $h_S(\delta)=\alpha_0$.  Thus~$\delta$ simultaneously lifts~$\alpha_0$ and
$(\beta_b)_{b \in \Sigma}$.

Let~$P$ be a right $S$\nobreakdash-torsor over~$F$ whose isomorphism class is~$\delta$.
For any field~$L$ containing~$F$ and any $S' \in \{S, S/S^0\}$,
if ${}_PS'$ denotes the algebraic group over~$F$ obtained by twisting~$S'$ by~$P$ through
the action of~$S$ on~$S'$ by conjugation,
twisting right torsors by~$P$ produces bijections $\tau_P:H^1(L,{}_PS') \isoto H^1(L,S')$
and $\tau'_P:H^0(\Sper(L),\sH^1({}_PS')) \isoto H^0(\Sper(L),\sH^1(S'))$
such that $\tau_P(0)$ and $\tau'_P(0)$ coincide with the images of~$\delta$
and of $h_S(\delta)$, respectively
(see \cite[Ch.~I, \textsection5.3, Proposition~35]{CohoGalois}).
Thus, after replacing~$S$ with~${}_PS$, the classes $\alpha_0$, $\beta_0$
with $\tau_P'^{-1}(\alpha_0)$, $\tau_P'^{-1}(\beta_0)$,
and the classes
$\alpha_b$, $\beta_b$, $\delta$ with $\tau_P^{-1}(\alpha_b)$, $\tau_P^{-1}(\beta_b)$, $\tau_P^{-1}(\delta)$,
respectively,
we may assume that $\delta=0$, and hence that $\alpha_0=\beta_0=\beta_b=0$ for all $b \in \Sigma$.

For each $b \in \Sigma$, as $\beta_b=0$, there exists $\alpha_b^0 \in H^1(F_b,S^0)$
which induces~$\alpha_b$.  Let us fix such $\alpha_b^0$.
As $\alpha_0=0$, the hypothesis that~$\alpha_b$ and~$\alpha_0$ are compatible
implies that $h_S(\alpha_b)=0$.
In view of the commutative diagram of pointed sets with exact columns
\begin{align}
\label{cd:h0ss0h1s}
\begin{aligned}
\xymatrix@R=3ex{
H^0(\Sper(F), \sH^0(S/S^0)) \ar[d] \ar[r] & \displaystyle\smash[b]{\prod_{b \in \Sigma}} H^0(\Sper(F_b), \sH^0(S/S^0)) \ar[d] \\
\ar[d] H^0(\Sper(F), \sH^1(S^0)) \ar[r] & \displaystyle\smash[b]{\prod_{b \in \Sigma}} H^0(\Sper(F_b), \sH^1(S^0)) \ar[d] \\
H^0(\Sper(F), \sH^1(S)) \ar[r] &
\displaystyle\smash[b]{\prod_{b \in \Sigma}} H^0(\Sper(F_b), \sH^1(S))
\rlap{,}
}
\end{aligned}
\end{align}
whose vertical arrows are induced by the exact sequence of sheaves of pointed sets
$\sH^0(S/S^0) \to \sH^1(S^0) \to \sH^1(S)$
(\emph{loc.\ cit.}, Proposition~36) on the boolean spaces~$\Sper(F)$ and~$\Sper(F_b)$
(see \cite[(C.2)]{Scheiderer}),
it follows that $(h_{S^0}(\alpha_b^0))_{b\in \Sigma}$ belongs to the image of
the top right-hand side vertical arrow of~\eqref{cd:h0ss0h1s}.
On the other hand, the top horizontal arrow of~\eqref{cd:h0ss0h1s} is surjective
since $\coprod_{b \in \Sigma} \Sper(F_b)$ is a finite subset of the boolean space~$\Sper(F)$
(\emph{loc.\ cit.}).
Therefore $(h_{S^0}(\alpha_b^0))_{b\in\Sigma}$ comes from an element of the kernel of the
bottom left-hand side vertical arrow of~\eqref{cd:h0ss0h1s}.
By a theorem of Scheiderer \cite[Theorem~4.1]{Scheiderer},
the maps $h_{S^0}:H^1(F,S^0) \to H^0(\Sper(F),\sH^1(S^0))$
and $h_{S^0}:H^1(F_b,S^0) \to H^0(\Sper(F_b),\sH^1(S^0))$
for $b \in \Sigma$
are bijective.
We deduce that $(\alpha_b^0)_{b\in\Sigma}$ comes from an element of $H^1(F,S^0)$
whose image~$\xi$ in $H^1(F,S)$ satisfies $h_S(\xi)=0$.
Thus~$\xi$ induces both~$\alpha_0$ and~$(\alpha_b)_{b\in\Sigma}$, as desired.

We can therefore assume that the algebraic group~$S$ is finite.
Let us extend~$S$ to a finite \'etale group scheme~$\kS$ over a dense
open subset $B^0 \subset B$.

As the sheaf $\sH^1(S)$ on the compact space $\Sper(F)$ is
locally constant (see \cite[Theorem~2.13]{Scheiderer}),
we may cover~$\Sper(F)$ with finitely many constructible open subsets on which the sheaf~$\sH^1(S)$
and the global section~$\alpha_0$ are constant.
After shrinking~$B^0$, we may therefore assume that~$\sH^1(S)$ and~$\alpha_0$ are constant on the open subsets
of~$\Sper(F)$ given by the connected components of the semi-algebraic space $B^0(R)$.
After further shrinking~$B^0$, we may assume that~$B^0$ misses at least
two points of each
connected component of the semi-algebraic space~$B(R)$.
Finally, in order to prove the proposition,
we may enlarge~$\Sigma$, as \cite[Theorem~5.1]{Scheiderer} guarantees the existence of
suitable~$\alpha_b$ for the new points~$b$.
In particular, by enlarging~$\Sigma$ and shrinking~$B^0$, we may assume that $B^0=B\setminus \Sigma$.

Let $F_b^h$ denote the fraction field of the henselization of the local ring $\sO_{B,b}$.
As the inclusion $F_b^h \subset F_b$ induces an isomorphism between the absolute Galois groups of these
fields, the pull-back map $H^1(F_b^h,S) \to H^1(F_b,S)$ is a bijection.
Writing the algebraic extension $F_b^h/F$ as the union of its finite subextensions,
we thus find, for each $b \in \Sigma$, a smooth connected curve~$B'_b$,
an \'etale map $\pi_b:B'_b \to B$, an $F$\nobreakdash-linear embedding $R(B'_b)\hookrightarrow F_b^h$,
a rational point~$\sigma_b$ of the fibre $\pi_b^{-1}(b)$
and a class $\alpha'_b \in \check{H}^1_{\et}(B'_b \setminus \{\sigma_b\},\kS)$
whose image by the composition
of the pull-back maps
$\check{H}^1_{\et}(B'_b \setminus \{\sigma_b\},\kS) \to H^1(R(B'_b),S)\to H^1(F_b^h,S) \to H^1(F_b,S)$
coincides with~$\alpha_b$.

For $b \in \Sigma$, let us view $b$ and~$\sigma_b$ as subsets of~$B(C)$ and of~$B'_b(C)$
of cardinality~$1$ or~$2$ (depending on whether~$b$ is a rational point of~$B$ or a closed point of degree~$2$).
The semi-algebraic map $\pi_b(C): B'_b(C) \to B(C)$
induces an isomorphism from a semi-algebraic open neighbourhood of~$\sigma_b$ in $B'_b(C)$
to a semi-algebraic open neighbourhood~$\Omega_b$ of~$b$ in $B(C)$
(see \cite[Theorem~6.9]{DK2}, \cite[Example~5.1]{delfsknebuschintrolocallysemialg}).
After shrinking the~$\Omega_b$,
we may assume
that $\Omega_b$ is stable under~$G$
for each $b \in \Sigma$ and that the~$\Omega_b$ for $b \in \Sigma$ are pairwise disjoint.

Applying the triangulation theorem of Delfs and Knebusch~\cite[Theorem~2.1]{delfsknebuschonthehomology}
to the quotient
semi-algebraic space $B(C)/G$
(see \cite[Corollary~1.6]{brumfielquotient}),
we find a $G$\nobreakdash-equivariant triangulation of the semi-algebraic space~$B(C)$
in which $B^0(C)$, the points of~$\Sigma$, the $\Omega_b$ for $b \in \Sigma$
and the connected components of $B^0(R)$
are unions of simplices.
For $b \in \Sigma$, the star neighbourhood of~$b$ (see
\cite[p.~22]{delfsknebuschbook} for the definition we use) is then a semi-algebraic
open subset of~$B(C)$ contained in~$\Omega_b$ and stable under~$G$.  After shrinking~$\Omega_b$,
we may thus assume that~$\Omega_b$ coincides with the star neighbourhood of~$b$ for each $b \in \Sigma$;
in particular, the connected components of~$\Omega_b$ are now semi-algebraically contractible.
For each connected component~$e$ of~$B^0(R)$,
let $\Omega_e$ denote the star neighbourhood of~$e$.
This is a semi-algebraic open subset
of~$B(C)$ that is stable under~$G$.
After replacing the triangulation with its barycentric subdivision (and shrinking~$\Omega_b$
accordingly, so that it still denotes the star neighbourhood of~$b$), we may assume
that the triangulation is ``good on~$e$'', in the sense of
\cite[Definition~8]{delfsknebuschonthehomology}, for each
$e\in\pi_0(B^0(R))$ (\emph{loc.\ cit.}, Proposition~2.4),
that the~$\Omega_e$ for $e \in \pi_0(B^0(R))$ are pairwise disjoint and are
semi-algebraically contractible
(see \cite[Chapter~III, Proposition~1.6]{delfsknebuschbook}),
that the~$\Omega_b$ for $b \in \Sigma \cap (B(C)\setminus B(R))$ have two connected
components and
that for any $e \in \pi_0(B^0(R))$ and any $b \in \Sigma$,
the intersection $\Omega_b \cap \Omega_e$
is semi-algebraically contractible
if~$b$ is one of the two points of~$\Sigma$
that belong to the closure of~$e$ and is empty otherwise
(see \cite[Chapter~II, Lemma~9.7]{delfsknebuschbook}).
Let $\Omega_\Sigma = \bigcup_{b\in\Sigma}\Omega_b$
and $\Omega_0= \bigcup_{e \in \pi_0(B^0(R))} \Omega_e$.

The connected components of $\Omega = \Omega_\Sigma \cup \Omega_0$ can now be described as follows:
the pairs of components that are permuted by~$G$
are the star neighbourhoods of the non-rational points of~$\Sigma$
and the components that are fixed by~$G$
are the star neighbourhoods of the connected components of~$B(R)$.
After replacing once more the triangulation with its barycentric subdivision
(which in effect makes these star neighbourhoods smaller),
we may assume that the semi-algebraic closures in~$B(C)$ of the connected components of~$\Omega$
are pairwise disjoint.
The semi-algebraic boundary $\partial \Omega = \overline\Omega\setminus\Omega$ of~$\Omega$
then coincides with the semi-algebraic boundary of $B(C) \setminus \overline\Omega$.

For each $b \in \Sigma$, the element $\alpha'_b \in \check{H}^1_{\et}(B'_b \setminus \{\sigma_b\},\kS)$
is the isomorphism class of a right $\kS$\nobreakdash-torsor, say $\mathfrak q_b:\kQ_b \to B'_b \setminus \{\sigma_b\}$.
By \cite[Example~5.5]{delfsknebuschintrolocallysemialg},
this torsor induces a
semi-algebraic $G$\nobreakdash-equivariant right $\kS(C)$\nobreakdash-torsor $p_b:\sP_b\to \Omega_b \cap B^0(C)$.
Let $p_\Sigma:\sP_\Sigma \to \Omega_\Sigma \cap B^0(C)$ denote the disjoint union of the~$p_b$.

For a semi-algebraic subset~$U$ of~$B^0(R)$, let $\sF(U)$ denote the set of isomorphism classes
of semi-algebraic $G$\nobreakdash-equivariant right $\kS(C)$\nobreakdash-torsors over~$U$.
(We recall
from Definition~\ref{def:semialgequivtorsor}
that these are really torsors under the semi-algebraic group $U \times_{B^0(C)} \kS(C) \to U$.)
With the obvious restriction maps, this defines a presheaf~$\sF$ on the semi-algebraic space~$B^0(R)$.
If~$U$ is semi-algebraically contractible, the restriction map $\sF(U) \to \sF(\{u\})$
is a bijection for any $u \in U$,
by Remark~\ref{rmk:geqdict},
since~$U$ is then connected and simply connected
(in the sense that the semi-algebraic fundamental group $\pi_1(U,u)$ is trivial, see \cite[\textsection4]{delfsknebuschintrolocallysemialg}).
As the connected components of any
semi-algebraic open subset of~$B^0(R)$ are semi-algebraically contractible,
it follows that the presheaf~$\sF$ is in fact a sheaf and that its restriction to any
connected component of~$B^0(R)$ is constant.
Let $\sF|_{\Sper(F)}$ denote the sheaf obtained by restricting
to~$\Sper(F)$
the sheaf
corresponding to~$\sF$ on the ``abstraction'' of~$B^0(R)$
(in the sense of \cite[Appendix~A to Chapter~I]{delfsknebuschbook}).
As the natural map $\sH^1(S) \to \sF|_{\Sper(F)}$
(see \cite[\textsection2.3]{Scheiderer})
induces bijections between the stalks of these two sheaves on~$\Sper(F)$,
it is an isomorphism.  On the other hand, we recall that~$\alpha_0$ is constant on each connected component of~$B^0(R)$.
Thus~$\alpha_0$ belongs to the image of the resulting map
$H^0(B^0(R),\sF) \to H^0(\Sper(F),\sH^1(S))$.  In other words~$\alpha_0$ is induced by
a
semi-algebraic $G$\nobreakdash-equivariant right $\kS(C)$\nobreakdash-torsor $p_{B^0(R)}:\sP_{B^0(R)} \to B^0(R)$.

For each connected component~$e$ of~$B^0(R)$, as
the two semi-algebraic spaces~$e$ and~$\Omega_e$ are semi-algebraically contractible, the torsor $p_{B^0(R)}^{-1}(e) \to e$ can be extended (uniquely up to isomorphism) to a
semi-algebraic $G$\nobreakdash-equivariant right $\kS(C)$\nobreakdash-torsor $p_e:\sP_e \to \Omega_e$,
by Remark~\ref{rmk:geqdict}.
Let $p_0:\sP_0 \to \Omega_0$
be the disjoint union of the~$p_e$.

Let $e \in B^0(R)$ and $b \in \Sigma$ such that $\Omega_b \cap \Omega_e \neq\emptyset$.
Viewing~$\Sper(F_b)$ as a subset (of cardinality~$2$) of~$\Sper(F)$,
let~$b'$ be
the unique point of~$\Sper(F_b)$ that belongs to the constructible open
subset of~$\Sper(F)$ corresponding to~$e$.
Let~$R'$ denote the real closure of $F_b$ at~$b'$, let $C'=R'(\sqrt{-1})$,
and let $p_b':\sP_b' \to \Omega_b' \cap B^0(C')$ and
$p_e':\sP_e' \to \Omega_e'$
denote the morphisms of semi-algebraic spaces over~$R'$ obtained from~$p_b$ and~$p_e$ by an extension
of scalars (see \cite[\textsection4]{delfsknebuschonthehomology}).
Our assumption that~$\alpha_0$ and~$\alpha_b$ have the same image
in $H^0(\Sper(F_b),\sH^1(S))$ implies that $p_b'^{-1}(b')$
and $p_e'^{-1}(b')$ are isomorphic as $G$\nobreakdash-equivariant $\kS(C')$\nobreakdash-torsors
over $b'=\Sper(R')$.
As $\Omega_b \cap \Omega_e$ is semi-algebraically contractible, it follows,
by Remark~\ref{rmk:geqdict},
that $p_b'^{-1}(\Omega_b' \cap \Omega_e')$
and~$p_e'^{-1}(\Omega_b' \cap \Omega_e')$ are themselves isomorphic.
By \cite[Theorem~4.1]{delfsknebuschintrolocallysemialg}, we deduce
that $p_b^{-1}(\Omega_b \cap \Omega_e)$
and~$p_e^{-1}(\Omega_b \cap \Omega_e)$ are isomorphic.
Letting~$b$ and~$e$ vary, we conclude that
the restrictions of~$p_\Sigma$ and~$p_0$ to
$\Omega_\Sigma \cap B^0(C) \cap \Omega_0=\Omega_\Sigma \cap \Omega_0$
are isomorphic as $G$\nobreakdash-equivariant $\kS(C)$\nobreakdash-torsors.

Choosing an isomorphism between them and using it to glue~$p_\Sigma$ and~$p_0$,
we obtain a
semi-algebraic $G$\nobreakdash-equivariant right $\kS(C)$\nobreakdash-torsor $p:\sP \to \Omega \cap B^0(C)$.

After shrinking the~$\Omega_b$ and the~$\Omega_e$ by replacing a third time the given triangulation with
its barycentric subdivision, we may assume that~$p$ extends to
a semi-algebraic $G$\nobreakdash-equivariant right $\kS(C)$\nobreakdash-torsor $\overline p:\overline{\sP} \to \overline\Omega \cap B^0(C)$.

Let us fix a finite subset $\Sigma' \subset \partial\Omega$ that is stable under~$G$ and contains
at least one point in each connected component of~$\partial\Omega$.  We identify~$\Sigma'$
with a finite subset of (non-rational) closed points of~$B^0$ and let $B^1 = B^0 \setminus \Sigma'$.
As the connected components of $\partial\Omega \cap B^1(C)$ are semi-algebraically contractible
and as none of them is stable under~$G$,
we can trivialize the restriction of~$\overline p$ to $\overline p^{-1}(\partial\Omega \cap B^1(C))$
and thus glue $\overline p^{-1}(\overline\Omega \cap B^1(C))$,
along $\overline p^{-1}(\partial\Omega \cap B^1(C))$,
with the trivial torsor over $B^1(C) \setminus \Omega$,
using \cite[Chapter~II, Theorem~1.3]{delfsknebuschbook},
to finally obtain a
semi-algebraic $G$\nobreakdash-equivariant right $\kS(C)$\nobreakdash-torsor $p':\sP' \to B^1(C)$.

By Corollary~\ref{cor:huberknebuschpeterzilstarchenko}, there exists
a right $\kS$\nobreakdash-torsor $\mathfrak p':\kP'\to B^1$
such that~$\mathfrak p'(C)$ and~$p'$ are semi-algebraically isomorphic
as $G$\nobreakdash-equivariant $\kS(C)$\nobreakdash-torsors.

To conclude the proof, it is enough to check that the class in $H^1(F,S)$ of the generic fibre of~$\mathfrak p'$
induces~$\alpha_0$ and~$(\alpha_b)_{b \in \Sigma}$ by the maps from~\eqref{eq:commsquarelemma}.
By construction~$\alpha_0$ is induced by~$p_{B^0(R)}$ and hence indeed by~$\mathfrak p'(C)$.
Let us now fix $b \in \Sigma$.

We recall that~$\alpha_b$ is induced by a right $\kS$\nobreakdash-torsor $\mathfrak q_b:\kQ_b \to B'_b \setminus \{\sigma_b\}$
and that~$B'_b$ is a smooth connected curve
endowed with an \'etale map $\pi_b:B'_b \to B$
and with an $F$\nobreakdash-linear embedding $R(B'_b) \hookrightarrow F_b^h$.
Let~$B''_b$ denote a smooth connected curve endowed with
an \'etale map $\pi'_b:B''_b \to B'_b$,
an $R(B'_b)$\nobreakdash-linear
embedding $R(B''_b) \hookrightarrow F_b^h$ and
a rational point $\sigma''_b$ of the fibre $\pi_b'^{-1}(\sigma_b)$,
such that $\pi_b(\pi_b'(B''_b \setminus \{\sigma_b''\}))\subseteq B^1$.
Let $\Bbsecondezero = B''_b \setminus \{\sigma_b''\}$.
Let
$\kQ''=\kQ_b \times_{B'_b \setminus \{\sigma_b\}} \Bbsecondezero$
and
$\kP''=\kP' \times_{B^1} \Bbsecondezero$.
The quotient~$\kR$ of $\kQ'' \times_{\Bbsecondezero} \kP''$ by the diagonal right
action of~$\kS$ is an \'etale covering of~$\Bbsecondezero$.
Let $\kR' \to B''_b$ denote the normalisation of~$B''_b$ in $R(\kR)$,
so that $\kR \subset \kR'$.
By an appropriate choice of~$B''_b$, we may assume that every connected component of~$\kR'$
contains a unique closed point above~$\sigma_b''$.

By construction, there exists a semi-algebraic neighbourhood~$\Omega''_b$ of~$\sigma_b''$
in~$B''_b(C)$, stable under~$G$, connected if $b \in B(R)$ and with two connected components
permuted by~$G$ otherwise, such that
the
semi-algebraic $G$\nobreakdash-equivariant right $\kS(C)$\nobreakdash-torsors over
$\Omega''_b \cap \Bbsecondezero$
induced
by $\kQ''$ and by $\kP''$ are isomorphic.  Thus, the map $\kR'(C) \to B''_b(C)$
admits a $G$\nobreakdash-equivariant semi-algebraic section over $\Omega''_b \cap \Bbsecondezero$.
Viewing the connected component of~$\kR'$ that contains this section as a covering of~$B''_b$,
we have now found
a finite map between smooth connected curves over~$R$ that is totally ramified over
a closed point of the base and that at the level of $C$\nobreakdash-points admits a
$G$\nobreakdash-equivariant semi-algebraic section over a punctured $G$\nobreakdash-equivariant semi-algebraic neighbourhood of the point in question.
Such a map has to be \'etale.  It follows that the generic fibre of $\kR' \to B''_b$
admits an $F_b^h$\nobreakdash-point; hence the two right $S$\nobreakdash-torsors over~$F_b^h$
obtained by base change from $\mathfrak q_b$ and from $\mathfrak p'$ are isomorphic,
as desired.
\end{proof}

\subsubsection{Descent for weak approximation and the strong Hasse principle}
\label{subsubsec:descentwastronghp}

The weak approximation property for varieties over~$F$ is defined in this context
in the same way as when $R=\R$ (see Definition~\ref{weakdef}).  For the formulation of
our descent theorem, we shall need to consider this property in combination with the strong Hasse principle.

\begin{defn}
A variety~$X$ over~$F$ \textit{satisfies the strong Hasse principle}
if either
 $X(F_b)=\emptyset$ for infinitely many closed points $b \in B$
or $X(F)\neq\emptyset$.  In other words, the strong Hasse principle holds if the existence of $F_b$\nobreakdash-points for all but finitely many~$b$
implies the existence of a rational point.
\end{defn}

\begin{thm}
\label{th:wadescent}
Let $X$ be a smooth variety over $F$. Let $S$ be a linear algebraic group over $F$. Let $Q\to X$ be a left $S$-torsor over $X$. Assume that any twist of $Q$ by a right $S$-torsor over $F$ satisfies
weak approximation and the strong Hasse principle.
Then $X$ satisfies
weak approximation and the strong Hasse principle.
\end{thm}

\begin{proof}
To prove the strong Hasse principle and the weak approximation property for~$X$ in one go, we fix a finite
closed subset $\Sigma \subset B$, a family $(x_b)_{b\in \Sigma} \in \prod_{b \in \Sigma} X(F_b)$,
we assume that $X(F_b)\neq\emptyset$ for all but finitely many closed points~$b$ of~$B$, and we
seek a rational point of~$X$ arbitrarily close to $(x_b)_{b \in \Sigma}$.

We start with a lemma.

\begin{lem}
\label{lem:almosteverywhere}
For any non-empty variety~$Y$ over~$F$, the following are equivalent:
\begin{enumerate}[(i)]
\item $Y(F_b)\neq\emptyset$ for all but finitely many closed points~$b$ of~$B$;
\item $Y(F')\neq\emptyset$ for all real closed field extensions~$F'/F$.
\end{enumerate}
\end{lem}

\begin{proof}
After replacing~$Y$ with the disjoint union of finitely many locally closed subsets
which set-theoretically cover~$Y$, we may assume that~$Y$ is smooth.
After replacing~$Y$ with a smooth compactification, we may then assume that~$Y$ is
smooth and proper.  Let us fix a dense open subset $B^0 \subset B$ and a variety~$\kY$
endowed with a smooth and proper morphism $g:\kY \to B^0$ whose generic fibre is isomorphic
to~$Y$.
For any closed point $b \in B^0$,
the existence of an $F_b$\nobreakdash-point of~$Y$ is now equivalent to
that of a rational point of $g^{-1}(b)$.
As the latter condition is satisfied whenever~$b$ has degree~$2$ over~$R$, it follows that~(i) holds
if and only if
the image of the map $g(R):\kY(R)\to B^0(R)$
is the complement of a finite subset of~$B^0(R)$.
On the other hand, letting $g_r:\kY_r \to B^0_r$ denote the map induced by~$g$ at the level of real spectra
(\emph{i.e.}\ of abstract semi-algebraic spaces
in the sense of \cite[Appendix~A to Chapter~I]{delfsknebuschbook}),
condition~(ii) holds if and only if $g_r(\kY_r)$ contains the points of~$B_r$ that
lie over the generic point of~$B$.
Let us view~$B^0(R)$ as a subset of $B^0_r$.
As $g(R)(\kY(R))=g_r(\kY_r)\cap B^0(R)$
(\emph{loc.\ cit.}, Example~A.3~(ii) and \cite[Theorem~4.1.2]{bcr}),
as $B_r^0 \setminus g_r(\kY_r)$ is a constructible subset
of $B_r^0$
(see \cite[Proposition~2.3]{costeroyspectrereel})
and as any constructible subset of~$B_r^0$ either
contains both a point lying over the generic point of~$B$
and an infinite subset of~$B^0(R)$ or is itself a
finite subset of~$B^0(R)$
(see \cite[Theorem~7.2.3]{bcr}, \cite[Lemma~9.3]{DK2}),
we conclude that (i)$\Leftrightarrow$(ii).
\end{proof}

Let $\pi: X_r \to \Sper(F)$ denote the map induced
at the level of real spectra by the structural morphism of the variety~$X$
(see \cite[Appendix~A to Chapter~I]{delfsknebuschbook}).
For an open subset $U \subset \Sper(F)$, let~$J_U$ denote the category
introduced by Scheiderer \cite[(2.3)]{Scheiderer}, \cite[(1.13)]{scheidererbook};
its objects are the pairs $(L,s)$ where~$L$ is an \'etale $F$\nobreakdash-algebra
and $s:U \to \Sper(L)$ is a continuous section over~$U$ of the natural map $\Sper(L) \to \Sper(F)$.
The formula
\begin{align*}
H^0(U,\sH^0(X))=\varinjlim_{(L,s)\in J_U} X(L)
\end{align*}
defines a sheaf of sets $\sH^0(X)$ on~$\Sper(F)$.
For $\xi \in \Sper(F)$, the stalk
of~$\sH^0(X)$ at~$\xi$ is $X(F_\xi)$,
if~$F_\xi$ denotes the real closure of~$F$ with respect to the ordering~$\xi$.

Evaluation of the class $[Q^{-1}] \in \check{H}_{\et}^1(X,S)$ of the right $S$\nobreakdash-torsor $Q^{-1} \to X$
(see~\S\ref{torsors}) provides functorial maps $X(L) \to H^1(L,S)$ when~$L$ ranges over all $F$\nobreakdash-algebras,
hence a morphism
$\sH^0(X)\to \sH^1(S)$ of sheaves on~$\Sper(F)$.  These evaluation maps induce a morphism from the
commutative square
\begin{align}
\label{eq:commsquarebis}
\begin{aligned}
\xymatrix@R=3ex{
X(F) \ar[r] \ar[d] & \displaystyle\smash[b]{\prod_{b \in \Sigma}} X(F_b) \ar[d] \\
H^0(\Sper(F),\sH^0(X)) \ar[r] & \displaystyle\prod_{b \in \Sigma} H^0(\Sper(F_b),\sH^0(X))
}
\end{aligned}
\end{align}
to the commutative square~\eqref{eq:commsquarelemma}.

By Lemma~\ref{lem:almosteverywhere} applied to~$X$, all of the stalks of the sheaf~$\sH^0(X)$ are non-empty.
The bottom horizontal map of~\eqref{eq:commsquarebis}
is therefore surjective (see \cite[(C.2)]{Scheiderer}).
Let us fix $x_0 \in H^0(\Sper(F),\sH^0(X))$
having the same image as $(x_b)_{b\in \Sigma}$ in
$\prod_{b \in \Sigma} H^0(\Sper(F_b),\sH^0(X))$
and denote by $\alpha_0 \in H^0(\Sper(F),\sH^1(S))$ and $(\alpha_b)_{b \in \Sigma} \in \prod_{b \in \Sigma} H^1(F_b,S)$ the images of~$x_0$ and of~$(x_b)_{b \in \Sigma}$ in~\eqref{eq:commsquarelemma}.

By Proposition~\ref{prop:horrible}, there exists a right $S$\nobreakdash-torsor~$P$ over~$F$
whose isomorphism class induces~$\alpha_0$ and $(\alpha_b)_{b \in \Sigma}$.
After replacing the algebraic group~$S$ with its inner form~$_PS$
and the left $S$\nobreakdash-torsor $Q \to X$ with its twist $_PQ\to X$, which is a left $_PS$\nobreakdash-torsor
(see~\textsection\ref{torsors}), we may assume that $\alpha_0$ and the~$\alpha_b$ for $b \in \Sigma$ are
all trivial.

For any
real closed field
extension~$F'/F$, the $F'$\nobreakdash-point of~$X$ induced by~$x_0$
can be lifted to an $F'$\nobreakdash-point of~$Q$ since~$\alpha_0$ is trivial.
Thus~$Q(F')\neq\emptyset$ for all real closed field extensions~$F'/F$.
Since~$Q$ satisfies the strong Hasse principle,
Lemma~\ref{lem:almosteverywhere} now implies that~$Q(F)\neq\emptyset$.
As the~$\alpha_b$ for $b\in\Sigma$ are trivial, each $x_b \in X(F_b)$ can be lifted
to some $q_b \in Q(F_b)$, which we fix.  As~$Q$ satisfies the weak approximation property and as $Q(F)\neq\emptyset$,
we can find $q \in Q(F)$ arbitrarily close to $(q_b)_{b \in \Sigma}$.
The image of~$q$ in~$X(F)$ is then
arbitrarily close to $(x_b)_{b \in \Sigma}$, as desired.
\end{proof}

\subsubsection{Colliot-Thélène's conjecture over real closed fields}
\label{subsubsec:ctconjrealclosed}

We can at last prove the main theorem of~\textsection\ref{subsec:wahom},
Theorem~\ref{weakhomo} below.
Under the additional assumption that the stabilizers are connected,
Theorem~\ref{weakhomo}
is due to Scheiderer \cite[Corollary~6.2 and Theorem~8.9]{Scheiderer},
on whose work our proof of Theorem~\ref{weakhomo} depends.

\begin{thm}
\label{weakhomo}
Let~$R$ be a real closed field.  Let~$B$ be a smooth projective connected curve over~$R$, with function field~$F$.
Any homogeneous space of a connected linear algebraic group over~$F$ satisfies the weak approximation property.
\end{thm}

\begin{proof}
Let~$X$ be a homogeneous space of a connected linear algebraic group~$H$ over~$F$.
If $X(F_b)=\emptyset$ for some closed point $b\in B$, weak approximation trivially holds for~$X$.
Otherwise, as~$X$ satisfies the Hasse principle (see \cite[Corollary~6.2]{Scheiderer}), one has $X(F)\neq\emptyset$,
so that one can write $X=S \backslash H$ for some algebraic subgroup~$S$ of~$H$.
The quotient map $H\to S\backslash H$
is a left $S$\nobreakdash-torsor.
For any right $S$\nobreakdash-torsor $P$ over~$F$,
the variety $_PH$ is a right $H$\nobreakdash-torsor over~$F$;
in particular, it satisfies the strong Hasse principle and weak approximation
(see \cite[Corollary~4.2, Remark~4.3, Proposition~8.8]{Scheiderer}).
Applying Theorem~\ref{th:wadescent} now concludes the proof.
\end{proof}

\section{The image of the Borel--Haefliger map}

In this section, we consider a property that we think of as a cohomological variant of the tight approximation property and that we can verify in a greater generality.

\subsection{The Borel--Haefliger map}

For all smooth varieties $X$ over $\R$ and all $i\geq 0$, Borel and Haefliger have constructed a map $\cl_\R: \CH^i(X)\to H^i(X(\R),\Z/2\Z)$ compatible with pull-backs, proper push-forwards and cup products: the Borel--Haefliger cycle class map \cite{BH} (see also \cite[\S 1.6.2]{BWI}).
We restrict to the case where $X$ has pure dimension $d$ and $i=d-1$. The Borel--Haefliger map then reads $\cl_\R: \CH_1(X)\to H_1(X(\R),\Z/2\Z)$ and we let $H_1^{\alg}(X(\R),\Z/2\Z)$ be its image.

Recall from~\textsection\ref{conventions}
that~$B$ denotes a fixed smooth projective connected curve over~$\R$ (with function field~$F$).
One has $H_1(B(\R),\Z/2\Z)\simeq(\Z/2\Z)^{\pi_0(B(\R))}$, and the group $H_1^{\alg}(B(\R),\Z/2\Z)$, generated by $[B(\R)]$, is isomorphic to $\Z/2\Z$ unless $B(\R)=\varnothing$. If $f: \kX\to B$ is a proper flat morphism with $\kX$ regular, then
 \begin{equation}
\label{obviousincl}
H_1^{\alg}(\kX(\R),\Z/2\Z)\subset\{\alpha\in H_1(\kX(\R),\Z/2\Z)\mid f(\R)_*\alpha\in \langle[B(\R)]\rangle\}
\end{equation}
by the compatibility of $\cl_{\R}$ with proper push-forwards.

\begin{defn}
\label{defn:falg}
Let  $f: \kX\to B$ be a proper flat morphism with $\kX$ regular. We say that $H_1(\kX(\R),\Z/2\Z)$ is $f$-\textit{algebraic} if the inclusion (\ref{obviousincl}) is an equality.
\end{defn}

\begin{rmks}
\label{rmkBH}
(i) Suppose that $\kX=X\times B$ for a smooth proper variety $X$ over~$\R$, that $f$ is the second projection, that $B(\R)\neq\varnothing$, and that $H_1(\kX(\R),\Z/2\Z)$ is $f$\nobreakdash-algebraic. Then $H_1^{\alg}(X(\R),\Z/2\Z)=H_1(X(\R),\Z/2\Z)$.

(ii) If $H_1(\kX(\R),\Z/2\Z)$ is $f$-algebraic and $g: \kX\times \P^d_\R\to B$ is the composition of the first projection and of the morphism $f$, the K\"unneth formula shows that  $H_1(\kX(\R)\times\P^d(\R),\Z/2\Z)$ is $g$-algebraic.
\end{rmks}

\subsection{Birational invariance}
This property is a birational invariant.

\begin{prop}
\label{birinvBH}
Let $f: \kX\to B$ and $f': \kX'\to B$ be proper flat morphisms with $\kX$ and $\kX'$ regular.
Let $g: \kX'\dashrightarrow \kX$ be a birational map such that $f\circ g=f'$. If
$H_1(\kX(\R),\Z/2\Z)$ is $f$-algebraic, then $H_1(\kX'(\R),\Z/2\Z)$ is $f'$-algebraic.
\end{prop}

\begin{proof}
We first examine the case where~$g$ is a morphism.

\begin{lem}
\label{blowup}
Let $g: \kX'\to\kX$ be a birational morphism between smooth and proper varieties over $\R$.
\begin{enumerate}[(i)]
\item The map $g(\R)_*: H_1(\kX'(\R),\Z/2\Z)\to H_1(\kX(\R),\Z/2\Z)$ is onto.
\item One has $g(\R)_*(H_1^{\alg}(\kX'(\R),\Z/2\Z))=H_1^{\alg}(\kX(\R),\Z/2\Z)$.
\item If~$g$ is the blow-up of a smooth
subvariety of $\kX$, the kernel of $g(\R)_*$ is algebraic.
\end{enumerate}
\end{lem}

\begin{proof}
The equality $g(\R)_*\circ g(\R)^*=\Id$ in cohomology with $\Z/2\Z$ coefficients implies
the injectivity
of $g(\R)^*:H^1(\kX(\R),\Z/2\Z) \to H^1(\kX'(\R),\Z/2\Z)$,
and hence~(i) by Poincar\'e duality.
To prove~(ii), note that $\CH_1(\kX)$ is generated by classes of irreducible curves
that meet a dense open subset of~$\kX$ over which~$g$ is an isomorphism,
and consider their strict transforms in $\kX'$.
Let us now assume that~$g$ is the blow-up of a smooth subvariety $Z \subset \kX$ of codimension~$c\geq 2$,
with exceptional divisor~$E$,
and verify~(iii).
Let $d=\dim(\kX)$.
Let $i:E \hookrightarrow \kX$ and $\tau:E\to Z$ denote the natural morphisms.
Let $\xi \in H^1(E(\R),\Z/2\Z)$ be the first Stiefel--Whitney class
of the tautological line bundle $\sO_E(1)$ of the projective bundle~$\tau$.
The map
\begin{align}
\label{eq:realblowup}
H^{d-1}(\kX(\R),\Z/2\Z) \oplus H^{d-c}(Z(\R),\Z/2\Z) \isoto H^{d-1}(\kX'(\R),\Z/2\Z)
\end{align}
given by $x\oplus y \mapsto g(\R)^*x+  i(\R)_*(\xi^{c-2} \smile \tau(\R)^*y) $
is an isomorphism, as follows from the computation
of the cohomology with $\Z/2\Z$ coefficients of a blow-up of real manifolds
(combine \cite[Theorem~3.7, Theorem~3.10 and~\textsection5]{gitler} with the Thom isomorphism
and with the computation of the cohomology with $\Z/2\Z$ coefficients of a real projective bundle
\cite[Chapter~17, Theorem~2.5 and Theorem~10.3]{husemoller}).
Identifying the cohomology of the compact manifolds~$\kX(\R)$, $Z(\R)$, $\kX'(\R)$ with their homology
and using the equality $g(\R)_*\circ g(\R)^\star=\Id$ and the projection formula, we deduce from the
isomorphism~\eqref{eq:realblowup} the exactness of the sequence
\begin{align}
0 \to H_0(Z(\R),\Z/2\Z) \xrightarrow{\;\alpha\;} H_1(\kX'(\R),\Z/2\Z) \xrightarrow{g(\R)_*} H_1(\kX(\R),\Z/2\Z)\to 0\rlap,
\end{align}
where~$\alpha$ is defined by $\alpha(y)= i(\R)_*(\xi^{c-2} \smile \tau(\R)^*y)$.
As~$\alpha$ sends the class
of any point $z \in Z(\R)$ to the Borel--Haefliger cycle class in~$\kX'(\R)$
of a line of the projective space~$\tau^{-1}(z)$, assertion~(iii) follows.
\end{proof}

Let us deduce Proposition~\ref{birinvBH}, whose notation we take up, from Lemma~\ref{blowup}.
By Hironaka's theorem on the resolution of indeterminacies \cite[\S 5]{Hironaka} applied to~$g^{-1}$, there exist
a composition of blow-ups with smooth centres $\pi:\kX'' \to \kX$ and a morphism $h:\kX''\to \kX'$
such that $g \circ h = \pi$.
If $H_1(\kX(\R),\Z/2\Z)$ is $f$\nobreakdash-algebraic, then $H_1(\kX''(\R),\Z/2\Z)$ is
$(f\circ \pi)$\nobreakdash-algebraic
by Lemma~\ref{blowup}~(ii)--(iii) applied to the individual blow-ups that make up~$\pi$.
This, in turn, implies that $H_1(\kX'(\R),\Z/2\Z)$ is $f'$\nobreakdash-algebraic,
by Lemma~\ref{blowup}~(i) applied to~$h$.
\end{proof}

\subsection{Fibrations}

We now prove the counterpart of Theorem \ref{up}.

\begin{thm}
\label{fibBH}
Let $f: \kX\to B$ and $f': \kX'\to B$ be proper flat morphisms with $\kX$ and $\kX'$ regular, let $g: \kX\to\kX'$ be a dominant morphism with $f'\circ g=f$ and let $(\kX')^0\subset \kX'$ be a dense open subset.
Suppose that $H_1(\kX'(\R),\Z/2\Z)$ is $f'$-algebraic. Suppose moreover that for all  connected smooth projective curves $\widetilde{B}$ over $\R$, all morphisms $\nu: \widetilde{B}\to \kX'$ with $f'\circ\nu$ dominant and $\nu(\widetilde{B})\cap(\kX')^0\neq\varnothing$, all regular modifications $\kY\to\kX\times_{\kX'}\widetilde{B}$ with projection $h: \kY\to \widetilde{B}$ and all $\ci$ sections $v: \widetilde{B}(\R)\to \kY(\R)$ of $h(\R)$, one has $v_*[\widetilde{B}(\R)]\in H_1^{\alg}(\kY(\R),\Z/2\Z)$.
Then $H_1(\kX(\R),\Z/2\Z)$ is $f$-algebraic.
\end{thm}

\begin{proof}
By Proposition \ref{birinvBH} and \cite[Theorem 1.1]{ADK}, we may assume that $g$ underlies a toroidal morphism $g:(\kU\subset \kX)\to(\kU'\subset\kX')$ of snc toroidal embeddings and that $\kX$ and $\kX'$ are projective.
Choose $\alpha\in H_1(\kX(\R),\Z/2\Z)$ such that $f(\R)_*\alpha\in \langle[B(\R)]\rangle$. We will show that $\alpha$ is algebraic. Representing $\alpha$ by a topological cycle, perturbing it and applying Whitney's approximation theorem \cite[Theorem 2]{Whitney} shows that $\alpha=\iota_*[\bS]$ for some  compact one-dimensional real-analytic manifold $\bS$ and some real-analytic map $\iota: \bS\to \kX(\R)$ such that the image by $\iota$ of any connected component of $\bS$ meets $g^{-1}((\kX')^0)(\R)\cap\kU(\R)$ and is not included in a fiber of $f(\R)$.

Let $\lambda: \bS\to\P^2(\R)$ be a real-analytic embedding. By Remark \ref{rmkBH} (ii), in order to show the algebraicity of $\iota_*[\bS]$, we may replace $\kU\subset\kX$, $\kU'\subset\kX'$, $(\kX')^0$ and $\iota$ with $\kU\times\P^2_\R\subset\kX\times\P^2_\R$, $\kU'\times\P^2_\R\subset\kX'\times\P^2_\R$, $(\kX')^0\times\P^2_\R$ and $(\iota,\lambda): \bS\to\kX(\R)\times\P^2(\R)$. We may thus assume that $\iota$ and $\iota':=g(\R)\circ\iota$ are embeddings.

Define $\Sigma\subset\bS$ to be the finite set of $x\in\bS$ such that
$\iota(x) \notin \kU(\R)$.
Applying Proposition \ref{toroidaltheorem} to the toroidal morphism $g$ and to the discrete valuation rings $\sO_{\bS,x}^{\an}$ for $x\in\Sigma$ shows the existence of regular modifications $\pi:\kZ\to\kX$ and $\pi':\kZ'\to\kX'$ and of a morphism $h:\kZ\to\kZ'$ such that $g\circ\pi=\pi'\circ h$, such that $\pi$ (resp.\ $\pi'$) is an isomorphism above $\kU$ (resp.\ above $\kU'$) and such that the strict transform $j: \bS\to \kZ(\R)$ of $\iota$ has the property that $h(\R)$ is submersive at $j(x)$ for all $x\in \Sigma$. Since $g$ was already smooth along $\kU$, the latter property in fact holds for all $x\in\bS$.
It is legitimate,
in order to show the algebraicity of $\iota_*[\bS]$, to replace $\kX$, $\kX'$, $g$ and $\iota$ with $\kZ$, $\kZ'$, $h$ and $j$. We may thus assume that $g(\R)$ is submersive along $\iota(\bS)$.

It then follows from \cite[Lemma 6.11]{BWII} that there exists an open neighbourhood $W$ of $\iota'(\bS)$ in $\kX'(\R)$ and a $\ci$ section $w: W\to \kX(\R)$ of $g(\R)$ above $W$ such that $\iota=w \circ \iota'$. Note that the existence of the section $w$ shows that $g(\R)$ is submersive along $w(W)$.
Since $H_1(\kX'(\R),\Z/2\Z)$ is $f'$-algebraic, a theorem of Akbulut--King (\cite{AkbulutKing}, see \cite[Theorem 6.8]{BWII}) shows the existence of a curve $\widetilde{B}\subset \kX'$ smooth along $\widetilde{B}(\R)$ such that $\widetilde{B}(\R)\subset W$, such that $[\widetilde{B}(\R)]=\iota'_*[\bS]\in H_1(W,\Z/2\Z)$ (where we still denote by $\iota':\bS\to W$ the natural map), and such that $f'$ does not contract any component of $\widetilde{B}$.

Let $(\nu_i:\widetilde{B}_i\to \kX')_{i\in I}$ be the connected components of the normalization of~$\widetilde{B}$, and denote by $\mu_i:\widetilde{B}_i(\R)\to W$ the induced maps.
Consider a regular modification $\kY_i\to\kX\times_{\kX'}\widetilde{B}_i$ that is an isomorphism above the regular locus, where the fiber product is taken with respect to $g$ and $\nu_i$, and let $h_i:\kY_i\to \widetilde{B}_i$ and $p_i:\kY_i\to\kX$ be the natural morphisms. The strict transform $v_i: \widetilde{B}_i(\R)\to \kY_i(\R)$ of $(w\circ \mu_i,\Id): \widetilde{B}_i(\R)\to(\kX\times_{\kX'}\widetilde{B}_i)(\R)$ is a $\ci$ section of $h_i(\R)$.
By hypothesis,  $v_{i,*}[\widetilde{B}_i(\R)]\in H^{\alg}_1(\kY_i(\R),\Z/2\Z)$.
It follows that $\alpha\in H_1^{\alg}(\kX(\R),\Z/2\Z)$ since
\begin{equation*}
\alpha=\iota_*[\bS]=w_*\iota'_*[\bS]=\sum_{i\in I}w_*\mu_{i,*}[\widetilde{B}_i(\R)]=\sum_{i\in I}p_{i}(\R)_*v_{i,*}[\widetilde{B}_i(\R)].\qedhere
\end{equation*}
\end{proof}

The hypothesis $v_*[\widetilde{B}(\R)]\in H_1^{\alg}(\kY(\R),\Z/2\Z)$
appearing in Theorem~\ref{fibBH}
 is of course verified if $H_1(\kY(\R),\Z/2\Z)$ is $h$-algebraic.
 It is also verified if $h$ satisfies the tight approximation property, by Corollary \ref{realsectight}. Thus, applying Theorem \ref{fibBH} with $\kX'=B$ and $f'=\Id$ yields:

\begin{cor}
\label{BHtight}
Let $f: \kX\to B$ be a proper flat morphism with $\kX$ regular. Let~$X$ be the generic fiber of~$f$.
If for every finite extension $\widetilde{F}$ of $F$, the $\widetilde{F}$-variety $X_{\widetilde{F}}$ satisfies the tight approximation property, then $H_1(\kX(\R),\Z/2\Z)$ is $f$-algebraic.
\end{cor}

We note that the assumptions of Corollary~\ref{BHtight} imply that~$X$ is rationally connected
(see Proposition~\ref{tightimpliesRC}).  If one views the tight approximation property
as the analogue, for varieties over~$F$,
of the property,
for rationally connected varieties over a number field,
that rational points are dense in the Brauer--Manin set,
then Corollary~\ref{BHtight} is the analogue of Liang's theorem \cite[Theorem~B]{liang}
on the relationship between rational points
and zero-cycles in the latter context.
Correspondingly, Theorem~\ref{fibBH} is the analogue of \cite[Corollary~8.4~(2)]{harpazwittenberg}.

\begin{example}
\label{ex:falgiteratedfibhomspace}
Let $f: \kX\to B$ be a proper flat morphism with $\kX$ regular.
If the generic fiber of~$f$ is, birationally, an iterated fibration into homogeneous spaces
of connected linear algebraic groups,
then $H_1(\kX(\R),\Z/2\Z)$ is $f$-algebraic.
Indeed, by Theorem~\ref{up} and Theorem~\ref{homogeneous}, the hypotheses of Corollary \ref{BHtight} are satisfied.
\end{example}

\begin{rmk}
Let~$X$ be a smooth compactification of a homogeneous
space of a connected linear algebraic group over~$\R$.
By Example~\ref{ex:falgiteratedfibhomspace} and Remark~\ref{rmkBH}~(i),
one has $H_1^\alg(X(\R),\Z/2\Z)=H_1(X(\R),\Z/2\Z)$.
We do not know whether this remains true if~$\R$ is replaced with
an arbitrary real closed field (also replacing $H_1(X(\R),\Z/2\Z)$ with
semi-algebraic homology and~$\cl_\R$ with the map \cite[(1.56)]{BWI}).
\end{rmk}

\subsection{Cubics}

We now give applications of the above results to cubic hypersurfaces.

\begin{thm}
\label{cubicfib}
Let $f:\kX\to B$ be a proper flat morphism with $\kX$ regular whose generic fiber is a cubic hypersurface of dimension $d\geq 2$. Then $H_1(\kX(\R),\Z/2\Z)$ is $f$-algebraic.
\end{thm}

\begin{proof}
We argue by induction on $d$, which we fix, letting all other data, including~$B$, vary. We let $X$ be the generic fiber of $f$.

If $d=2$, this is \cite[Theorem 8.1 (iv), Proposition 8.4 (i)]{BWII}. More directly, one can use the following variant of the argument of \emph{loc.\ cit.}: after replacing $B$ with a finite covering $\widetilde{B}$ of odd degree and $\kX$ with a regular modification of $\kX\times_B\widetilde{B}$, we may assume that $X$ contains a line $\ell$.
% Projecting from $\ell$ shows that $X$ is birational to a conic bundle over $\P^1_F$.
%Apply Proposition \ref{birinvBH} and Theorem~\ref{fibBH} to conclude.
Apply Example~\ref{cubique} and Corollary~\ref{BHtight} to conclude.

If $d>2$, choosing a generic pencil of hyperplane sections on~$X$ yields a morphism $g:\widetilde{X}\to\P^1_F$ where $\widetilde{X}$ is the blow-up of a smooth subvariety of $X$, and where the smooth fibers of $g$ are cubic hypersurfaces of dimension $d-1$. Choose a proper regular model $\widetilde{f}:\widetilde{\kX}\to B$ of $\widetilde{X}$ for which $g$ extends to a morphism $g:\widetilde{\kX}\to B \times \P^1_\R$.
Let $\widetilde{\kX}'=B\times \P^1_\R$, let $\widetilde{f}':\widetilde{\kX}'\to B$ be the first projection and let
$(\widetilde{\kX}')^0\subset \widetilde{\kX}'$ be a dense open subset over which the fibers of $g$ are smooth cubic hypersurfaces.
By Remark~\ref{rmkBH}~(ii),
$H_1(\widetilde{\kX}'(\R),\Z/2\Z)$ is $\widetilde{f}'$-algebraic.
By the induction hypothesis and by Proposition~\ref{birinvBH},
we may apply Theorem \ref{fibBH} to $\widetilde{f}$, $\widetilde{f}'$
and $g$, thus showing that $H_1(\widetilde{\kX}(\R),\Z/2\Z)$ is $\widetilde{f}$-algebraic. One concludes by Proposition~\ref{birinvBH}.
\end{proof}

As a sample application, we deduce:

\begin{thm}
\label{BHcubic}
Let $f:\kX\to B$ be a proper flat morphism with $\kX$ regular.
Assume that the generic fiber of~$f$ is, birationally, an iterated fibration
into varieties of the following types (allowed to appear both, and to be interleaved):
\begin{enumerate}
\item smooth cubic hypersurfaces of dimension $\geq 2$;
\item homogeneous spaces of connected linear algebraic groups.
\end{enumerate}
Then $H_1(\kX(\R),\Z/2\Z)$ is $f$-algebraic.
\end{thm}

\begin{proof}
This follows from
Proposition~\ref{birinvBH} and
Theorem~\ref{fibBH},
in view of
Theorem~\ref{cubicfib}
and of Example~\ref{ex:falgiteratedfibhomspace}.
\end{proof}

\begin{cor}
\label{corBHcubic}
Let~$X$ be a smooth and proper variety over~$\R$.
Assume that~$X$ is, birationally, an iterated fibration into smooth cubic hypersurfaces
of dimension~$\geq 2$.
Then $H_1^\alg(X(\R),\Z/2\Z)=H_1(X(\R),\Z/2\Z)$.
\end{cor}

\begin{proof}
Apply Theorem~\ref{BHcubic}
to $B=\P^1_\R$ and $\kX=X \times B$
and see Remark~\ref{rmkBH}~(i).
\end{proof}

\begin{rmk}
When~$X$ is itself a smooth cubic hypersurface of dimension~$\geq 2$,
we have already seen that the Borel--Haefliger classes of rational curves
span the group $H_1(X(\R),\Z/2\Z)$,
since~$X_{\R(t)}$ satisfies the tight approximation property
(see Example~\ref{cubique}
and apply Corollary~\ref{realsectight}).
A theorem of
Shen \cite[Theorem~1.7]{Shen}
allows one to deduce
(at least in dimension~$\geq 3$) that the Borel--Haefliger classes of the lines contained in~$X$ suffice to span this group, a
fact first proved in \cite[Theorem~9.23~(ii), Remark~9.24~(i)]{BWII} and in~\cite{slavasergei}
and which holds over
arbitrary real closed fields (see \cite[Theorem~9.23~(ii)]{BWII}).
For more general~$X$ as in Corollary~\ref{corBHcubic},
however, we do not know whether
the Borel--Haefliger classes of rational curves
span the group $H_1(X(\R),\Z/2\Z)$
and we do not know whether Corollary~\ref{corBHcubic} can be generalized
to arbitrary real closed fields.
\end{rmk}

\subsection{Del Pezzo surfaces}
\label{dPpar}

Consider the following question:

\begin{question}
\label{qrat}
Let $X$ be a smooth proper rationally connected variety over $\R$.
Is $H_1(X(\R),\Z/2\Z)$ generated by Borel--Haefliger classes of rational curves?
\end{question}

Over a non-archimedean real closed field $R$, Question \ref{qrat} has a negative answer in general, as $H_1(X(R),\Z/2\Z)$ may not even be generated by Borel--Haefliger classes of curves \cite[Example~9.7]{BWII}. Over $\R$, the techniques of \cite{Kollocal, Kolfertile} shed no light on Question \ref{qrat} as the real locus of the rational curves on $X$ that are constructed in \emph{op.\ cit.} is always homologically trivial (even homotopically trivial) in $X(\R)$.

If $X_{\R(t)}$ satisfies the tight approximation property, then
Question \ref{qrat} has a positive answer for~$X$,
as a consequence of Corollary~\ref{realsectight};
by Theorems \ref{up} and~\ref{homogeneous}, such is the case if~$X$ is, birationally, an iterated fibration into homogeneous spaces of connected linear algebraic groups.
Our goal in \S\ref{dPpar} is to give a positive answer for arbitrary rationally connected surfaces.

\begin{prop}
\label{dP}
Let $X$ be a rationally connected surface over $\R$.  Then the Borel--Haefliger classes of rational curves on $X$ generate $H_1(X(\R),\Z/2\Z)$.
\end{prop}

\begin{proof}
We may assume that $X(\R)\neq\varnothing$.
Since the conclusion of Proposition~\ref{dP} is invariant under blow-ups at closed points, we may replace $X$ by a surface birational to it, and thus assume that $X$ is minimal and not birational to $\P^2_{\R}$. One can then use the classification of minimal rationally connected surfaces over $\R$ going back to Comessatti \cite[Teorema VI]{comessatti}.

 If $X$ has a conic bundle structure, then $X(\R)$ is a union of spheres (see \cite[VI, Corollary 3.1, Proposition~3.2, Proposition~6.1]{Silhol}), and the proposition holds since in this case $H_1(X(\R),\Z/2\Z)=0$.

 If $X$ is a del Pezzo surface, then $X$ is isomorphic to a degree~$2$ del Pezzo surface with~$4$ spheres as real connected components, or to a degree~$1$ del Pezzo surface with~$4$ spheres and~$1$ real projective plane as real connected components (see \cite[VI, Proofs of Theorem~4.6 and Proposition 6.3]{Silhol}). The proposition is trivial in the first case since $H_1(X(\R),\Z/2\Z)=0$.

In the second case, we let $\widetilde{X}\to X$ be the blow-up of the unique base point of the anticanonical linear system of $X$, with exceptional divisor $E$. The surface $\widetilde{X}$ carries an elliptic fibration $\widetilde{X}\to \P^1_{\R}$ of which $E$ is a section. We denote by $\Pi\subset \widetilde{X}(\R)$ the connected component containing $E(\R)$, and by $\Pi'\subset\widetilde{X}(\R)$ any other connected component. The smooth fibers of $f$ are elliptic curves, whose real loci have at most $2$  connected components. Since $\Pi$ and $\Pi'$ are not the only components of $\widetilde{X}(\R)$ (there are $3$ others), and since $\Pi\to \P^1(\R)$ is surjective, we deduce that $\Pi'\to \P^1(\R)$ is not surjective. Its image is a closed interval with distinct endpoints $a,b\in\P^1(\R)$.
The fibers $\widetilde{X}_a$ and $\widetilde{X}_b$ are singular (along their intersection with $\Pi'$), so that their geometric irreducible components are rational curves. Since $f$ is smooth along the section $E$, $\widetilde{X}_a$ and $\widetilde{X}_b$ contain a unique geometric irreducible component meeting $E$, denoted by $C_a$ and $C_b$. They are defined over $\R$ and meet $E$ transversally.

Note that the real loci of the normalizations of $E$, $C_a$, $C_b$ are connected. Consequently, the Borel--Haefliger classes of $E$, $C_a$ and $C_b$ belong to $H_1(\Pi,\Z/2\Z)$. Obviously, $\cl_{\R}(C_a)\cdot\cl_{\R}(E)=1$, $\cl_{\R}(C_b)\cdot \cl_{\R}(E)=1$ and $\cl_{\R}(C_a)\cdot\cl_{\R}(C_b)=0$. From these equalities, it follows that these classes generate a sub-vector space of dimension at least $2$ in $H_1(\Pi,\Z/2\Z)$. But this is only possible if the connected component of $X(\R)$ that was blown-up was the one isomorphic to a projective plane, and shows that $H_1(\Pi,\Z/2\Z)$ is in fact generated by $\cl_{\R}(E)$, $\cl_{\R}(C_a)$ and $\cl_{\R}(C_b)$. Since the other~$4$ connected components of $\widetilde{X}(\R)$ are isomorphic to spheres, we deduce that $H_1(\widetilde{X}(\R),\Z/2\Z)$ is generated by rational curves, which concludes the proof.
\end{proof}

\begin{rmk}
Proposition \ref{dP} remains true over an arbitrary real closed field $R$. Indeed,
in the proof of Proposition \ref{dP},
one may replace without any loss
the word ``sphere'' (resp.\ ``real projective plane'') with ``space~$S$
such that $H_1(S,\Z/2\Z)$ has dimension~$0$'' (resp.\ ``dimension~$1$'');
the proofs given in \cite{Silhol} of the assertions used in the above argument
then work over an arbitrary real closed field.
\end{rmk}

\appendix
\section{\texorpdfstring{$G$-equivariant}{𝐺-equivariant} complex analytic spaces}

In this appendix, we develop the basics of $G$-equivariant complex geometry, and collect the results that we need.

\subsection{Definition}
\label{Ranalytic}

A \textit{$G$-equivariant complex analytic space} is a complex analytic space $(Z,\sO_Z)$ in the sense of \cite[p.~16]{Stein} whose underlying locally ringed space is endowed with an action of $G$ such that the complex conjugation $\sigma$ acts $\C$-antilinearly on $\sO_Z$.
It is said to be a manifold (resp.\ Stein, projective...)\ if so is the underlying complex analytic space.
A $G$-equivariant complex manifold is nothing but a complex analytic manifold endowed with an action of $G$ such that $\sigma$ acts antiholomorphically.

If $Z$ is a $G$-equivariant complex analytic space, a $G$-equivariant coherent sheaf on $Z$ is a coherent sheaf on the underlying complex analytic space that is endowed with an action of $G$ compatible with its $\sO_Z$-module structure.

\vspace{1em}
There are two equivalent approaches to $G$-equivariant complex geometry. One can consider the $G$-equivariant spaces defined above, as in \cite[\S II.4]{GMT} where they are called complex analytic spaces with an antiinvolution, as in \cite[p.~250]{Beilinson} where they are called analytic spaces over $\R$, or as in \cite[\S 2.1]{CiliPe} where they are called real structures on complex spaces.
 One can also consider their quotients by the action of~$G$ in the category of locally ringed spaces: those are the Berkovich $\R$-analytic spaces hinted at in \cite[Examples 1.5.4]{Berkovich}, which were also considered by Huisman \cite{Huismanexp} under the name of real analytic spaces.
The reason for our choice is that classical results of complex geometry apply more directly in the former context.

\subsection{Analytification}
\label{Ranal}
Let $(Z,\sO_Z)$ be a complex analytic space with structural morphism $\mu: \C\to\sO_Z$. We define
its conjugate $(Z^{\sigma},\sO_{Z^{\sigma}})$ to be equal to $(Z,\sO_Z)$ as a locally ringed space, but with structural morphism $\mu\circ \sigma$. With a coherent sheaf $\sF$ on $Z$, one associates a coherent sheaf $\sF^\sigma$ on $Z^{\sigma}$: it is equal to $\sF$ as a sheaf and its $\sO_{Z^{\sigma}}$-module structure is induced by the equality $\sO_{Z^{\sigma}}=\sO_Z$.
One verifies that a $G$-equivariant complex analytic space is nothing but a complex analytic space $Z$ endowed with an isomorphism $\alpha: Z^\sigma\to Z$ of complex analytic spaces such that $\alpha\circ\alpha^{\sigma}=\Id_Z$, and that a $G$-equivariant coherent sheaf on it is a coherent sheaf $\sF$ on the underlying complex space endowed with an isomorphism $\beta: \alpha^*\sF\to\sF^{\sigma}$ such that $\beta^{\sigma}\circ \alpha^{\sigma *}\beta=\Id_{\sF}$.

If $X$ is a variety over $\R$, the analytification $X_{\C}^{\an}$ of $X_{\C}$ has a natural structure of $G$-equivariant complex analytic space, denoted by $X^{\an}$. Similarly, if $\sF$ is a coherent sheaf on~$X$, the analytification $\sF^{\an}_{\C}$ of $\sF_{\C}$ has a natural structure of a $G$-equivariant coherent sheaf, denoted by $\sF^{\an}$.

When~$X$ is proper,
the analytification functor $Y\mapsto Y^{\an}$ induces an equivalence between the
categories of closed subvarieties of~$X$
and of closed $G$\nobreakdash-equivariant analytic subspaces of~$X^{\an}$,
and the functor
$\sF\mapsto \sF^{\an}$ is an equivalence between the categories of coherent sheaves on~$X$ and of $G$\nobreakdash-equivariant coherent sheaves on $X^{\an}$.
These facts follow from the above description of $G$\nobreakdash-equivariant complex analytic spaces and $G$-equivariant coherent sheaves, together with
Serre's GAGA theorem for proper varieties
\cite[Expos\'e~XII, Th\'eor\`eme~4.4]{SGA1} (see also \emph{loc.\ cit.}, Corollaire~4.6)
and Galois descent (see for instance \cite[I \S 1]{Silhol}).

Arguing in the same way, we obtain a $G$-equivariant version of Riemann's existence theorem from the non-equivariant statement \cite[Expos\'e~XII, Th\'eor\`eme~5.1]{SGA1}: for any variety~$X$ over~$\R$, the analytification functor induces an equivalence between the categories of finite \'etale coverings of $X$ and of $G$-equivariant finite topological coverings of~$X(\C)$.

\subsection{The Stein property}
\label{parStein}
Recall that a complex analytic space $Z$ is \emph{Stein} if $H^q(Z,\sF)=0$ for all coherent sheaves $\sF$ on $Z$ and all $q>0$.
We collect here for later use $G$-equivariant analogues of well-known consequences of the Stein property.

For a $G$-equivariant coherent sheaf $\sF$ on a $G$-equivariant complex analytic space~$Z$,
the group $G$ acts on $H^q(Z,\sF)$, the action of
$\sigma$ being $\C$-antilinear.
Since the $H^q(Z,\sF)$ are $\C$-vector spaces, one has $H^p(G, H^q(Z,\sF))=0$ for $p>0$, and the second spectral sequence of equivariant cohomology
\cite[Th\'eor\`eme 5.2.1]{Tohoku}
 shows that $H^q_G(Z,\sF)=H^q(Z,\sF)^G$. Thus, if $Z$ is Stein, then $H^q_G(Z,\sF)=0$ for all $q>0$.

The real vector space $H^q(Z,\sF)^G$ satisfies
 $H^q(Z,\sF)=H^q(Z,\sF)^G\otimes_\R \C$, giving rise to a real structure on $H^q(Z,\sF)$.
The cohomology long exact sequence induced by a short exact sequence of $G$-equivariant coherent sheaves is $G$-equivariant, hence is defined over $\R$ for these real structures.

\begin{lem}
\label{SteinR}
Let $Z$ be a Stein $G$-equivariant complex analytic space.
\begin{enumerate}[(i)]
\item If $Z$ has finite dimension, a $G$-equivariant coherent sheaf $\sF$ on $Z$ whose fibers have bounded dimensions
 is a quotient of a trivial $G$-equivariant coherent sheaf.
\item A short exact sequence $0\to\sF_1\to\sF_2\to\sF_3\to 0$ of $G$-equivariant coherent sheaves on $Z$ splits $G$-equivariantly if $\sF_3$ is locally free.
\end{enumerate}
\end{lem}

\begin{proof}
(i) Combining Cartan's Theorem A and \cite[Theorem 1]{Kripke}, we see that $\sF$ is generated by finitely many global sections. The smallest $G$-stable sub-$\C$-vector space $V\subset H^0(Z,\sF)$ containing these sections is defined over $\R$: an $\R$-basis $\zeta_1,\dots,\zeta_N$ of $V\cap H^0(Z,\sF)^G$ is also a $\C$-basis of $V$. The $\zeta_i$ induce a $G$-equivariant surjection $\sO_Z^N\to \sF$ of coherent sheaves, as required.

(ii)  Consider the exact sequence $0\to \sF_1\otimes \sF_3^\vee\to \sF_2\otimes \sF_3^\vee\to \sF_3\otimes \sF_3^\vee\to 0$. Since $Z$ is Stein, there is an induced short exact sequence of global sections on $Z$, hence of the underlying real vector spaces:
$$0\to H^0(Z,\sF_1\otimes \sF_3^\vee)^G\to H^0(Z,\sF_2\otimes \sF_3^\vee)^G\to H^0(Z,\sF_3\otimes \sF_3^\vee)^G\to 0.$$
A lift of $\Id_{\sF_3}\in H^0(Z,\sF_3\otimes \sF_3^\vee)^G$ in $H^0(Z,\sF_2\otimes \sF_3^\vee)^G$ corresponds to a $G$-equivariant morphism $\sF_3\to \sF_2$ of coherent sheaves inducing the required splitting.
\end{proof}

\begin{prop}
\label{Siu}
 Let $Z$ be a $G$-equivariant complex analytic space. Any  $G$-stable locally closed Stein subspace $Y\subset Z$ has a $G$-stable Stein open neighbourhood in $Z$.
\end{prop}

\begin{proof}
By Siu's theorem \cite{Siu}, $Y$ has a Stein open neighbourhood $Y'$ in $Z$. The $G$-stable open neighbourhood $Y'\cap\sigma(Y')$ is then Stein by \cite[p.~127]{Stein}.
\end{proof}

\subsection{The Picard group}
\label{RPicard}
Let $Z$ be a $G$-equivariant complex analytic space.
The isomorphism classes of $G$\nobreakdash-equivariant invertible sheaves on $Z$ form a group for the tensor product: the Picard group $\Pic_G(Z)$ of $Z$.
Letting $\sO^*_Z\subset\sO_Z$ be the $G$-equivariant subsheaf of invertible analytic functions, one has an isomorphism $\Pic_G(Z)\simeq H^1_G(Z,\sO_Z^*)$. This follows from the \v{C}ech description of $G$-equivariant cohomology \cite[Th\'eor\`eme~5.5.6]{Tohoku} (for details
in the topological setting see \cite[p.~698]{kahnchern} or \cite[Proposition~1.1.1]{krasnovcharacteristicclasses}).

Let $\Z(1)\subset \C$ be the sub-$G$-module generated by $\sqrt{-1}$.
Viewing the exponential exact sequence  \cite[Lemma p.~142]{Stein}
\begin{equation}
\label{expex}
0\to \Z(1)\to\sO_Z\xrightarrow{f\mapsto \exp(2\pi f)}\sO_Z^*\to 0
\end{equation}
as a short exact sequence of $G$-equivariant sheaves on $Z$ yields a boundary map
$\cl:\Pic_G(Z)\to H^2_G(Z,\Z(1))$, the so-called equivariant cycle class map.
Composing it with the restriction map $H^2_G(Z,\Z(1))\to H^2_G(Z^G,\Z(1))$ and with the canonical isomorphism $H^2_G(Z^G,\Z(1))\simeq H^1(Z^G,\Z/2\Z)$ described in \cite[Theorem 1.3]{krasnovequivariant} induces the Borel--Haefliger cycle class map $\cl_\R: \Pic_G(Z)\to H^1(Z^G,\Z/2\Z)$.

\begin{prop}
\label{linebundleR}
The map $\cl: \Pic_G(Z)\to H^2_G(Z,\Z(1))$ is an isomorphism if $Z$ is a Stein $G$-equivariant complex analytic space. So is $\cl_\R: \Pic_G(Z)\to H^1(Z^G,\Z/2\Z)$ if $Z$ is a Stein $G$-equivariant complex manifold of pure dimension $1$.
\end{prop}

\begin{proof}
If $Z$ is Stein, $H^q_G(Z,\sO_Z)=0$ for $q>0$ (see \S\ref{parStein}).
The long exact sequence of $G$-equivariant cohomology induced by the exponential exact sequence then shows that $\cl$ is an isomorphism.

It remains to prove that the restriction map $H^2_G(Z,\Z(1))\to H^2_G(Z^G,\Z(1))$ is an isomorphism if $Z$ is a Stein $G$-equivariant complex manifold of pure dimension $1$.
To do so, we let $i: Z\setminus Z^G\inj Z$ be the inclusion and we let $i_!$ denote the extension by zero.
Since $G$ acts antiholomorphically on $Z$, the fixed point set $Z^G\subset Z$ is a one-dimensional $\ci$ closed submanifold, and the quotient $Z/G$ is a $\ci$ manifold with boundary $Z^G$ and interior $(Z\setminus Z^G)/G$. We denote by $j: (Z\setminus Z^G)/G\inj Z/G$ the inclusion, and by $\widetilde{\Z}$ the sheaf on $(Z\setminus Z^G)/G$ induced by the $G$-equivariant sheaf $\Z(1)$ on $Z\setminus Z^G$.
Since $G$ acts antiholomorphically on $Z$, it reverses its orientation, so that
$j_!\widetilde{\Z}$ is the orientation sheaf of $Z/G$ in the sense of \cite[V, Definition 9.1]{Bredon}.

 In view of the  long exact sequence of relative equivariant cohomology
$$H^2_G(Z,i_!\Z(1))\to H^2_G(Z,\Z(1))\to H^2_G(Z^G,\Z(1))\to H^3_G(Z,i_!\Z(1))\rlap{,}$$
it suffices to show that $H^q_G(Z,i_!\Z(1))=0$ for $q\geq 2$.
 By the first spectral sequence of equivariant cohomology \cite[Th\'eor\`eme 5.2.1]{Tohoku}, $H^q_G(Z, i_!\Z(1))\simeq H^q(Z/G, j_!\widetilde{\Z})$.
By Poincar\'e duality,
$H^q(Z/G, j_!\widetilde{\Z})\simeq H_{2-q}^{\BM}(Z/G,\Z)$, where $H_*^{\BM}$ denotes Borel--Moore homology (apply \cite[V, Theorem 9.3]{Bredon} to $X=Z/G$, $\sM=\Z$, and with $\Phi$ the family of closed subsets of $X$).
%Voir [Bredon 16.23, 16.27] pour une justification que les variétés à bord sont des n-whm_L avec faisceau d'orientation attendu.
% To identify the group in Theorem 9.3 with usual Borel--Moore homology, use the top of p.289 and Corollary 12.21.
This group obviously vanishes if $q\geq 3$. It also vanishes if $q=2$ as $Z/G$ has no compact component since $Z$ is Stein.
\end{proof}

\subsection{Meromorphic functions and ramified coverings}
\label{meropar}

If $Z$ is a complex manifold, we let $\sM(Z)$ be the ring of meromorphic functions on $Z$. It is the product of the fields of meromorphic functions of the connected components of $Z$.
%Pour le faisceau des fonctions méromorphes en général, voir [Grauert-Remmert, Coherent analytic sheaves, §6.3]

\begin{prop}
\label{ramifiedcover}
Let $Z$ be a complex manifold of pure dimension $1$ with finitely many connected components. Associating with $\pi:Z'\to Z$ the $\sM(Z)$-algebra $\sM(Z')$ induces an equivalence between the categories of finite morphisms $\pi:Z'\to Z$ of complex manifolds of pure dimension $1$ and of finite \'etale $\sM(Z)$-algebras.
\end{prop}

\begin{proof}
When $Z$ is connected, and if one restricts to the subcategories of finite morphisms $\pi:Z'\to Z$ with $Z'$ non-empty and connected, and of finite field extensions of $\sM(Z)$, this is \cite[Chapter~1, \textsection4.14, Corollary~4]{ShokuRiemann}.
%Earlier but not really better reference is [Röhrl, Unbounded coverings..., §4]
%Also, [Forster, lectures on Riemann surfaces]
The general case follows at once.
\end{proof}

 If $Z$ is a $G$-equivariant complex manifold, let $\sM(Z)^G$ be the ring of $G$-equivariant meromorphic functions on $Z$. If $Z$ is connected, one has $\sM(Z)^G(\sqrt{-1})=\sM(Z)$. If $Z$ has two connected components $Z'$ and $\sigma(Z')$, one has $\sM(Z)^G\simeq\sM(Z')$. In general, $\sM(Z)^G$ is the product of the fields of $G$-equivariant meromorphic functions on the $G$-orbits of connected components of $Z$.

\begin{prop}
\label{Gramifiedcover}
Let $Z$ be a $G$-equivariant complex manifold of pure dimension~$1$ with finitely many connected components. Associating with $\pi:Z'\to Z$ the $\sM(Z)^G$\nobreakdash-algebra $\sM(Z')^G$ induces an equivalence between the categories of finite morphisms $\pi:Z'\to Z$ of $G$-equivariant complex manifolds of pure dimension~$1$ and of finite \'etale $\sM(Z)^G$-algebras.
\end{prop}

\begin{proof}
This follows from Proposition \ref{ramifiedcover} and from the description of $G$-equivariant complex analytic spaces $Z$ given in \S\ref{Ranal}, as complex analytic spaces $Z$ endowed with the datum of an isomorphism $\alpha: Z^\sigma\to Z$ such that $\alpha\circ\alpha^{\sigma}=\Id_Z$.
\end{proof}

\subsection{Meromorphic functions and cohomological dimension}
\label{merodim}

The next proposition is attributed to Artin by Guralnick \cite{Guralnick}.
% Serre's letter ?

\begin{prop}
\label{cohodim1}
Let $Z$ be a connected complex manifold of dimension $1$. Then the field $\sM(Z)$ has cohomological dimension $1$.
\end{prop}

\begin{proof}
If $Z$ is compact, this is Tsen's theorem.
If $Z$ is not compact, let $L$ be a finite extension of $\sM(Z)$. By Proposition \ref{ramifiedcover},
it is the field of meromorphic functions of some connected complex manifold of dimension $1$. As a consequence,
the Brauer group of $L$ vanishes by \cite[Proposition 3.7]{Guralnick}.
The proposition then follows from \cite[II 3.1, Proposition 5]{CohoGalois}.
\end{proof}
%[Estes-Guralnick, A stable range...] soulève la question C_1 pour ces corps. Et une allusion à un argument alternatif de Serre pour la dim coho.

Recalling that a field $k$ is said to have virtual cohomological dimension $1$ if $k(\sqrt{-1})$ has cohomological dimension $1$, we deduce from Proposition \ref{cohodim1}:

\begin{cor}
\label{cohodimvirt1}
Let $Z$ be a $G$-equivariant complex manifold of dimension $1$ such that $Z/G$ is connected. Then the field $\sM(Z)^G$ has virtual cohomological dimension~$1$.
\end{cor}

Corollary \ref{cohodimvirt1} can be refined if $Z^G=\varnothing$.

\begin{prop}
\label{cohodim1sanspoint}
Let $Z$ be a $G$-equivariant complex manifold of dimension $1$ such that $Z/G$ is connected. If $Z^G=\varnothing$,  the field $\sM(Z)^G$ has cohomological dimension $1$.
\end{prop}

\begin{proof}
We claim that $-1$ is a sum of two squares in $\sM(Z)^G$.
It follows that $\sM(Z)^G$ cannot be ordered, so that its absolute Galois group contains no element of finite order by the Artin--Schreier theorem \cite{ArtinSchreier}. The main theorem of \cite{Serreprofini} and Corollary \ref{cohodimvirt1} then imply that $\sM(Z)^G$ has cohomological dimension $1$.

It remains to prove the claim. If $Z$ is compact, $\sM(Z)^G$ is the function field of a smooth projective connected curve over $\R$ with no $\R$-point by GAGA (see \S\ref{Ranal}), and the claim is due to Witt \cite[Satz 22]{Witt}.
 We assume from now on that $Z$ is not compact, hence Stein \cite[p.~134]{Stein}.

Since $H^1(Z,\sO_Z)=0$ because $Z$ is Stein, and since $H^2(Z,\Z)\simeq H_{0}^{\BM}(Z,\Z)=0$ by Poincar\'e duality \cite[V, Theorem 9.3]{Bredon}
and because $Z$ has no compact component, the exponential exact sequence (\ref{expex}) yields $H^1(Z,\sO_Z^*)=0$. Since $H^2_G(Z,\sO_Z)=0$ because $Z$ is Stein (see \S\ref{parStein}), and since $H^3_G(Z,\Z(1))=H^3(Z/G,\widetilde{\Z})=0$ (the first equality stems from the first spectral sequence of equivariant cohomology \cite[Th\'eor\`eme 5.2.1]{Tohoku} and the second from the fact that $Z/G$ is a surface), the exponential exact sequence (\ref{expex})
shows that $H^2_G(Z,\sO_Z^*)=0$.

The second spectral sequence of equivariant cohomology \cite[Th\'eor\`eme 5.2.1]{Tohoku}  $E_2^{p,q}=H^p(G,H^q(Z,\sO_Z^*))\Rightarrow H^{p+q}_G(Z,\sO_Z^*)$ now shows that $H^2(G,\sO(Z)^*)=0$. By \cite[VIII \S 4]{Corpslocaux}, $H^2(G,\sO(Z)^*)=(\sO(Z)^*)^G/\{f\cdot\sigma(f)\ , \ f\in\sO(Z)^*\}$. We deduce from this vanishing that there exists $f\in \sO(Z)^*$ such that $-1=f\cdot\sigma(f)$. It follows that $-1=\big(\frac{f+\sigma(f)}{2}\big)^2+\big(\frac{f-\sigma(f)}{2\sqrt{-1}}\big)^2$ is a sum of two squares in $\sM(Z)^G$, as desired.
\end{proof}

 Let $Z$ be a $G$-equivariant complex manifold of dimension $1$ such that $Z/G$ is connected. If $x\in Z^G$ and if $t\in\sM(Z)^G$ is a uniformizer at $x$, expanding in power series at $x$ yields an inclusion $\sM(Z)^G\subset \R((t))$.
Restricting to $\sM(Z)^G$ the unique ordering $<$  of the field $\R((t))$ for which $t>0$ gives rise to an ordering $\prec_{x,t}$ of the field $\sM(Z)^G$.
It is easily verified, using the fact that $Z$ is either projective or Stein, that $\prec_{x,t}$ and $\prec_{x',t'}$ coincide if and only if $x=x'$ and $(t/t')(x)\in\R_{>0}$. We show that if $Z^G$ is compact, there are no other orderings of $\sM(Z)^G$.

\begin{prop}
\label{spectrereel}
Let $Z$ be a $G$-equivariant complex manifold of dimension $1$ such that $Z/G$ is connected and $Z^G$ is compact. Then all the orderings of $\sM(Z)^G$ are of the form $\prec_{x,t}$ described above.
\end{prop}

\begin{proof}
Fix an ordering $\prec$ of $\sM(Z)^G$. With $\prec$ is associated  a valuation ring $$A:=\{f\in\sM(Z)^G\mid -r\prec f\prec r\textrm{ for some }r\in\R\}$$ with maximal ideal $\mathfrak{m}:=\{f\in\sM(Z)^G\mid -r\prec f\prec r\textrm{ for all }r\in\R_{>0}\}$ and residue field isomorphic to $\R$ \cite[Theorems 2 and 3]{langrealplaces}.

Assume for contradiction that for all $x\in Z^G$, there exists $f_x\in\mathfrak{m}$ such that $f_x$ does not vanish at $x$. Since $Z^G$ is compact, there exist a finite subset $\Sigma\subset Z^G$ and $\varepsilon\in\R_{>0}$ such that $f:=\sum_{x\in\Sigma}f_x^2\in\mathfrak{m}$ does not take any finite value that is $\leq\varepsilon$ on~$Z^G$.
The field $\sM(Z)^G(\sqrt{\varepsilon-f})$ is the field of $G$-equivariant meromorphic functions of a $G$-equivariant complex manifold $Y$ of dimension $1$, which is a ramified covering of~$Z$ (see Proposition \ref{ramifiedcover}). Since $f>\varepsilon$ on $Z^G$, we see that $Y^G$ lies above the poles of $\varepsilon-f$, hence is discrete. But $Y^G$ is a one-dimensional differentiable manifold as $G$ acts antiholomorphically on $Y$. It follows that $Y^G=\varnothing$.
Proposition \ref{cohodim1sanspoint} implies that $\sM(Z)^G(\sqrt{\varepsilon-f})$ cannot be ordered. Since $\varepsilon-f\succ 0$, this contradicts \cite[VIII Basic Lemma 1.4]{Lam}.

We have shown the existence of a point $x\in Z^G$ such that all $f\in\mathfrak{m}$ vanish at $x$. All $g\in A$ can be written in a unique way as $g=r+f$ with $r\in \R$ and $f\in \mathfrak{m}$, hence have no poles at $x$. Since for all $g\in (\sM(Z)^G)^*$, one of $g$ and $1/g$ must belong to $A$, and since $\R\subset A$, we deduce that $A\subset \sM(Z)^G$ is the set of functions with no poles at $x$. It follows that $\mathfrak{m}=\{f\in\sM(Z)^G\mid f(x)=0\}$.

 Let $t\in\sM(Z)^G$ be a uniformizer at $x$. After replacing $t$ with $-t$, assume that $t\succ 0$. For $f\in (\sM(Z)^G)^*$, there are unique $n\in\Z$, $r\in\R^*$ and $g\in\mathfrak{m}$ such that $f=t^n(r+g)$, and $f\succ 0$ if and only if $r>0$.  It follows that $\prec$ and $\prec_{x,t}$ coincide.
\end{proof}

\subsection{Sections of submersions}
Let $Z$ be a $G$-equivariant complex manifold.  The total space $E$ of the holomorphic vector bundle $E\to Z$ associated with a $G$\nobreakdash-equivariant locally free coherent sheaf $\sE$ on $Z$ has a natural structure of $G$\nobreakdash-equivariant complex manifold.

The first part of the following proposition is a $G$-equivariant variant of a particular case of \cite[Proposition 3.2]{Forstrat}. We explain how to make the proof work $G$\nobreakdash-equivariantly.

\begin{prop}
\label{tubularR}
Let $f: Z\to Y$ be a $G$-equivariant holomorphic map of $G$\nobreakdash-equivariant complex
manifolds, and let $u: Y\to Z$ be a $G$-equivariant holomorphic section of $f$. Suppose that $Y$ is Stein.
\begin{enumerate}[(i)]
\item There exist a $G$-stable open neighbourhood $U$ of $u(Y)$ in $Z$,  a $G$-stable open neighbourhood $U'$ of the zero section $u(Y)$ in  $N_{u(Y)/Z}$ and a $G$-equivariant biholomorphism $U'\isoto U$ respecting the projections to $Y$ that is the identity on $u(Y)$ and induces the identity
$N_{u(Y)/Z}=N_{u(Y)/N_{u(Y)/Z}}\isoto N_{u(Y)/Z}$ between normal bundles.
\item Suppose that $Y$ has no isolated point. Let $K\subset Y$ be a $G$-stable compact subset and $S\subset Z$ be a nowhere dense analytic subset.
Choose $b_1,\dots,b_m\in K$ and $r\geq 0$.
Then there exist a $G$-stable open neighbourhood $Y'$ of $K$ in $Y$ and a sequence $u_n: Y'\to Z$ of $G$-equivariant holomorphic sections of $f$ above $Y'$ with the same $r$-jets as $u$ at the $b_i$, converging uniformly to $u$ on $K$, and such that no connected component of $u_n(Y')$ is included in $S$.
\item If $Z'$ is a $G$-equivariant complex manifold and $f':Z'\to Y$ and $g:Z\to Z'$ are $G$-equivariant holomorphic maps with $f=f'\circ g$ such that $g$ is submersive along $u(Y)$, there exist a $G$-stable open neighbourhood $W$ of $g\circ u(Y)$ in $Z'$ and a $G$-equivariant holomorphic map $w:W\to Z$ with $g\circ w=\Id_{W}$ and $w\circ g\circ u=u$.
\end{enumerate}
\end{prop}

\begin{proof}
(i)
Since $u$ is a holomorphic section of $f$, the map $f$ is submersive along~$u(Y)$, and we may assume that $f$ is submersive. We may moreover assume that $Z$ is Stein by Proposition \ref{Siu}.
The $G$\nobreakdash-stable sub-vector bundle $E\subset T_Z$ consisting of those vectors tangent to the fibers of $f$ satisfies $E|_{u(Y)}\simeq N_{u(Y)/Z}$. By Lemma~\ref{SteinR}~(i), there exists a $G$-equivariant surjection $p:  Z\times\C^N\surj E$ of holomorphic vector bundles on $Z$, induced by vector fields $V_1,\dots,V_N$ on $Z$. By Lemma \ref{SteinR} (ii), there is a $G$-equivariant morphism $q: N_{u(Y)/Z}\to u(Y)\times\C^N$ of holomorphic vector bundles on $u(Y)$ that is a section of~$p|_{u(Y)\times \C^N}$. Let $\phi^i_t$ be the holomorphic flow of~$V_i$. For $(x,t_1,\dots,t_N)\in u(Y)\times\C^N$ in an appropriate neighbourhood of $u(Y)\times\{0\}$ in $u(Y)\times\C^N$, define $F(x,t_1,\dots,t_N)=\phi^1_{t_1}\circ\dots\circ\phi^N_{t_N}(x)\in Z$. The map $x\mapsto F(q(x))$ induces the required biholomorphism between neighbourhoods of $u(Y)$  in $N_{u(Y)/Z}$ and~$Z$ by the inverse function theorem, as explained in \cite[Proof of Proposition~3.2]{Forstrat}. Its $G$-equivariance follows from our choices.

(ii)
By (i), we may assume that $Z$ is a neighbourhood of the zero section of a $G$\nobreakdash-equivariant vector bundle $E$ on $Y$ and that $u$ is the zero section. Let $K'$ be a compact $G$-stable neighbourhood of $K$ in $Y$, and let $Y'$ be the union of the connected components of the interior of $K'$ that meet $K$. The compactness of $K$ implies that~$Y'$ has finitely many connected components $Y'_1,\dots,Y'_l$.
The set $\Sigma$ of $x\in Y$ such that $S$ contains a neighbourhood of $u(x)$ in $f^{-1}(x)$ is nowhere dense in $Y$ because $S\subset Z$ is a nowhere dense analytic subset. Since $G$ acts antiholomorphically on $Y$ and $Y$ has no isolated point, the subset $Y^G\subset Y$ is nowhere dense.
It follows that we can choose $y_j\in Y'_j$ for $1\leq j\leq l$ such that $y_j\notin \Sigma\cup Y^G$ and such that $y_j$ and $\sigma(y_j)$ are distinct from the $b_i$.
Since $y_j\notin\Sigma$, there exists $z_j\in f^{-1}(y_j)\subset E_{y_j}$ such that $tz_j\notin S$ for all $0<t\ll 1$. Since $Y$ is Stein, one can find a section $\zeta\in H^0(Y,E)$ vanishing to order $r$ at the $b_i$ and such that $\zeta(y_j)=z_j$ and $\zeta(\sigma(y_j))=\sigma(z_j)$. Replacing $\zeta$ with $(\zeta+\sigma(\zeta))/2$ ensures that $\zeta\in H^0(Y,E)^G$.  Since $K'$ is compact, $\zeta/n\in H^0(Y,E)^G$ induces a section $u_n$ of $f$ over $Y'$ as soon as $n\gg 0$. The sequence $u_n$ has the required properties.

(iii)
By (i), we may assume that $Z$ is a neighbourhood of the zero section $u(Y)$ of $N_{u(Y)/Z}$. Since
$g$ is submersive along $u(Y)$, the map $g$ induces a surjection of $G$-equivariant vector bundles $p:N_{u(Y)/Z}\surj N_{g\circ u(Y)/Z'}$.
By Lemma \ref{SteinR} (ii), we can find a $G$-equivariant splitting $s:N_{g\circ u(Y)/Z'}\to N_{u(Y)/Z}$ of $p$.
The composition $g\circ (s|_{s^{-1}(Z)}):s^{-1}(Z)\to Z'$ is the identity on $g\circ u (Y)$ and a local diffeomorphism along $g\circ u (Y)$, hence induces a $G$-equivariant biholomorphism $\psi:W'\isoto W$ between some $G$-stable open neighbourhoods $W'$ and $W$ of $g\circ u (Y)$ in $s^{-1}(Z)$ and $Z'$. To conclude, define $w:=s\circ\psi^{-1}:W\to Z$.
\end{proof}

\bibliographystyle{myamsalpha}
\bibliography{tight}
\end{document}